\documentclass[a4paper]{amsart}
\usepackage{amssymb}
\usepackage[hyphens]{url} 

\usepackage[colorlinks]{hyperref}
\usepackage{mathtools}
\usepackage{bm}
\usepackage{cancel}
\usepackage{xcolor}
\usepackage{enumerate}
\usepackage{geometry,color}
\usepackage[final]{graphicx}
\usepackage{subcaption}
\usepackage{float}

\DeclareMathOperator*{\esssup}{ess\,sup}

\newcommand{\explain}[2]{\overset{\mathclap{\underset{\downarrow}{#2}}}{#1}}
\newcommand{\defeq}{{:=}}

\newcommand{\weaklyto}{{\rightharpoonup}}
\newcommand{\weakstarto}{{\overset{\ast}{\rightharpoonup}}}
\newcommand{\T}{\mathbb{T}}
\newcommand{\R}{\mathbb{R}}

\newcommand{\N}{\mathbb{N}}
\newcommand{\Z}{\mathbb{Z}}
\newcommand{\C}{\mathbb{C}}
\renewcommand{\div}{{\mathrm{div}}}
\newcommand{\curl}{{\mathrm{curl}}}
\newcommand{\Lip}{{\mathrm{Lip}}}
\newcommand{\err}{{\mathrm{err}}}
\newcommand{\sign}{{\mathrm{sign}}}

\renewcommand{\vec}{\bm}

\renewcommand{\P}{{\mathcal{P}}}
\newcommand{\D}{\mathcal{D}}
\newcommand{\M}{\mathcal{M}}

\newcommand{\uhat}{{\widehat{u}}}
\newcommand{\what}{{\widehat{\omega}}}
\newcommand{\e}{{\epsilon}}

\newcommand{\embeds}{{\hookrightarrow}}
\newcommand{\embedsc}{\overset{c}{\hookrightarrow}}

\newcommand{\dxdt}{{\, dx \, dt}}
\newcommand{\dx}{{\, dx}}

\newcommand{\revision}{}

\newtheoremstyle{mythm} 
    {.3cm}                    
    {.3cm}                    
    {\itshape}                   
    {}                           
    {\bfseries}                   
    {.}                          
    {.5em}                       
    {}  

\theoremstyle{mythm}
 \newtheorem{thm}{Theorem}[section]
 \newtheorem{prop}[thm]{Proposition}
 \newtheorem{lem}[thm]{Lemma}
 \newtheorem{cor}[thm]{Corollary}
\theoremstyle{mythm}
 \newtheorem{defn}[thm]{Definition}
\theoremstyle{remark}
 \newtheorem{remark}[thm]{Remark}
 \numberwithin{equation}{section}

\renewcommand{\le}{\leqslant}\renewcommand{\leq}{\leqslant}
\renewcommand{\ge}{\geqslant}

\title[Spectral Viscosity]{On the convergence of the spectral viscosity method for the incompressible Euler equations with rough initial data}

\author[Lanthaler]{\bfseries Samuel~Lanthaler} 

\address{ 
Seminar for Applied Mathematics, Department of Mathematics \\ 
ETH Zurich   \\ 
Zurich\\
Switzerland}
\email{samuel.lanthaler@sam.math.ethz.ch}

\author[Mishra]{Siddhartha~Mishra}
\address{ }

\begin{document}


%
%
%


  {\begin{flushleft}\baselineskip9pt\scriptsize
MANUSCRIPT
\end{flushleft}}
\vspace{18mm} \setcounter{page}{1} \thispagestyle{empty}

\begin{abstract}
We propose a spectral viscosity method to approximate the two-dimensional Euler equations with rough initial data and prove that the method converges to a weak solution for a large class of initial data, including when the
initial vorticity is in the so-called Delort class i.e. it is a sum of a signed measure and an integrable function. This provides the first convergence proof for a numerical method approximating the Euler equations with such rough
initial data and closes the gap between the available existence theory and rigorous convergence results for numerical methods. We also present numerical experiments, including computations of vortex sheets and confined eddies, to illustrate the proposed method.
\end{abstract}

  \maketitle

\section{Introduction}
Flow of incompressible fluids at (very) high Reynolds numbers is often approximated by the \emph{incompressible Euler} equations that model the motion of an ideal (incompressible and inviscid) fluid, \cite{majda2001} and references therein. The incompressible Euler equations are nonlinear partial differential equations of the form,
\begin{gather} \label{eq:Eulerfull}
\left\{
\begin{aligned}
\partial_t \vec{u} 
+\vec{u}\cdot \nabla \vec{u} 
+ \nabla p
&=
0, 
\\
\div(\vec{u}) 
&= 
0, 
\\
\vec{u}|_{t=0} 
&=
\vec{u}_0. 
\end{aligned}
\right.
\end{gather}
Here,  the velocity field is denoted by $\vec{u} \in \R^d$ (for $d=2,3$), and the pressure is denoted by $p \in \R_+$. The pressure acts as a Lagrange multiplier to enforce the divergence-free constraint. The equations need to be supplemented with
suitable boundary conditions. For simplicity, we will only consider the case of periodic boundary conditions in this paper. 
\subsection{Mathematical results.}
Although short-time (or small data) well-posedness results are classical \cite{Lich1}, the questions of well-posedness, i.e. existence, uniqueness, stability and regularity of \emph{global} solutions of the three-dimensional Euler equations are largely open. Notable exceptions are provided by the striking results of \cite{Sh1,Sch1,DL1,DL2}, where it is established that weak solutions, even H\"older continuous ones (with H\"older exponent $< \frac{1}{3}$), are not necessarily unique. 

On the other hand, the analysis of the Euler equations \eqref{eq:Eulerfull} in two space dimensions is significantly more mature.  This is mainly due to the fact that, in two dimensions, the vorticity $\omega = \curl(\vec{u})$ of a solution $\vec{u}$ to the PDE \eqref{eq:Eulerfull} satisfies a \emph{transport} equation
\begin{gather} \label{eq:Eulervort}
\partial_t \omega 
+ \vec{u} \cdot \nabla \omega 
= 0, 
\end{gather}
providing a priori control on various norms of $\omega$, such as $L^p$-norms \cite{majda2001}.

Global existence and uniqueness results for the two-dimensional incompressible Euler equations with smooth initial data are classical \cite{majda2001,Diperna1987}. For non-smooth initial conditions, the work by Yudovich \cite{Yudovich1963} has established existence and uniqueness for bounded initial vorticity, i.e. $\omega_0 \in L^\infty$. The uniqueness result of \cite{Yudovich1963} has later been extended to vorticities belonging to slightly more general spaces \cite{Vishik1998,Vishik1999,Yudovich1995}.  An existence result for vorticity $\omega_0 \in L^p$ , $1<p<\infty$ has been obtained by Diperna and Majda \cite{Diperna1987}. It is shown in \cite{Diperna1987} that the sequence obtained by solving the Euler equations for mollified initial data is strongly compact in $L^2$. The existence of a weak solution for $\omega_0\in L^p$ is then established by passing to the limit. Further extensions of the result of Diperna and Majda can be obtained by compensated compactness methods for initial vorticity $\omega_0$ belonging to e.g. Orlicz spaces such as $\omega_0 \in L \log(L)^\alpha$, $\alpha\ge 1/2$, which are compactly embedded in $H^{-1}$ \cite{Morgulis1992,Chae1993,Chae1994,Lopes2000}. These methods break down for $\omega_0 \in L^1$.

In his celebrated work \cite{Delort1991}, Delort has shown the existence of solutions to the Euler equations with initial vorticity $\omega_0 = \omega_0' + \omega_0''$, where $\omega_0'$ is a finite, non-negative Radon measure belonging also to $H^{-1}$, and $\omega_0'' \in L^p$, for some $p>1$. These initial data correspond to the interesting case of \emph{vortex sheets} i.e. vorticity concentrated on curves in the two-dimensional spatial domain \cite{majda2001}. In \cite{Delort1991}, it is remarked that the proof can be extended to allow for $\omega_0'' \in L^1$. A detailed proof of this claim has subsequently been provided by Vecchi and Wu \cite{Vecchi1993}. The results of Delort \cite{Delort1991}, and Vecchi and Wu \cite{Vecchi1993}, remain the most general existence results for the  incompressible Euler equations in two dimensions. The question of existence of solutions beyond this \emph{Delort class}, for instance, when $\omega_0$ is an arbitrary signed bounded measure, remains open. The uniqueness question also remains open, even for vorticities $\omega_0 \in L^p$, $p<\infty$.

\subsection{Numerical schemes.} It is not possible to represent solutions of the incompressible Euler equations in terms of analytical solution formulas, even in two space dimensions. Hence, numerical approximation of \eqref{eq:Eulerfull} is a necessary and key ingredient in the study of the incompressible Euler equations. A wide variety of numerical methods have been developed to robustly approximate the incompressible Euler equations. These include spectral methods \cite{Ors1}, finite difference-projection methods \cite{Cho1,BCG1},  finite element methods \cite{GShu1} and vortex methods \cite{Kras1,Kras2,majda2001}.

Although finite difference and finite element methods are very useful when discretizing the Euler equations in domains with complex geometry, spectral methods, based on projecting \eqref{eq:Eulerfull} into a finite number of Fourier modes are the method of choice for approximating \eqref{eq:Eulerfull} with periodic boundary conditions. These methods are very efficient to implement (aided by the Fast Fourier transform (FFT)), fast to run and have \emph{spectral}, i.e. superpolynomial convergence rates for smooth solutions of \eqref{eq:Eulerfull} \cite{Ors1}. Consequently, spectral methods are widely used in the simulation of homogeneous and isotropic turbulence \cite{Ghoshal,Karamanos2000a}. 

Rigorous convergence results for numerical approximations of the incompressible Euler equations are mostly available when the underlying (continuous) solution is sufficiently smooth, see \cite{Bardos2015} for spectral methods, \cite{LSSID1} for finite-difference projection methods, \cite{GShu1} for discontinuous Galerkin methods and \cite{majda2001} for vortex methods. This represents a significant gap between the existence results for the underlying weak solutions (at least for the two-dimensional case) and convergence results for numerical methods. In particular, it is essential to design (and prove) convergent numerical methods for the two-dimensional Euler equations with rough initial data such as with initial vorticity in $L^p$, for $1 \leq p < \infty$, or even for initial vorticity in the afore-mentioned \emph{Delort class}. 

In this context, we survey a few available results for convergence of numerical methods approximating \eqref{eq:Eulerfull} with rough initial data. A notable result in this regard is the convergence of a \emph{central finite difference scheme} (\cite{Levy1997}) for the vorticity formulation \eqref{eq:Eulervort} of the two-dimensional Euler equations \cite{Levy1997}. This scheme was shown to possess a discrete maximum principle for the vorticity. Hence, one can prove that it converges to a weak solution of \eqref{eq:Eulervort}, as long as the initial vorticity $\omega_0 \in L^p$ for $1 < p \leq \infty$ \cite{Lopes2000}. However, it is unclear if the convergence analysis for this scheme can be extended to the case where the initial data $\omega_0 \in L^1$, let alone in the Delort class. Similarly for spectral methods and for finite difference-projection methods, the only available results for \eqref{eq:Eulerfull}, are of convergence to the significantly weaker solution framework of dissipative measure-valued solutions in \cite{LM2015} and in \cite{LeonardiPhD}, respectively. 

\revision{When $\omega_0 \in \mathcal{M} \cap H^{-1}$ is a bounded measure, the best available convergence results to date have been achieved by Liu and Xin for the vortex blob method in \cite{LiuXin1995} and by Schochet for both the vortex point and blob methods in \cite{Schochet1996} (see also the related work by Liu and Xin \cite{LiuXin2001}). In \cite{LiuXin1995,Schochet1996,LiuXin2001}, it is shown that for initial data with vorticity $\omega_0 \in H^{-1}$ a finite, non-negative Radon measure in $\mathcal{M}_{+}$, the vortex methods will converge weakly to a weak solution of the incompressible Euler equations with $\omega \in \mathcal{M}_{+} \cap H^{-1}$. The assumption on the definite sign (either positive or negative in the whole domain) of the initial vorticity appears to be an essential ingredient in these convergence results \cite{LiuXin1995,Schochet1996,LiuXin2001}: If $\omega_0$ has a definite sign, then the conserved Hamiltonian of these vortex methods can be leveraged to provide a priori control the concentration of the discretized vorticity. When the initial vorticity $\omega_0$ is not necessarily of definite sign, then the Hamiltonian no longer provides control on vorticity concentration and the available convergence results are somewhat weaker in this case. Without any sign restriction, convergence of the vortex point/blob methods has been shown by Schochet \cite{Schochet1996} for initial data with vorticity $\omega_0 \in L(\log L)$.}

\revision{The fundamental difficulty that prevents the convergence results of vortex methods to be extended to initial data of the form $\omega_0 = \omega_0' + \omega_0''$, $\omega_0' \in \mathcal{M}_+ \cap H^{-1}$, $\omega_0'' \in L^1$ considered by Delort \cite{Delort1991,Vecchi1993}, apparently lies in the fact that at the continuous level, concentration of $\omega_0''\in L^1$ is prevented by the incompressibility of the advecting flow. However, in the case of vortex methods, incompressibility of the advecting flow is not known to be sufficient to prevent concentration of the discretized vortices. In the definite sign case ($\omega_0'' = 0$), it turns out that the discrete energy conservation can be used to circumvent this issue \cite{Majda1993,LiuXin1995,Schochet1996,LiuXin2001}. Without any sign restriction, but assuming that $\omega_0 \in L (\log L)$, the conservation of phase-space volume (Liouville's theorem) can be used to show that no concentration occurs for suitable vortex approximations to the initial data $\omega_0$ \cite{Schochet1996}. Therefore a considerable gap remains between the available convergence results for vortex methods and the existence result of Delort.}

\subsection{Aims and scope of this paper.}
Our main aim in this paper is to design a suitable numerical method to approximate the two-dimensional version of the Euler equations \eqref{eq:Eulerfull} and to prove convergence to the underlying weak solutions, when the initial data for \eqref{eq:Eulerfull} is rough, i.e. for instance the initial vorticity $\omega_0 \in L^p$, for $1 \leq p < \infty$  or when the initial vorticity is in the \emph{Delort class}. We focus on spectral methods in this paper. However, it is well known that spectral methods may not suffice to approximate weak solutions of the incompressible Euler equations in a stable manner and need to be modified. 

This situation is somewhat analogous for the much simpler case of the Burgers' equation or in general, for scalar conservation laws. Given the formation of singularities such as shock waves for these problems, spectral methods do not contain enough numerical diffusion to damp down oscillations that arise on account of Gibbs' phenomena \cite{Tadmor1989}. Consequently, one modifies spectral methods by adding numerical diffusion to a sufficient number of (high) Fourier modes to stabilize solutions while still maintaining (superpolynomial) spectral accuracy for smooth problems. Such \emph{spectral viscosity} (SV) methods were first proposed by Tadmor  in \cite{Tadmor1989} and  have been shown to converge to entropy solutions of the underlying scalar conservation laws in a series of papers  \cite{Tadmor1989,Tadmor1990,Tadmor1993,MadayTadmor1989,Schochet1990}. Spectral viscosity methods have been employed to robustly approximate turbulent flow in  \cite{Karamanos2000a,Gunzburger2010} and references therein. 

In this paper, we will modify the spectral viscosity method of Tadmor to approximate the Euler equations with rough initial data and prove convergence to underlying weak solutions when the initial vorticity belongs to $L^p$ for $1\le p\le \infty$, or more generally if it belongs the Delort class, i.e.
\begin{equation}
\label{eq:dcls}
\omega_0 \in \left(\M_{+} + L^1\right) \cap H^{-1}.
\end{equation}

Our main ingredients in the present work will be the equivalence between the \emph{primitive and vorticity} formulations for the spectral viscosity method and the determination of sufficient conditions on the free parameters of the spectral viscosity method, namely the strength of the numerical diffusion and the number of modes to which the diffusion is applied, from a rigorous stability analysis of the scheme. Moreover, we also introduce a novel modification of the standard Fourier discretization of the initial data that allows us to handle its roughness. The resulting scheme is carefully analyzed and sharp estimates are derived that allow us to apply compensated compactness arguments of \cite{Lopes2000} for $\omega_0 \in L^p$, with $1< p\le \infty$ and of \cite{Delort1991,Vecchi1993} when the initial vorticity is in the \emph{Delort class}, in order to show convergence to weak solutions. 

Thus, we close the aforementioned gap between existence results and rigorous convergence results for numerical approximations of the two dimensional Euler equations and provide the \emph{first rigorous proof of convergence for any numerical method to the weak solutions of the Euler equations with rough initial data \revision{in the Delort class \eqref{eq:dcls}}}.  

This paper is organized as follows: In section \ref{sec:spectralvisc}, we describe the spectral vanishing viscosity method for the incompressible Euler equaitons \eqref{eq:Eulerfull}, and point out the equivalence of the spectral approximation in primitive variable form and the formulation in terms of the vorticity. In section \ref{sec:compensatedcompactness}, we review the notion of approximate solution sequences and prove that the vanishing viscosity method provides an approximate solution sequence. We also recall some notions of compensated compactness for uniformly bounded vorticities. In section \ref{sec:spectraldecay} we establish simple a priori $L^2$ bounds. These $L^2$ estimates will be used to prove a spectral decay result, allowing us to control the discretization error. Refined estimates providing $L^p$-control will be based on the spectral decay result established in section \ref{sec:spectraldecay}. These spectral decay estimates must be complemented by short time estimates, providing control of the $L^p$-norm of the vorticity over a short initial interval of time. During this initial time interval, viscosity will dampen out higher-order modes and provide the required spectral decay. The short-time estimates are the subject of section \ref{sec:shorttime}. In section \ref{sec:convergence}, we prove convergence of the spectral viscosity method in the case when the initial vorticity, $\omega_0 \in L^p$, with $1< p\le \infty$ and when it is in the \emph{Delort class} \eqref{eq:dcls}. Numerical experiments illustrating the theory are presented in section \ref{sec:numerical}. Appendix \ref{sec:Bernstein} collects some known results from the literature, which are needed throughout this work.

\section{Spectral viscosity method} \label{sec:spectralvisc}

The incompressible Euler equations \eqref{eq:Eulerfull} are to be understood in the weak (distributional) sense. The notion of a weak solution to the incompressible Euler equations is made precise in the following definition \cite{majda2001}:
\begin{defn} \label{def:Eulerweak}
A vector field $\vec{u}\in L^\infty([0,T];L^2(\T^2;\R^2))$ is a weak solution of the incompressible Euler equations with initial data $\vec{u}_0 \in L^2(\T^2;\R^2)$, if
\begin{gather}\label{eq:Eulerweak}
\int_0^T\int_{\T^2}
\vec{u}\cdot \partial_t \vec{\phi} 
+ (\vec{u}\otimes \vec{u}):\nabla \vec{\phi} \dxdt
= -\int_{\T^2} \vec{u}_0\cdot \vec{\phi}(x,0) \dx,
\end{gather}
for all test vector fields, $\vec{\phi}\in C^\infty(\T^2\times [0,T];\R^2)$, $\div(\vec{\phi})=0$, and 
\begin{gather}\label{eq:incomprweak}
\int_{\T^2} \vec{u}\cdot \nabla \psi \dx = 0,
\end{gather}
for all test functions $\psi\in C^\infty(\T^2)$.
\end{defn}

Given the fact that we consider the incompressible Euler equations in a periodic domain and in order to ensure equivalence between the primitive (velocity-pressure) and vorticity formulations of the underlying equations, henceforth, we assume that 
\begin{equation}
\label{eq:init0}
\int_{\T^2} \vec{u}_0 \, dx = 0.
\end{equation}

In the following, we will consider the spectral vanishing viscosity (SV) scheme for the incompressible Euler equations: We write $\vec{u}_N(x,t) = \sum_{|\vec{k}|_\infty\le N} \widehat{\vec{u}}_{\vec{k}}(t) e^{i\vec{k}\cdot\vec{x}}$, where $|\vec{k}|_\infty\defeq \max(|k_1|,|k_2|)$, and consider the following approximation of the incompressible Euler equations
\begin{gather} \label{eq:Euler}
\left\{
\begin{aligned}
\partial_t \vec{u}_N 
+\P_N(\vec{u}_N\cdot \nabla \vec{u}_N) 
+ \nabla p_N 
&=
\epsilon_N \Delta (Q_N \ast \vec{u}_N), 
\\
\div(\vec{u}_N) 
&= 
0, 
\\
\vec{u}_N|_{t=0} 
&=
K_{a_N} \ast \vec{u}_0,
\end{aligned}
\right.
\end{gather}
\revision{with periodic boundary conditions}. Here $\P_N$ is the spatial Fourier projection operator, mapping an arbitrary function $f(x,t)$ onto the first $N$ Fourier modes: $\P_N f(x,t) = \sum_{|\vec{k}|_\infty\le N} \widehat{f}_{\vec{k}}(t) e^{i\vec{k}\cdot\vec{x}}$.
$Q_N$ is a Fourier multiplier of the form 
\begin{gather}
Q_N(x) = \sum_{m_N < |\vec{k}| \le N} \widehat{Q}_{\vec{k}} e^{i\vec{k}\cdot\vec{x}}, 
\end{gather}
and we assume
\begin{gather}\label{eq:Qk}
0 \le \widehat{Q}_{\vec{k}} \le 1, \quad 
\widehat{Q}_{\vec{k}} = 
\begin{cases}
0, &|\vec{k}| \le m_N, \\
1, &|\vec{k}| > 2 m_N.
\end{cases}
\end{gather}
The parameters $m_N$ and $\e_N$ already appear in the original formulation of the SV method as applied to scalar conservation laws \cite{Tadmor1989}. Their dependence on $N$ will be specified later. The idea behind the SV method is that dissipation is only applied on the upper part of the spectrum, i.e. for $|\vec{k}| > m_N$, thus preserving the formal spectral accuracy of the spectral method, while at the same time enabling us to enforce a sufficient amount of energy dissipation on the small scale Fourier modes needed to stabilize the method and ensure its convergence to a weak solution.

\begin{remark}\label{rem:mN}
  In equation \eqref{eq:Qk}, we have assumed that the coefficients $\widehat{Q}_{\vec{k}}$ change only in the interval $|\vec{k}| \in [m_N,2m_N]$. This assumption could have been replaced by taking $[m_N,cm_N]$, for any constant $c>1$, without changing the results of this paper. We have chosen $c=2$ here for simplicity, and in order not to introduce further parameters into the numerical scheme. In practice, a different choice may be more suitable.
\end{remark}

As a slight extension to \cite{Tadmor1989}, we have introduced an additional Fourier kernel $K_{a_N}$. This kernel gives an another degree of freedom in our numerical method, and will be necessary to obtain suitably approximated initial data, providing further control on the numerical solution. The Fourier kernel $K_{a_N}$ is a trigonometric polynomial of the form 
\[
K_{a_N}(x) = \sum_{|\vec{k}|\le a_N} \widehat{K}_{\vec{k}} e^{i\vec{k}\cdot\vec{x}}, \qquad |\widehat{K}_{\vec{k}}| \le 1.
\]
The exact form of the kernel $K_{a_N}$ and the choice of parameters $a_N$ will be specified later. However, we shall assume that $K_{a_N}$ satisfies a bound of the form 
\begin{equation} \label{eq:Klog}
\Vert K_{a_N} \Vert_{L^1} \le C \log(N)^2, \quad \text{for all }N\in \N.
\end{equation}
The above discretization of the initial conditions will be necessary in our convergence proofs for initial vorticity in spaces $\omega_0 \in L^p$ with $p<2$, or indeed for initial vorticity, which is a a vortex sheet as considered by Delort.

For the numerical implementation, the system \eqref{eq:Euler} can conveniently be expressed in terms of the Fourier coefficients:
\begin{gather} \label{eq:EulerFourier}
\left\{
\begin{aligned}
\partial_t \widehat{\vec{u}}_{\vec{k}} 
+
\left(\vec{1}-\frac{\vec{k}\otimes \vec{k}}{|\vec{k}|^2}\right)
\sum_{|\vec{\ell}|,|\vec{k}-\vec{\ell}|\le N}  i(\vec{\ell} \cdot \widehat{\vec{u}}_{\vec{k}-\vec{\ell}}) \widehat{\vec{u}}_{\vec{\ell}}
=
-\epsilon_N |\vec{k}|^2 \widehat{Q}_{\vec{k}} \widehat{\vec{u}}_{\vec{k}}, 
\\
\widehat{\vec{u}}_{\vec{k}}|_{t=0} 
=
\widehat{K}_{\vec{k}} \,\widehat{[\vec{u}_0]}_{\vec{k}},
 \qquad (\text{for all } 0<|\vec{k}|\le N).
\end{aligned}
\right.
\end{gather}
Note that we suppress the time dependence $\widehat{\vec{u}}_{\vec{k}} = \widehat{\vec{u}}_{\vec{k}}(t)$ for notational convenience. From \eqref{eq:init0}, we shall assume that $\widehat{[\vec{u}_0]}|_{\vec{k}=0} = 0$, which then implies that also $\widehat{\vec{u}}|_{\vec{k}=0}  = 0$, for all later times. In addition, we shall assume that the initial data is divergence-free initially, i.e. that $\widehat{[\vec{u}_0]}_{\vec{k}}\cdot \vec{k}=0$ for all $|\vec{k}|\le N$. Again, this can be shown to imply that $\widehat{\vec{u}}_{\vec{k}} \cdot \vec{k} = 0$ also at later times, as discussed e.g. in \cite{LM2015}.

\begin{remark}
The SV scheme for the incompressible Euler equations depends on the three parameter sequences $\e_N, m_N, a_N$. To fix ideas, we note that we will later on choose $\e_N \to 0$, $a_N \sim m_N \sim N^\theta \to \infty$ for some $\theta\le \frac12$.
\end{remark}

Since the $\vec{u}_N$ are smooth, and since the Fourier projection commutes with differentiation, it turns out that we can equivalently write the system \eqref{eq:Euler} in its vorticity form
\begin{gather} \label{eq:vorticity}
\left\{
\begin{aligned}
\partial_t \omega_N 
+ \P_N(\vec{u}_N \cdot \nabla \omega_N) 
&= \e_N \Delta (Q_N \ast \omega_N), 
\\
\curl(\vec{u}_N) 
&= \omega_N, 
\\
\omega_N|_{t=0} 
&=
\curl\left(K_{a_N} \ast \vec{u}_0\right). 
\end{aligned}
\right.
\end{gather}

We recall the following simple result, which will be of fundamental importance for the current work:
\begin{prop}[Lemma 3.10, \cite{LM2015}] \label{prop:equivalence}
The systems \eqref{eq:Euler} for $\vec{u}_N$ and \eqref{eq:vorticity} for $\omega_N$ are equivalent.
\end{prop}

\begin{remark}
Proposition \ref{prop:equivalence} allows us to focus on the vorticity formulation \eqref{eq:vorticity}. The strategy is then as follows: The vorticity formulation will be used to obtain uniform a priori control on the $L^p$-norm of the approximate vorticities $\omega_N$, for some $1\le p \le \infty$. The bounds on $\omega_N$ in turn provide additional control on the velocity $\vec{u}_N$, which can be used to prove the convergence of the non-linear terms in the \emph{primitive variable formulation} \eqref{eq:Euler}. The convergence of the non-linear terms will rely either on establishing pre-compactness of the sequence $\vec{u}_N$ in $L^2(\T^2;\R^2)$, following the original ideas of Diperna and Majda \cite{Diperna1987}, or by employing compensated compactness results established by Delort \cite{Delort1991,Vecchi1993,Schochet1995}. It is thus the interplay between the primitive and the vorticity formulation, which will allow us to obtain convergence proofs even for rough initial data.
\end{remark}

As a first step towards proving the convergence of the SV method, we make the error terms more apparent. We rewrite the system \eqref{eq:vorticity} in the following form 
\begin{gather} \label{eq:vorterr}
\partial_t \omega_N + \vec{u}_N \cdot \nabla \omega_N - \e_N \Delta \omega_N
= \underbrace{(I-\P_N)(\vec{u}_N\cdot \nabla \omega_N)}_{=: \err_1} + \underbrace{\e_N \Delta R_{m_N} \ast \omega_N}_{=: \err_2}.
\end{gather}
The left-hand side corresponds to the vorticity formulation of the Navier-Stokes equations in 2d with viscosity $\e_N$. The right hand side consists of a projection error ($\err_1$), and a "viscosity" error ($\err_2$), which is written in terms of a convolution with $R_{m_N} \equiv 1 - Q_N$. We note that $R_{m_N}(x)$ has Fourier coefficients
\begin{gather}
0\le \widehat{R}_{\vec{k}} \le 1, \quad 
\widehat{R}_{\vec{k}} = 
\begin{cases}
1, & |\vec{k}|\le m_N, \\
0, & |\vec{k}|> 2m_N.
\end{cases}
\end{gather}
Similar to \eqref{eq:Klog}, we will assume a bound of the form 
\begin{equation} \label{eq:Rlog}
\Vert R_{m_N} \Vert_{L^1} \le C \log(N)^2, \quad \text{for all }N\in \N,
\end{equation}
for the kernel $R_{m_N}$.

It will turn out that for an appropriate choice of $\e_N, m_N$, the second error term $\err_2$ is  benign, since it is a bounded operator on $L^p$ \cite{Tadmor1989}. Our main tool used to obtain bounds on the projection error $\err_1$ will be a spectral decay estimate of the Fourier coefficients in the range $N/2 \le |\vec{k}| \le N$. This will imply that the coefficients, corresponding to the high Fourier modes, decay (exponentially) fast. We will then use this spectral decay estimate, to obtain estimates providing uniform $L^p$-control of the vorticity, provided that $\omega_0\in L^p$. The case of Delort-type initial data will pose additional difficulties as compared to the case $1<p\le \infty$. This is discussed in section \ref{sec:Delort}. 

\section{A brief overview of compensated compactness}
\label{sec:compensatedcompactness}

In this section, we list some results from the literature that we will use for proving convergence of the spectral viscosity scheme when the initial vorticity $\omega_0 \in L^p(\T^2)$ for $1  < p \leq \infty$. The convergence proofs are based on the compensated compactness method for the incompressible Euler equations developed by Lopes Filho, Nussenzveig Lopes and Tadmor in \cite{Filho2000}. We first need the following definition.
\begin{defn} \label{def:approxsolseq}
Let $\{\vec{u}^\e\}$ be uniformly bounded in $L^\infty([0,T];L^2(\T^2;\R^2))$. The sequence $\{\vec{u}^\e\}$ is an approximate solution sequence for the incompressible Euler equations, if the following properties are satisfied:
\begin{enumerate}
\item The sequence $\{\vec{u}^\e\}$ is uniformly bounded in $\Lip((0,T);H^{-L}(\T^2;\R^2))$, for some $L>1$.
\item For any test vector field $\vec{\Phi} \in C^\infty([0,T)\times \T^2;\R^2)$ with $\div(\vec{\Phi})=0$, we have:
\[
\lim \limits_{\e \rightarrow 0} \int_0^T \int_{\T^2} \vec{\Phi}_t \cdot \vec{u}^\e + (\nabla \vec{\Phi}):(\vec{u}^\e\otimes \vec{u}^\e) \, dx \, dt
+ \int_{\T^2} \vec{\Phi}(x,0)\cdot \vec{u}^\e(x,0) \, dx 
= 0.
\]
\item $\div(\vec{u}^\e)=0$ in $\D'([0,T]\times \T^2)$.
\end{enumerate}
\end{defn}

It will be shown in section \ref{sec:spectraldecay}, below, that the approximations obtained by the spectral vanishing viscosity method are approximate solutions in this sense.

The authors of \cite{Filho2000} introduce the following definition (slightly adapted here to the case of a domain $\T^2$, rather than $\R^2$):

\begin{defn}[$H^{-1}$-stability, \cite{Filho2000}]
A sequence of divergence-free vector fields $\vec{u}^\e\in L^2(\T^2;\R^2)$ is called $H^{-1}$-stable if $\{\curl(\vec{u}^\e) = \omega^\e\}$ is a precompact subset of $C([0,T];H^{-1}(\T^2))$.
\end{defn}

For the current purposes, the following remark (which we formulate as a Theorem) will be sufficient.
\begin{thm}[Rmk. 2. to Thm. 1.1, \cite{Filho2000}] \label{thm:Hstab}
Let $\{\vec{u}^\e\}$ be an approximate solution sequence of the incompressible Euler equations. If $\{\vec{u}^\e\}$ is $H^{-1}$-stable then there exists a subsequence which converges strongly in $C([0,T];L^2(\T^2;\R^2))$ to a weak solution $\vec{u}$.
\end{thm}

Finally, we recall the following lemma from \cite{Filho2000}:
\begin{lem}[see e.g. \cite{Filho2000}] \label{lem:Lpcompact}
$L^{p}(\T^2)$ is compactly embedded in $H^{-1}(\T^2)$ for $p>1$.
\end{lem}

{\color{black}
Now let $\vec{u}^\e$ be an approximate solution sequence for the incompressible Euler equations, with vorticity $\omega^\e$ uniformly bounded in $L^\infty([0,T];L^p(\T^2))$, for some $p>1$. We can then apply the Aubin-Lions lemma, which we have stated as Theorem \ref{thm:AubinLions} in the appendix, applied to the family of functions $F=\{\omega^\e\}$, and the spaces $X\subset B \subset Y$, where
\[
X=L^p(\T^2), \; B = H^{-1}(\T^2) \; \text{and} \; Y=H^{-L-1}(\T^2).
\]
To check the applicability of the Aubin-Lions lemma, we note that the embedding $X\to B$ is compact by Lemma \ref{lem:Lpcompact},  $F=\{\omega^\e\}$ is uniformly bounded in $B$, by the assumed $L^2$-boundedness of $\vec{u}^\e$ (cf. Definition \ref{def:approxsolseq}), and $\{\omega^\e\}$ satisfies the equicontinuity property of Theorem \ref{thm:AubinLions}, due to the assumed Lipschitz continuity in Definition \ref{def:approxsolseq}. Applying the Aubin-Lions lemma, we can conclude that $\{\omega^\e\}$ is relatively compact in $C([0,T];H^{-1}(\T^2))$. 

In particular, it now follows from Theorem \ref{thm:Hstab}, that
}
\begin{cor} \label{cor:Lpcompconv}
If $\vec{u}^\e$ is an approximate solution sequence of the incompressible Euler equations, and if $\omega^\e$ is uniformly bounded in $L^\infty([0,T];L^p(\T^2))$ with $p>1$, then there exists a subsequence $\e' \to 0$, such that $\vec{u}^{\e'}$ converges strongly in $C([0,T];L^2(\T^2;\R^2))$ to a weak solution $\vec{u}$ of the incompressible Euler equations.
\end{cor}

\section{Spectral decay estimate}
\label{sec:spectraldecay}

Before establishing more detailed $L^p$-type estimates for the vorticity, we note that $L^2$ estimates for the approximate solutions, $\vec{u}_N$ and $\omega_N$ are readily obtained.

\begin{prop} \label{prop:L2bound}
If $\vec{u}_0\in L^2$, then the approximation sequence $\vec{u}_N$ satisfies
\[
\Vert \vec{u}_N(\cdot,t) \Vert_{L^2}
\le 
\Vert \vec{u}_N(\cdot,0) \Vert_{L^2} 
\le \Vert \vec{u}_0 \Vert_{L^2}.
\]
In particular, this implies that we have a uniform bound
\[
\Vert \omega_N(\cdot,t) \Vert_{H^{-1}} \le \Vert \vec{u}_0 \Vert_{L^2}.
\]
\end{prop}
\begin{proof}
Multiply \eqref{eq:Euler} by $\vec{u}_N$, integrate over the spatial variable, we find 
\[
\frac{d}{dt} \int_{\T^2} |\vec{u}_N|^2 \, dx 
= - \int_{\T^2} \nabla \vec{u}_N : \nabla (Q_N\ast \vec{u}_N) \, dx
\explain{=}{\text{(Plancherel)}} - \sum_{\vec{k}} \widehat{Q}_{\vec{k}} |\vec{k}|^2 |\widehat{(\vec{u}_N)}_{\vec{k}}|^2
\le 0.
\]
Integration over time yields the first inequality. The second inequality follows from
\[
\Vert \vec{u}_N(\cdot,0) \Vert_{L^2}^2
= \Vert K_{a_N}\ast \vec{u}_0 \Vert_{L^2}^2
= \sum_{\vec{k}} \widehat{K}_{\vec{k}}^2 |\widehat{(\vec{u}_0)}_{\vec{k}}|^2
\le \sum_{\vec{k}} |\widehat{(\vec{u}_0)}_{\vec{k}}|^2
= \Vert \vec{u}_0 \Vert_{L^2}^2.
\]
The non-linear terms in \eqref{eq:Euler} cancel out after multiplication with $\vec{u}_N$ in the above estimate. The upper bound for $\Vert \omega \Vert_{H^{-1}}$ is trivial.
\end{proof}

And similarly for the vorticity, we also have
\begin{prop} \label{prop:L2vortbound}
If $\omega_0\in L^2$, then the approximation sequence $\omega_N$ satisfies
\[
\Vert \omega_N(\cdot,t) \Vert_{L^2}
\le 
\Vert \omega_N(\cdot,0) \Vert_{L^2} 
\le \Vert \omega_0 \Vert_{L^2}.
\]
\end{prop}
Multiplying \eqref{eq:vorticity} by $\omega_N$ and integrating over the spatial variable,  we can readily observe that the proof follows analogously to the proof of the previous proposition.

\revision{
Let us also note that the approximations obtained by the spectral viscosity method are approximate solutions in the sense of Definition \ref{def:approxsolseq}. To show the $\Lip$-boundedness, we simply note that for any $\vec{\Phi} \in C^\infty(\T^2)$, and $0\le t_1<t_2 \le T$, we have \revision{from \eqref{eq:Euler}}
\begin{align*}
\langle \vec{\Phi}, \vec{u}_N(\cdot,t_2)-\vec{u}_N(\cdot,t_1)\rangle
&\le 
C(t_2-t_1) \Vert \nabla \vec{\Phi} \Vert_{L^\infty(\T^2)} \Vert \vec{u}_N \Vert_{L^\infty([0,T];L^2)}^2 \\
&\quad + 
\e_N (t_2-t_1) \Vert |\nabla|^{2} \vec{\Phi} \Vert_{L^\infty(\T^2)} \Vert \vec{u}_N \Vert_{L^\infty([0,T];L^2)}
\\
&\le 
CE_0 (t_2-t_1)\Vert \nabla \vec{\Phi} \Vert_{L^\infty(\T^2)}
+ 
\e_N \sqrt{E_0} (t_2-t_1) \Vert |\nabla|^{2} \vec{\Phi} \Vert_{L^2},
\end{align*}
where $E_0 = \int_{\T^2} |\vec{u}_0|^2 \, dx$ is the kinetic energy of the initial data $\vec{u}_0$ (cp. Proposition \ref{prop:L2bound}). Now choose $L$ large enough so that, by Sobolev embedding:
\[
H^L(\T^2;\R^2) 
\;
\embeds
\;
W^{1,\infty}(\T^2;\R^2)
\cap
H^{2}(\T^2;\R^2).
\]
Then 
\[
\langle \vec{\Phi}, \vec{u}_N(\cdot,t_2)-\vec{u}_N(\cdot,t_1)\rangle
\le C |t_2-t_1| \Vert \vec{\Phi} \Vert_{H^L(\T^2)},
\]
with a constant $C$ depending on $\sup_N \e_N$ (assumed finite) and $E_0$, but independent of $N$. Taking the supremum of all $\vec{\Phi}\in H^L(\T^2)\cap C^\infty(\T^2)$ with $\Vert \vec{\Phi} \Vert_{H^L} \le 1$ on the left, we find
\[
\Vert \vec{u}_N(\cdot,t_2)-\vec{u}_N(\cdot,t_1) \Vert_{H^{-L}(\T^2)}
\le C|t_2-t_1|,
\]
proving that $\vec{u}_N \in \Lip((0,T);H^{-L})$, with a uniformly bounded Lipschitz constant. The other two properties are easily shown; The consistency property 2. has been shown in \cite[Lemma 3.2]{LM2015}, the divergence-free property 3. is satisfied exactly according to \eqref{eq:Euler}. Thus, we have shown
\begin{thm} \label{thm:approxverified}
The sequence $\vec{u}_N$ obtained from the spectral vanishing viscosity approximation of the incompressible Euler equations form an approximate solution sequence in the sense of Definition \ref{def:approxsolseq}.
\end{thm}
}
The main tool employed to prove the convergence results in this paper will be the decay estimate for the vorticity stated below in Proposition \ref{prop:gevreydecay}. A similar idea has in fact been used in the context of the one-dimensional Burgers equation to prove the uniform $L^\infty$-boundedness of the numerical apporoximations by the SV method \cite{MadayTadmor1989}. The method employed in \cite{MadayTadmor1989}, which is based on a bootstrap argument adapted from \cite{Henshaw1990}, does not appear to allow a straightforward extension to the present case. Instead, we shall adapt a different method from \cite{Doering1995}.

To state the next proposition, we first need to define the operators $e^{\alpha |\nabla |}$ for $\alpha\in \R$, and $|\nabla |$. They are defined as distributions $\D'(\T^2)$ via their Fourier coefficients, as follows:
\begin{align}
\widehat{\left(e^{\alpha |\nabla|}\right)}_{\vec{k}}
= 
e^{\alpha |\vec{k}|}, \qquad 
\widehat{\left(|\nabla|\right)}_{\vec{k}}
= 
|\vec{k}|.
\end{align}
We can now state the spectral decay estimate, based on the method employed in \cite{Doering1995}.
\begin{prop} \label{prop:gevreydecay}
Let $\omega_N$ be a solution of the voriticty equation \eqref{eq:vorticity}, with arbitrary parameters $\e_N, m_N, a_N >0$. Let 
\begin{gather} \label{eq:defparam}
\left\{
\begin{aligned}
\beta_N &= \alpha^2 + 8 \e_N^2 m_N^2, \\
\gamma_N &= C\log(N),
\end{aligned}
\right.
\end{gather}
where $C$ is a constant such that ($\vec{k}\in \Z^2$)
\[
\sum_{|\vec{k}|\le N} \frac{1}{|\vec{k}|^2} 
\le C \log(N).
\]
Then for any $\alpha>0$, we have the estimate
\begin{equation}
\Vert e^{\alpha t|\nabla|} \omega_N(\cdot,t)\Vert_{L^2}^2
\le 
\frac{\Vert \omega_N(\cdot,0) \Vert_{L^2}^2  e^{\beta_N t/\e_N}}{1-\frac{\gamma_N \Vert \omega_N(\cdot,0)\Vert_{L^2}^2}{\beta_N} \left[e^{\beta_N t/\e_N}-1\right]},
\end{equation}
for all $t<t^\ast$, with
\[
t^\ast = \frac{\e_N}{\beta_N} \log\left(1+\frac{\beta_N}{\gamma_N \Vert \omega_N(\cdot,0)\Vert_{L^2}^2}\right).
\]

\end{prop}

\begin{proof}
To prove the spectral decay estimate, we consider the evolution equation for $e^{\alpha t|\nabla|} \omega_N$. We find from 
\[
\partial_t \omega_N = \e_N \Delta \omega_N + \e_N \Delta (R_{m_N}\ast \omega_N)- \P_N(\vec{u}_N\cdot \nabla \omega_N),
\]
that
\begin{gather}\label{eq:dtest}
\begin{aligned}
\frac{d}{dt} \revision{\frac12} \Vert e^{\alpha t|\nabla|} \omega_N \Vert_{L^2}^2
&= 
\langle e^{\alpha t|\nabla|} \omega_N, e^{\alpha t|\nabla|} \partial_t \omega_N + \alpha e^{\alpha t |\nabla|} |\nabla| \omega_N \rangle
\\
&= 
- \e_N \Vert e^{\alpha t|\nabla|} \nabla \omega_N \Vert_{L^2}^2
\\
&\quad 
+ \e_N \langle e^{\alpha t|\nabla|}\omega_N, \left(\Delta R_{m_N}\right)\ast e^{\alpha t|\nabla|}  \omega_N \rangle
\\
&\quad
-\langle e^{\alpha t|\nabla|} \omega_N, e^{\alpha t|\nabla|} \P_N( \vec{u}_N \cdot \nabla \omega_N ) \rangle 
\\
&\quad
+\alpha \langle e^{\alpha t |\nabla|} \omega_N, e^{\alpha t |\nabla|} |\nabla| \omega_N \rangle
\end{aligned}
\end{gather}
We proceed to estimate the individual terms: Firstly, taking into account that $\widehat{R}_k \le 1$, and that $R_{m_N}\ast \omega_N$ is a trigonometric polynomial of degree $2m_N$ at most, we have 
\begin{align} \label{eq:est2}
\e_N \langle e^{\alpha t|\nabla|}\omega_N, \left(\Delta R_{m_N}\right)\ast e^{\alpha t|\nabla|}  \omega_N \rangle
&\le \e_N (2m_N)^2 \Vert e^{\alpha t|\nabla|} \omega_N \Vert_{L^2}^2.
\end{align}
To analyse the non-linear term, we write it out in terms of Fourier series:
\begin{align*}
-\langle e^{\alpha t|\nabla|} \omega_N, e^{\alpha t|\nabla|}  \revision{\P_N}(\vec{u}_N \cdot \nabla \omega_N ) \rangle 
&= -\sum_{\revision{|\vec{k}|\le N}} e^{\alpha t|\vec{k}|}\widehat{\omega}_{\vec{k}}^\ast \left(e^{\alpha t |\vec{k}|}\sum_{\vec{k}'+\vec{k}''=\vec{k}} (\uhat_{\vec{k}'} \cdot \vec{k}'')\what_{\vec{k}''} \right) 
\\
&\le 
\sum_{\revision{|\vec{k}|\le N}} e^{\alpha t|\vec{k}|}|\what_{\vec{k}}| \left(e^{\alpha t |\vec{k}|}\sum_{\vec{k}'+\vec{k}''=\vec{k}} |\uhat_{\vec{k}'}| |\vec{k}''| |\what_{\vec{k}''}| \right)
\end{align*}
Since $\vec{k}=\vec{k}'+\vec{k}''$ implies by the triangle inequality 
$
e^{2\alpha t|\vec{k}|} \le e^{2\alpha t|\vec{k}'|} e^{2\alpha t |\vec{k}''|},
$
we find that the last term above is bounded by
\[
\sum_{\revision{|\vec{k}|\le N}} e^{\alpha t|\vec{k}|}|\what_{\vec{k}}| \left(\sum_{\vec{k}'+\vec{k}''=\vec{k}} e^{\alpha t |\vec{k}'|}|\uhat_{\vec{k}'}| \; e^{\alpha t |\vec{k}''|}|\vec{k}''| |\what_{\vec{k}''}| \right).
\]
Note furthermore that $|\vec{\uhat}_{\vec{k}}| = |\what_{\vec{k}}|/|\vec{k}|$. If we define a function 
\begin{equation}\label{eq:defw}
w_N \defeq \sum_{|\vec{k}|\le N} e^{\alpha t|\vec{k}|} |\what_{\vec{k}}|e^{i\vec{k}\cdot \vec{x}},
\end{equation}
then the last expression can be written in terms of $w_N$, as an integral
\[
\int w_N \left[|\nabla|^{-1} w_N\right] \left[|\nabla| w_N\right] \, dx.
\]
We thus find
\begin{gather}\label{eq:est1}
\begin{aligned}
-\langle e^{\alpha t|\nabla|} \omega_N, e^{\alpha t|\nabla|}  (\vec{u}_N \cdot \nabla \omega_N ) \rangle
&\le 
\int  w_N \left[|\nabla|^{-1} w_N\right] \left[|\nabla| w_N\right] \, dx \\
&\le
\Vert w_N \Vert_{L^2} \Vert  \left[|\nabla|^{-1} w_N\right] \Vert_{L^\infty} \Vert \left[|\nabla| w_N\right] \Vert_{L^2}.  
\end{aligned}
\end{gather}
Considering the Fourier representation of $w_N$, we have $\Vert w_N \Vert_{L^2} = \Vert e^{\alpha t |\nabla|} \omega_N \Vert_{L^2}$ and $\Vert |\nabla|w_N \Vert_{L^2} = \Vert  e^{\alpha t |\nabla|} \nabla \omega_N \Vert_{L^2}$.
To estimate $\Vert \left[|\nabla|^{-1} w_N\right]\Vert_{L^\infty}$, we note that
\begin{align*}
 \Vert \left[|\nabla|^{-1} w_N\right] \Vert_{L^\infty}
 &\le \sum_{|\vec{k}|\le N} \frac{|\widehat{w}_{\vec{k}}|}{|\vec{k}|} 
 \\
 &\le 
 \left(\sum_{|\vec{k}|\le N} \frac{1}{|\vec{k}|^2}\right)^{1/2}
 \left(\sum_{|\vec{k}|\le N} |\widehat{w}_{\vec{k}}|^2\right)^{1/2}
 \\
 &\le C^{1/2}\log(N)^{1/2} \Vert w_N \Vert_{L^2}.
\end{align*}
Combining this with \eqref{eq:est1}, and recalling the definition of $w_N$ \eqref{eq:defw}, we obtain
\begin{gather} \label{eq:est3}
\begin{aligned}
-\langle e^{\alpha t|\nabla|} \omega_N, & 
e^{\alpha t|\nabla|} \P_N(\vec{u}_N \cdot\nabla \omega_N ) \rangle  \\
&\qquad\le C^{1/2} \log(N)^{1/2} \Vert e^{\alpha t |\nabla|} \omega_N \Vert_{L^2}^2 \Vert e^{\alpha t |\nabla|} \nabla \omega_N \Vert_{L^2} 
\\ 
&\qquad\le 
\frac{C \log(N)}{2\e_N} \Vert e^{\alpha t |\nabla|} \omega_N \Vert_{L^2}^4 + \frac{\e_N}{2}\Vert e^{\alpha t |\nabla|} \nabla \omega_N \Vert_{L^2} ^2,
\end{aligned}
\end{gather}
where the last step follows form the inequality 
\[
ab \le \frac{\e}{2} a^2 + \frac{1}{2\e} b^2.
\]
Finally, we note that 
\begin{align} \label{eq:est4}
\alpha \langle e^{\alpha t|\nabla|} \omega_N, e^{\alpha t |\nabla|} |\nabla| \omega_N \rangle
\le 
\frac{\alpha^2}{2\e_N} \Vert e^{\alpha t|\nabla|}\omega_N \Vert_{L^2}^2 
+ \frac{\e_N}{2} \Vert  e^{\alpha t|\nabla|} \nabla \omega_N \Vert_{L^2}^2.
\end{align}
Combining estimates \eqref{eq:est2},\eqref{eq:est3}, \eqref{eq:est4} with \eqref{eq:dtest}, we obtain
\begin{gather*}
\frac{d}{dt} \Vert  e^{\alpha t|\nabla|}\omega_N \Vert_{L^2}^2
\le 
\left(\frac{\beta_N}{\e_N} \right)\Vert  e^{\alpha t|\nabla|}\omega_N \Vert_{L^2}^2
+ \frac{C\log(N)}{\e_N} \Vert e^{\alpha t|\nabla|}\omega_N \Vert_{L^2}^4
\end{gather*}
where $\beta_N \defeq \alpha^2 + 8\e_N^2 m_N^2$.

If we set $z \defeq \Vert  e^{\alpha t|\nabla|}\omega_N \Vert_{L^2}^2$ and the short-hand notation $\gamma_N = C\log(N)$, then we have the differential inequality
\[
\frac{dz}{dt} \le \frac{\beta_N}{\e_N} z + \frac{\gamma_N}{\e_N} z^2.
\]
Let $y \defeq ze^{-\beta_N t/\e_N}$, then
\[
\frac{dy}{dt} \le e^{-\beta_N t/\e_N} \frac{\gamma_N}{\e_N} z^2 = e^{\beta_N t/\e_N} \frac{\gamma_N}{\e_N} y^2.
\]
Integration yields
\[
\frac{1}{y_0} - \frac{1}{y} = \int_{y_0}^y \frac{dy}{y^2} \le \frac{\gamma_N}{\beta_N} \left[ e^{\beta_N t/\e_N} - 1 \right], 
\]
or
\[
y \le y_0 \frac{1}{1-\dfrac{\gamma_N y_0}{\beta_N}\left[e^{\beta_N t/\e_N} - 1\right]}.
\]
With $y=e^{-\beta_N t/\e_N}\Vert e^{\alpha t |\nabla|} \omega_N \Vert_{L^2}^2$ and $y_0 = \Vert \omega_N(\cdot,0) \Vert^2_{L^2}$, this becomes
\[
{
\Vert e^{\alpha t |\nabla|} \omega_N \Vert_{L^2}^2
\le  
\frac{\Vert \omega_N(\cdot,0) \Vert_{L^2}^2 e^{\beta_N t/\e_N}}
{1-\dfrac{\gamma_N \Vert \omega_N(\cdot,0) \Vert_{L^2}^2}{\beta_N}\left[e^{\beta_N t/\e_N} - 1\right]}.
}
\]
\end{proof}
Note that the $L^2$ norm on the left provides a very crude upper bound for the Fourier coefficients of $\omega_N$ via
\begin{equation} \label{eq:crudespectral}
e^{2\alpha t |\vec{k}|} |\widehat{\omega}_{\vec{k}}|^2
\le 
\Vert e^{\alpha t |\nabla|} \omega_N \Vert_{L^2}^2.
\end{equation}
The following corollaries are then immediate
\begin{cor}[$L^p$ Fourier decay; $p\ge 2$] \label{cor:L2decay}
With the notation of Proposition \ref{prop:gevreydecay}; if $\omega_0 \in L^p$ and $p\ge 2$, then there exist absolute constants $A,B>0$ such that
\[
|\widehat{\omega}_{\vec{k}}(t)|^2
\le 
A\Vert \omega_0 \Vert_{L^p}^2 \left(1+\frac{\beta_N}{B\log(N)\Vert \omega_0 \Vert_{L^p}^2}\right) e^{-2\alpha t^\ast_N |\vec{k}|},
\]
for $t \in [t_N^\ast,T]$, and 
\[
t^\ast_N 
= 
\frac{\e_N}{\beta_N} \log\left(1+\frac{\beta_N}{B \log(N) \Vert \omega_0 \Vert_{L^p}^2}\right).
\]
\end{cor}

\begin{proof}
Fix $t_0\ge 0$. We note that Proposition \ref{prop:gevreydecay} applied to $(x,t) \mapsto \omega_N(x,t_0+t)$, together with the simple estimate \eqref{eq:crudespectral} yields
\begin{equation} \label{eq:extendt0} 
e^{2\alpha  |\vec{k}| (t_0+t_N^\ast)} |\widehat{\omega}_{\vec{k}}(t_0+t_N^\ast)|^2
\le 
\frac{\Vert \omega_N(\cdot,t_0) \Vert_{L^2}^2  e^{\beta_N t_N^\ast/\e_N}}{1-\frac{\gamma_N \Vert \omega_N(\cdot,t_0)\Vert_{L^2}^2}{\beta_N} \left[e^{\beta_N t_N^\ast/\e_N}-1\right]},
\end{equation}
The right-hand side of this estimate is a non-decreasing function of $\Vert \omega_N(\cdot,t_0)\Vert_{L^2}^2$. Since $\Vert \omega_N(\cdot,t_0)\Vert_{L^2}^2 \le \Vert \omega_N(\cdot,0)\Vert_{L^2}^2$ for all $t_0\ge 0$, it follows that \eqref{eq:extendt0} remains true if we replace $\Vert \omega_N(\cdot,t_0)\Vert_{L^2}^2$ on the right by $\Vert \omega_N(\cdot,0)\Vert_{L^2}^2$.
Since the right-hand side is then independent of $t_0 \ge 0$, we find that for any $t\equiv t_0+t_N^\ast\in [t_N^\ast,T]$:
\begin{equation} \label{eq:L2estspectral}
e^{2\alpha  |\vec{k}| t} |\widehat{\omega}_{\vec{k}}(t)|^2
\le 
\frac{\Vert \omega_N(\cdot,0) \Vert_{L^2}^2  e^{\beta_N t_N^\ast/\e_N}}{1-\frac{\gamma_N \Vert \omega_N(\cdot,0)\Vert_{L^2}^2}{\beta_N} \left[e^{\beta_N t_N^\ast/\e_N}-1\right]}.
\end{equation}
If $\omega_0\in L^p$, $p\ge 2$, then we have a simple estimate
\begin{equation} \label{eq:L2Lpest}
\Vert \omega_N(\cdot,0) \Vert_{L^2}
\le 
\Vert \omega_0 \Vert_{L^2} 
\le 
K \Vert \omega_0 \Vert_{L^p}
\end{equation}
where the last estimate follows from the Holder inequality applied to $\omega_N = 1 \cdot \omega_N$ (in fact $K = (2\pi)^4$ provides a uniform bound for all $2\le p\le \infty$). Again, by the monotonicity of the right-hand side in \eqref{eq:L2estspectral}, we finally obtain
\[
e^{2\alpha  |\vec{k}| t} |\widehat{\omega}_{\vec{k}}(t)|^2
\le 
\frac{K\Vert \omega_0 \Vert_{L^p}^2  e^{\beta_N t_N^\ast/\e_N}}{1-\frac{K\gamma_N \Vert \omega_0\Vert_{L^p}^2}{\beta_N} \left[e^{\beta_N t_N^\ast/\e_N}-1\right]}.
\]
Replacing $t_N^\ast$ by its definition in the statement of this corollary, we find
\[
e^{\beta_N t_N^\ast/\e_N}
= 
1 + \frac{\beta_N}{2\gamma_N \Vert \omega_0 \Vert_{L^p}^2},
\]
and we also recall that $\gamma_N = C \log(N)$. Hence, for $t_N^\ast \le t$:
\[
e^{2\alpha t_N^\ast} |\widehat{\omega}_{\vec{k}}(t)|^2
\le 
e^{2\alpha t} |\widehat{\omega}_{\vec{k}}(t)|^2
\le 
2K\Vert \omega_0 \Vert_{L^p}^2  \left(1 + \frac{\beta_N}{2KC\log(N)\Vert \omega_0 \Vert_{L^p}^2} \right).
\]
The claim thus follows with $A=2K=2(2\pi)^4$ and $B=2KC$.
\end{proof}

On the other hand, considering now $1<p<2$, our bound worsens and depends on $a_N$ (which defines the projection of the initial data via $K_{a_N}$).
\begin{cor}[$L^p$ Fourier decay; $1<p<2$] \label{cor:Lpdecay0}
With the notation of Proposition \ref{prop:gevreydecay}; if $\omega_0 \in L^p$ and $1<p<2$, then there exist absolute constants $A,B>0$ such that
\[
|\widehat{\omega}_{\vec{k}}(t)|^2
\le 
A a_N^{2\left(\frac{2}{p}-1\right)} \Vert \omega_0 \Vert_{L^p}^2 \left(1+\frac{\beta_N}{B\log(N)a_N^{2\left(\frac{2}{p}-1\right)} \Vert \omega_0 \Vert_{L^p}^2}\right) e^{-2\alpha t^\ast_N |\vec{k}|},
\]
for $t \in [t_N^\ast,T]$, and
\[
t^\ast_N 
= 
\frac{\e_N}{\beta_N} \log\left(1+\frac{\beta_N}{B \log(N) a_N^{2\left(\frac{2}{p}-1\right)} \Vert \omega_0 \Vert_{L^p}^2}\right).
\]
\end{cor}

\begin{proof}
The proof is a repetition of the proof of Corollary \ref{cor:L2decay}, except that the estimate \eqref{eq:L2Lpest} is replaced by the Bernstein inequality in Theorem \ref{thm:BernsteinLpLq}, yielding an estimate
\[
\Vert \omega_N(\cdot,0) \Vert_{L^2}^2
=
\Vert K_{a_N} \ast \omega_0 \Vert_{L^2}^2
\le 
K a_N^{2\left(\frac{2}{p}-1\right)} \Vert \omega_0 \Vert_{L^p}^2,
\]
with a constant $K>0$ depending only on $p$.
\end{proof}

Note that when $\omega_0 \in L^1(\T^2)$, the Bernstein inequality which was used to prove Corollary \ref{cor:Lpdecay0} is no longer available. Instead, we can prove the following result, which is valid for general initial data $\omega_0\in H^{-1}$:

\begin{cor}[general Fourier decay] \label{cor:Lpdecay}
With the notation of Proposition \ref{prop:gevreydecay}; if $u_0 \in L^2$ (and hence $\omega_0 \in H^{-1}$), then there exist absolute constants $A,B>0$ such that
\[
|\widehat{\omega}_{\vec{k}}(t)|^2
\le 
A a_N^{2} \Vert u_0 \Vert_{L^2}^2 \left(1+\frac{\beta_N}{B\log(N)a_N^{2} \Vert u_0 \Vert_{L^2}^2}\right) e^{-2\alpha t^\ast_N |k|},
\]
for $t \in [t_N^\ast,T]$, and
\[
t^\ast_N 
= 
\frac{\e_N}{\beta_N} \log\left(1+\frac{\beta_N}{B \log(N) a_N^{2} \Vert u_0 \Vert_{L^2}^2}\right).
\]
\end{cor}

\begin{proof}
The proof is again essentially a repetition of the proof of Corollary \ref{cor:L2decay}, except that the estimate \eqref{eq:L2Lpest} is replaced by the a priori estimate
\[
\Vert \omega_N(\cdot,0) \Vert_{L^2}^2
=
\Vert K_{a_N} \ast \omega_0 \Vert_{L^2}^2
\le 
(2a_N)^2 \Vert u_0 \Vert_{L^2}^2.
\]
\end{proof}

Let us combine these corollaries in a single theorem:

\begin{thm} \label{thm:allspectraldecay}
Let $u_0 \in L^2$ be given initial data for the incompressible Euler equations. Then, there exist constants $A,B$ depending only on the initial data, such that the approximations, $\omega_N = \curl(\vec{u}_N)$, obtained from the spectral viscosity method satisfy the following estimate on their Fourier coefficients:
\[
|\widehat{\omega}_{\vec{k}}(t)|^2
\le 
A a_N^{\nu(p)} \left(1 + \frac{\beta_N}{B a_N^{\nu(p)} \log(N)}\right) e^{-2\alpha t_N^\ast |\vec{k}|},
\]
for all $t\in [t_N^\ast,T]$ and $\alpha>0$. Here $\beta_N = \alpha^2 + 8\e_N^2 m_N^2$, and we have defined
\[
t_N^\ast 
= 
\frac{\e_N}{\beta_N}
\log\left(
1 + \frac{\beta_N}{B a_N^{\nu(p)} \log(N)}
\right),
\]
and 
\[
\nu(p) = 
\begin{cases}
0, & \text{if }\omega_0 \in L^p, \; 2\le p \le \infty, \\
2\left(\frac{2}{p}-1\right), & \text{if }\omega_0 \in L^p, \; 1<p<2, \\
2, & \text{for arbitrary } \omega_0 \in H^{-1}.
\end{cases}
\]
\qed
\end{thm}

We next observe that we can choose the sequences $\e_N \to 0$, $m_N, a_N \to \infty$ in a suitable manner, such that the Fourier coefficients in the range $N/2\le |\vec{k}| \le N$ decay superpolynomially in $N$. Suitable conditions on the asymptotic behaviour are described in the lemma below.
\begin{lem} \label{lem:Fourierdecay1}
We follow the notation of Theorem \ref{thm:allspectraldecay}. If 
\begin{gather}
\left\{
\begin{gathered}
\beta_N \sim a_N^{\nu(p)} \log(N)^r, 
\quad \alpha \sim \sqrt{\beta_N}, \\
\e_N \sim \frac{a_N^{\nu(p)/2}\log(N)^s}{N},
\end{gathered}
\right.
\end{gather}
with $s+2 < r \le 2s - 4$, 
then
\begin{gather}
t_N^\ast = o\left(\frac{1}{a_N^{\nu(p)} N \log(N)^2}\right) \to 0,
\end{gather}
and
\begin{gather}
\alpha t_N^\ast N \gtrsim \log(N)^2.
\end{gather}
\end{lem}
\begin{proof}
We follow the notation of Proposition \ref{prop:gevreydecay}. We recall that $t_N^\ast$ is defined as
\begin{align*}
t_N^\ast 
&= \frac{\e_N}{\beta_N} \log\left(1+\frac{\beta_N}{B\gamma_Na_N^{\nu(p)}}\right).
\end{align*}
Under the assumptions of this Lemma, we have
\[
t_N^\ast \sim \frac{1}{a_N^{\nu(p)/2} N} (\log N)^{s-r} (\log \log N).
\]
Therefore, if $r>s+2$, it follows that
\[
t_N^\ast \ll \frac{1}{a_N^{\nu(p)/2}N(\log N)^2}.
\]
At the same time
\[
\alpha t^\ast_N N
\sim \frac{\e_N N}{\beta_N^{1/2}} (\log \log N)
\sim (\log N)^{s-r/2} (\log \log N),
\]
satisfies
\[
\alpha t^\ast_N N \gtrsim (\log N)^2,
\]
for $s-r/2\ge 2$ (or equivalently $r \le 2s-4$), as claimed.
\end{proof}
\if{0}{
\begin{lem}[polynomial correction] \label{lem:Fourierdecay2}
We follow the notation of Theorem \ref{thm:allspectraldecay}. If 
\begin{gather}
\left\{
\begin{gathered}
\beta_N \sim a_N^{\nu(p)} N^r, 
\quad \alpha \sim \sqrt{\beta_N}, \\
\e_N \sim \frac{a_N^{\nu(p)/2} N^s}{N},
\end{gathered}
\right.
\end{gather}
with $s < r < 2s$, 
then
\begin{gather}
t_N^\ast = o\left(\frac{1}{a_N^{\nu(p)/2} N \log(N)^2}\right) \to 0,
\end{gather}
and
\begin{gather}
\alpha t_N^\ast N \gtrsim N^\delta \gg \log(N)^2,
\end{gather}
where $\delta = s-r/2>0$.
\end{lem}

\begin{proof}
Again, by Theorem \ref{thm:allspectraldecay}, we have
\[
|\what_{\vec{k}}(t)|^2
\le 
A a_N^{\nu(p)} \left(1 + \frac{\beta_N}{B\gamma_N a_N^{\nu(p)}}\right) e^{-2\alpha t |\vec{k}|}, \quad \text{for } t\ge t_N^\ast.
\]
We recall that $t_N^\ast$ is defined as
\begin{align*}
t_N^\ast 
&= \frac{\e_N}{\beta_N} \log\left(1+\frac{\beta_N}{B\gamma_Na_N^{\nu(p)}}\right).
\end{align*}
Under the assumptions of this Lemma, we have
\[
t_N^\ast \sim \frac{1}{a_N^{\nu(p)/2} N} N^{s-r} (\log N).
\]
Therefore, if $r>s$, it follows that
\[
t_N^\ast \ll \frac{1}{a_N^{\nu(p)/2}N(\log N)^2}.
\]
At the same time
\[
\alpha t^\ast_N N
\sim \frac{\e_N N}{\beta_N^{1/2}} (\log  N)
\sim N^{s-r/2} (\log N),
\]
satisfies
\[
\alpha t^\ast_N N \gtrsim N^{s-r/2} \gg  (\log N)^2,
\]
for $s-r/2 > 0$ (or equivalently $r < 2s$), as claimed.
\end{proof}
\fi

Based on Lemma \ref{lem:Fourierdecay1}, we can now deduce the following proposition.
\begin{thm} \label{thm:fourierdecay}
With the notation of Theorem \ref{thm:allspectraldecay}. Choose the free parameters $\e_N, a_N, m_N$ as follows
\begin{gather}
\left\{
\begin{gathered} \label{eq:choiceL1}
m_N\lesssim N^\theta, \quad\text{where }0 \leq \theta< \left(2 + \frac{\nu(p)}{2}\right)^{-1}, \\
a_N \sim 
\begin{cases}
N^\theta, &(\nu(p) \ne 0), \\
N,   &(\nu(p) = 0),
\end{cases} \\
\e_N \sim \frac{a_N^{\nu(p)/2} \log(N)^s}{N}, \; (s>6)
\end{gathered}
\right.
\end{gather}
Then, $\alpha$ in Theorem \ref{thm:allspectraldecay} can be chosen such that the assumptions of Lemma \ref{lem:Fourierdecay1} are satisfied, 
and for any $\sigma>0$, there exists a constant $C_\sigma>0$, such that
\begin{equation} \label{eq:specdecay}
|\what_{\vec{k}}(t)| \le C_\sigma N^{-\sigma}, \quad \text{for }N/2\le |\vec{k}| \le N, \quad t\in [t_N^\ast,\infty), 
\end{equation}
where $t_N^\ast \to 0$, at a convergence rate 
\begin{align} \label{eq:tNrate}
t_N^\ast \ll \frac{1}{a_N^{\nu(p)/2} N \log(N)^2}.
\end{align}
\end{thm}

\begin{proof}
We begin by proving that the assumptions of Lemma \ref{lem:Fourierdecay1} are satisfied. To this end, we choose the free parameter $\alpha$ such that $\alpha \sim a_N^{\nu(p)/2} \log(N)^{r/2}$ with exponent $r$ satisfying $s+2<r\le 2s-4$ (this requires $s>6$). Note that, by our choice of the other parameters, we now have
\begin{align*}
\e_N m_N^2
&\lesssim \frac{m_N^2 a_N^{\nu(p)/2} \log(N)^{s}}{N} 
\sim \frac{N^{\theta(2+\nu(p)/2))}\log(N)^{s}}{N}.
\end{align*}
Since $\theta < \left(2+\nu(p)/2\right)^{-1}$ (with a strict inequality), it follows that the term on the right-hand side converges to $0$ as $N\to \infty$. In particular, this implies that $\alpha^2 \sim m_N^{\nu(p)} \log(N)^r \gg \e_N^2 m_N^2$. So that
\begin{align*}
\beta_N = \alpha^2 + 8 \e_N^2m_N^2 \sim a_N^{\nu(p)}\log(N)^r.
\end{align*}
Thus, the assumptions of Lemma \ref{lem:Fourierdecay1} can be satisfied, and it follows that 
\[
\alpha t_N^\ast N \gtrsim \log(N)^2.
\]
From the spectral decay estimate of Theorem \ref{thm:allspectraldecay}, it follows that (for $|\vec{k}|\ge N/2$ and $t\ge t_N^\ast$):
\begin{align*}
|\widehat{\omega}_{\vec{k}}(t)|^2
&\lesssim 
AB^{-1} a_N^{\nu(p)} \log(N)^{r-1} e^{-2\log(N)^2 |k|/N} \\
&\lesssim N e^{-\log(N)^2} \\
&= N^{1-\log(N)},
\end{align*}
with a uniform implied constant. In particular, for any $\sigma>0$, we will have for $\log(N)>2\sigma+1$, and any $|\vec{k}|\ge N/2$:
\[
|\widehat{\omega}_{\vec{k}}(t)|^2 \lesssim N^{-2\sigma}, 
\quad \text{for }t\ge t_N^\ast.
\]
The convergence rate of $t_N^\ast$ has already been estimated in Lemma \ref{lem:Fourierdecay1}. 
\end{proof}

It will be convenient to state the following definition:

\begin{defn} \label{dfn:spectraldecay}
We will say that a choice of parameters $\e_N, m_N, a_N$ and Fourier kernels $Q_N, K_{a_N}$ for the SV method \emph{ensures spectral decay}, provided that the conclusions (estimates \eqref{eq:specdecay}, \eqref{eq:tNrate}) of Theorem \ref{thm:fourierdecay} hold true.
\end{defn}

As a consequence of Theorem \ref{thm:fourierdecay}, we next show that the projection error vanishes in the limit $N\to \infty$.
\begin{lem}\label{lem:err1}
If the parametrization for the SV method ensures spectral decay, then the projection error can be bounded from above, i.e. there exists a constant $C>0$ depending on the initial data $\vec{u}_0$, but independent of $N$, such that for all $t\in [t_N^\ast,\infty)$ and for any $1\le p<\infty$:
\[
\Vert (I-\P_N)(\vec{u}_N(t) \cdot \nabla \omega_N(t)) \Vert_{L^p}
\le C N^{-1} \Vert \omega_N(t) \Vert_{L^p}.
\]
Alternatively, one can find a constant $C'>0$, again depending on the initial data, but independent of $N$, such that
\[
\Vert (I-\P_N)(\vec{u}_N(t) \cdot \nabla \omega_N(t)) \Vert_{L^p}
\le C' N^{-1}.
\]
\end{lem}

\begin{proof}
The basic idea of this Lemma is that the trigonometric polynomial $(I-\P_N)(\vec{u}_N\cdot \nabla \omega_N)$ can be written as a sum
\begin{align*}
(I-\P_N)(\vec{u}_N\cdot \nabla \omega_N)
= 
(I-\P_N)\left(\vec{u}_N\cdot \nabla \omega_N^{>N/2}\right) 
+ (I-\P_N)\left(\vec{u}_N^{>N/2}\cdot \nabla \omega_N^{\le N/2}\right)
\end{align*}
where $(\ldots)^{>N/2}$ is a sum over Fourier modes $|\vec{k}|>N/2$, and $(\ldots)^{\le N/2}$ over Fourier modes $|\vec{k}|\le N/2$. Due to the spectral decay of the $(\ldots)^{>N/2}$-factors, very rough estimates on the $\omega_N^{>N/2}$ and $\vec{u}_N^{>N/2}$ can then be used to prove the Lemma. We proceed to provide the details. 

For $1\le p < \infty$:
\begin{gather*}
\begin{aligned}
\Vert (I-\P_N)(\vec{u}_N\cdot \nabla \omega_N) \Vert_{L^p}
&\le 
\Vert \vec{u}_N\cdot \nabla \omega_N^{>N/2} \Vert_{L^p}
+ 
\Vert \vec{u}_N^{>N/2}\cdot \nabla \omega_N^{\le N/2} \Vert_{L^p} 
\\
&\le 
C\Vert 
\vec{u}_N\Vert_{L^\infty} \Vert \nabla \omega_N^{>N/2} 
\Vert_{L^\infty}
+ 
C\Vert 
\vec{u}_N^{>N/2}\Vert_{L^\infty} \Vert \nabla \omega_N^{\le N/2} 
\Vert_{L^\infty}.
\end{aligned}
\end{gather*}
We further estimate
\begin{align*}
\Vert \vec{u}_N \Vert_{L^\infty}
\le \sum_{|\vec{k}|\le N} |\widehat{\vec{u}}_{\vec{k}}|
\le \left(\sum_{|\vec{k}|\le N} 1^2\right)^{1/2} \left(\sum_{|\vec{k}|\le N} |\widehat{\vec{u}}_{\vec{k}}|^2\right)^{1/2}
\le CN \Vert \vec{u}_N \Vert_{L^2},
\end{align*}
and, similarly, using also Bernstein's inequality,
\begin{align*}
\Vert \nabla \omega_N^{<N/2} \Vert_{L^\infty}
\le CN\Vert \omega_N^{<N/2} \Vert_{L^\infty} 
\le CN^2\Vert \omega_N^{<N/2} \Vert_{L^2}
\le CN^3\Vert \vec{u}_N \Vert_{L^2}.
\end{align*}
We thus obtain an estimate of the form 
\[
\Vert (I-\P_N)(\vec{u}_N\cdot \nabla \omega_N) \Vert_{L^p}
\le CN\Vert \vec{u}_N \Vert_{L^2} \left(\Vert \nabla \omega_N^{> N/2} 
\Vert_{L^\infty} + N^2\Vert \vec{u}_N^{> N/2} 
\Vert_{L^\infty}\right)
\]
We proceed to (crudely) estimate, for either of the two cases we have, 
\[
\Vert \vec{u}_N \Vert_{L^2}
\le 
\left\{
\begin{matrix}
\Vert \omega_N \Vert_{L^2}  \\
\Vert \vec{u}_0 \Vert_{L^2}
\end{matrix}
\right\}
\le
\left\{
\begin{matrix}
C\Vert \omega_N \Vert_{L^\infty}  \\
\Vert \vec{u}_0 \Vert_{L^2}
\end{matrix}
\right\}
\le
\left\{
\begin{matrix}
CN^2\Vert \omega_N \Vert_{L^p} \\
\Vert \vec{u}_0 \Vert_{L^2}
\end{matrix}
\right\}.
\]
We finally note that due to the spectral decay \eqref{eq:specdecay}, that for any $\sigma>0$ there exists $C_\sigma>0$ such that
\[
\Vert \nabla \omega_N^{>N/2}\Vert_{L^\infty} + 
N^2 \Vert \vec{u}_N^{>N/2} \Vert_{L^\infty}
\lesssim N^{-\sigma}.
\]
In particular, we chan choose $\sigma$ sufficiently large and find constants $C,C'$ depending only on the initial data, to ensure that 
\[
\Vert (I-\P_N)(\vec{u}_N\cdot \nabla \omega_N) \Vert_{L^p}
\le 
\left\{
\begin{matrix}
C' N^{-1} \Vert \omega_N \Vert_{L^p},  \\
C N^{-1}.
\end{matrix}
\right.
\]
\end{proof}
Next, we show that the second discretization error can also be bounded from above.
\begin{lem} \label{lem:err2}
For any $1\le p \le \infty$, we have
\[
\Vert \Delta(R_{m_N} \ast \omega_N) \Vert_{L^p}
\le 
2 m_N^2 \Vert R_{m_N} \Vert_{L^1} \Vert \omega_N \Vert_{L^p}.
\]
For suitably chosen $R_{m_N}$, the $L^1$ norm on the right hand side can furthermore be bounded by 
\[
\Vert R_{m_N} \Vert_{L^1} \le C \log(N)^2.
\]
For the last estimate, see Maday and Tadmor \cite[Appendix]{MadayTadmor1989}.
\end{lem}
Based on Lemmas \ref{lem:err1} and \ref{lem:err2}, we can now control the error terms on the right hand side. We conclude this section by proving the following theorem, stating that the $L^p$-norm is uniformly controlled for $t\ge t_N^\ast$.
\begin{thm}[$L^p$ control after short time]\label{thm:Lpcontrol1}
If the numerical parameters ensure spectral decay, then for any $1\le p < \infty$, there exists a sequence $c_N \to 0$ such that
\[
\Vert \omega_N(\cdot, t) \Vert_{L^p}
\le 
\left(1+c_N t\right)\Vert \omega_N(\cdot,t_N^\ast) \Vert_{L^p}, 
\quad \text{ for all }t\ge t_N^\ast.
\]
\end{thm}

\begin{proof}
We start from equation \eqref{eq:vorterr}:
\begin{gather*}
\partial_t \omega_N + \vec{u}_N \cdot \nabla \omega_N - \e_N \Delta \omega_N
= (I-\P_N)(\vec{u}_N\cdot \nabla \omega_N) + \e_N \Delta R_{m_N} \ast \omega_N.
\end{gather*}
Multiplying by $|\omega_N|^{p-1}\sign(\omega_N)$ (or a smooth approximation thereof), and integrating over $x$, we find
\[
\frac{d}{dt} \Vert \omega_N(\cdot,t) \Vert_{L^p}^p
\le 
\langle |\omega_N|^{p-1}, |\err_1|\rangle + \langle |\omega_N|^{p-1}, |\err_2|\rangle
\]
From Holder's inequality, we obtain for either of the two numerical error terms on the right:
$
\langle |\omega_N|^{p-1}, |\err|\rangle
\le \Vert \omega_N \Vert_{L^p}^{p-1} \Vert \err \Vert_{L^p}
$. Using Lemmas \ref{lem:err1} and \ref{lem:err2}, and dividing by $\Vert \omega_N(\cdot,t) \Vert_{L^p}^{p-1}$ on both sides, we obtain (for $t\ge t_N^\ast$)
\[
\frac{d}{dt} \Vert \omega_N(\cdot,t) \Vert_{L^p}\le C\left[N^{-1}\Vert \vec{u}_0\Vert_{L^2} + \e_N m_N^2 \log(N)^2\right] \Vert \omega_N(\cdot,t) \Vert_{L^p}.
\]
After an integration over $[t_N^\ast,t]$, it follows that
\[
\Vert \omega_N(\cdot,t) \Vert_{L^p}
\le 
\Vert \omega_N(\cdot,t_N^\ast) \Vert_{L^p}\exp\left(c_N t \right),
\]
where $c_N = C\left[N^{-1}\Vert \vec{u}_0\Vert_{L^2} + \e_N m_N^2 \log(N)^2\right] \to 0$.
\end{proof}

\section{Short-time estimates}
\label{sec:shorttime}

In the last section, we have shown that the numerical parameters can be chosen to ensure the spectral decay of the Fourier modes $N/2\le |\vec{k}|\le N$ for $t \in [t_N^\ast,+\infty)$, where 
\[
t_N^\ast \ll \frac{1}{a_N^{\nu(p)/2} N \log(N)^2}.
\]
As a consequence, we have proven $L^p$-control of the vorticity for $t\ge t_N^\ast$ in terms of $\Vert \omega_N(\cdot,t_N^\ast) \Vert_{L^p}$. In this section, we will bridge the gap $[0,t_N^\ast]$ and prove short time $L^p$ control of the vorticity for the initial interval $0\le t\le t_N^\ast$ in terms of $\Vert \omega_N(\cdot,0) \Vert_{L^p}$. We will prove the following theorem,

\begin{thm}\label{thm:Lpcontrol2}
If $\omega_N(\cdot,0)\in L^p$, $1\le p < \infty$, then there exists a sequence $c_N\to 0$ (depending only on the initial data and $p$), such that
\[
\Vert \omega_N(\cdot,t) \Vert_{L^p} 
\le 
\left(1+c_N\right)\Vert \omega_N(\cdot,0) \Vert_{L^p}
+ c_N, \quad \text{ for all } \, t\in [0,t_N^\ast].
\]
\end{thm}
\begin{proof}
This \revision{will follow} from Lemma \ref{lem:shortL1} for $p=1$, Lemma \ref{lem:shortLp1} for $1<p< 2$, and Lemma \ref{lem:shortLp2} for $p>2$\revision{, below}. For $p=2$, the result follows from Proposition \ref{prop:L2vortbound}.
\end{proof}
To complete the proof of \revision{Theorem} \ref{thm:Lpcontrol2}, we now consider the cases $p=1$, $1<p<2$ and $2<p<\infty$, separately. We begin by observing that 
\begin{align} \label{eq:evoeqshort}
\partial_t \omega_N = -\P_N(\vec{u}_N\cdot \nabla \omega_N) + \e_N \Delta \omega_N + \e_N \Delta (R_{m_N}\ast \omega_N),
\end{align}
implies that for any $1 \le p <\infty$:
\begin{gather}\label{eq:estimateshort}
\begin{aligned}
\frac{d}{dt} \Vert \omega_N \Vert_{L^p}
&\le 
\Vert \P_N(\vec{u}_N\cdot \nabla \omega_N) \Vert_{L^p}
+ \e_N \Vert \Delta(R_{m_N}\ast \omega_N)\Vert_{L^p} \\
&\le C \log(N)^2 \Vert \vec{u}_N \cdot \nabla \omega_N \Vert_{L^p}
+ C\e_N m_N^2 \log(N)^2 \Vert \omega_N \Vert_{L^p},
\end{aligned}
\end{gather}
for some constant $C>0$.
Setting $\delta_N \defeq C\e_N m_N^2 \log(N)^2$, we note that $\delta_N\to 0$ and $\delta_N \ge 0$, we find
\begin{equation}\label{eq:Lpeq}
\frac{d}{dt}\left(\Vert \omega_N \Vert_{L^p} e^{-\delta_N t}\right)
\le C \log(N)^2 \Vert \vec{u}_N \cdot \nabla \omega_N \Vert_{L^p}.
\end{equation}
On the right hand side, we have used the simple estimate $e^{-\delta_N t} \le 1$.

We will now estimate the non-linear term separately for the different values of $p$. We begin with the case $p=1$.
\begin{lem}[Short-time $L^1$-control]\label{lem:shortL1}
If $\omega_N(\cdot,0)\in L^1$, then there exists a constant $C>0$ such that
\[
\Vert \omega_N(\cdot,t) \Vert_{L^1} 
\le 
\Vert \omega_N(\cdot,0) \Vert_{L^1} e^{\delta_N t_N^\ast}
+ C\Vert \vec{u}_0 \Vert_{L^2}^2 \, [a_N N \log(N)^2] t_N^\ast,
\]
for all $t\in [0,t_N^\ast]$. Here $\delta_N \to 0$.
\end{lem}

\begin{proof}
We start by noting that
\begin{align*}
\Vert \vec{u}_N(t)\cdot \nabla \omega_N(t) \Vert_{L^1}
&\le 
C\Vert \vec{u}_N(t) \Vert_{L^2} \Vert \nabla \omega_N(t) \Vert_{L^2} \\
&\le 
CN\Vert \vec{u}_N(t) \Vert_{L^2} \Vert \omega_N(t) \Vert_{L^2}.
\end{align*}
From the a priori $L^2$-bounds for $\vec{u}_N$, $\omega_N$, we can furthermore estimate the right-hand side by $\Vert \vec{u}_N(\cdot,t) \Vert_{L^2} 
\le \Vert \vec{u}_0 \Vert_{L^2}$, and
\begin{align*}
\Vert \omega_N(\cdot,t) \Vert_{L^2}
&\le \Vert \omega_N(\cdot,0) \Vert_{L^2} 
=\Vert K_{a_N} \ast \omega_0 \Vert_{L^2} 
\le C a_N \Vert K_{a_N} \ast \vec{u}_0 \Vert_{L^2}
\le C a_N \Vert \vec{u}_0 \Vert_{L^2}.
\end{align*}
From \eqref{eq:Lpeq}, we now find
\begin{align*}
\frac{d}{dt}\left(\Vert \omega_N(\cdot,t)\Vert_{L^1} e^{-\delta_N t} \right)
&\le 
CN a_N \log(N)^2 \Vert \vec{u}_0 \Vert_{L^2}^2.
\end{align*}
Integrating in time from $0$ to $t$, we find, for some constant $C$,
\begin{align*}
\Vert \omega_N(\cdot,t)\Vert_{L^1} 
&\le 
\Vert \omega_N(\cdot,0)\Vert_{L^1} e^{\delta_N t}
+ CN a_N e^{\delta_N t} \log(N)^2 \Vert \vec{u}_0 \Vert_{L^2}^2 t.
\end{align*}
The right hand side is uniformly bounded for $t\in [0,t_N^\ast]$, by 
\[
\Vert \omega_N(\cdot,t)\Vert_{L^1} 
\le
\Vert \omega_N(\cdot,0)\Vert_{L^1} e^{\delta_N t_N^\ast}
+ CN a_N e^{\delta_N t_N^\ast} \log(N)^2 \Vert \vec{u}_0 \Vert_{L^2}^2 t.
\]
Furthermore, since $\delta_N t_N^\ast \to 0$, we can absorb the (uniformly bounded) factor $e^{\delta_N t_N^\ast}$ by increasing constant $C$, yielding the claimed estimate.
\end{proof}

\begin{lem}\label{lem:shortLp1}
If $\omega_N(\cdot,0)\in L^p$, $1< p < 2$, then there exists a constant $C>0$ (depending only on $p$ and on the initial conditions), such that
\[
\Vert \omega_N(\cdot,t) \Vert_{L^p} 
\le 
\Vert \omega_N(\cdot,0) \Vert_{L^p} \exp\left(Ca_N^{\nu(p)/2} N \log(N)^2 t_N^\ast\right),
\]
for $t\in [0,t_N^\ast]$.
\end{lem}

\begin{proof}
The non-linear term can be estimated by
\begin{align}
\Vert \vec{u}_N(t)\cdot \nabla \omega_N(t) \Vert_{L^p}
&\le 
C\Vert \vec{u}_N(t) \Vert_{L^{p^\ast}} \Vert \nabla \omega_N(t) \Vert_{L^2},
\end{align}
with $p^\ast = 2p/(2-p)>p$, chosen so that $\frac{1}{p^\ast} + \frac12 = \frac1p$. Note that this $p^\ast$ corresponds precisely to the gain we can get by combining the Sobolev embedding $W^{1,p} \embeds L^{p^\ast}$, and the Calderon-Zygmund estimate for the singular integral operator mapping $\omega_N \mapsto \vec{u}_N$; namely
\[
\Vert \vec{u}_N \Vert_{L^{p^\ast}}
\le
C\Vert \nabla \vec{u}_N \Vert_{L^p} 
\le 
C_p \Vert \omega_N \Vert_{L^p}.
\]
On the other hand, the $L^2$-norm of $\nabla\omega_N$ can be estimated by
\[
\Vert \nabla \omega_N(\cdot,t) \Vert_{L^2}
\le 
N \Vert \omega_N(\cdot,t) \Vert_{L^2}
\le N \Vert \omega_N(\cdot,0) \Vert_{L^2}.
\]
The last term can be estimated to yield
\[
\Vert \nabla \omega_N(\cdot,t) \Vert_{L^2}
\le C a_N^{\nu(p)/2} N \Vert \omega_0 \Vert_{L^p}.
\]
Thus the nonlinear term is bounded by
\[
\Vert \vec{u}_N(t)\cdot \nabla \omega_N(t) \Vert_{L^p}
\le 
C a_N^{\nu(p)/2} N \Vert \omega_0 \Vert_{L^p} \Vert \omega_N \Vert_{L^p}.
\]
Refering to \ref{eq:Lpeq}, we obtain
\[
\frac{d}{dt}\left(\Vert \omega_N \Vert_{L^p} e^{-\delta_N t}\right)
\le C a_N^{\nu(p)/2} N \log(N)^2 \Vert \omega_0 \Vert_{L^p}\Vert \omega_N \Vert_{L^p}.
\]
Since $e^{\delta_N t}$ is uniformly bounded on $[0,t_N^\ast]$, we can increase the constant $C$ to find an estimate
\[
\frac{d}{dt}\left(\Vert \omega_N \Vert_{L^p} e^{-\delta_N t}\right)
\le C a_N^{\nu(p)/2} N \log(N)^2 \Vert \omega_0 \Vert_{L^p} \left(\Vert \omega_N \Vert_{L^p}e^{-\delta_N t}\right).
\]
The claimed estimate for $\Vert \omega_N \Vert_{L^p}$ now follows from Gronwall's inequality.
\end{proof}
Finally, we consider the case for $p>2$.
\begin{lem}\label{lem:shortLp2}
If $\omega_N(\cdot,0)\in L^p$, $2< p < \infty$, then there exists a sequence $c_N \to 0$ (depending only on the initial conditions), such that
\[
\Vert \omega_N(\cdot,t) \Vert_{L^p} 
\le 
\Vert \omega_N(\cdot,0) \Vert_{L^p}\left(1+c_N\right),
\]
for all $t \in [0,t_N^\ast]$.
\end{lem}

\begin{proof}
For $p>2$, the Sobolev embedding and Calderon-Zygmund inequality gives 
\[
\Vert \vec{u}_N \Vert_{L^\infty}
\le 
C \Vert \nabla \vec{u}_N \Vert_{L^p}
\le 
C_p \Vert \omega_N \Vert_{L^p}.
\]
Thus, the non-linear term can be estimated as
\[
\Vert \vec{u}_N \cdot \nabla \omega_N \Vert_{L^p}
\lesssim
\Vert \vec{u}_N \Vert_{L^\infty} \Vert\nabla \omega_N \Vert_{L^p}
\lesssim_{p}
\Vert \omega_N \Vert_{L^p} \Vert\nabla \omega_N \Vert_{L^p}
\lesssim
N \Vert \omega_N \Vert_{L^p}^2.
\]
Estimating the non-linear term in \eqref{eq:Lpeq} in this manner, we arrive at an estimate of the form
\[
\frac{d}{dt}\left(\Vert \omega_N \Vert_{L^p} e^{-\delta_N t}\right)
\le C N \left(\Vert \omega_N \Vert_{L^p} e^{-\delta_N t} \right)^2,
\]
where -- once again -- we have used that $e^{\delta_N t}$ is uniformly bounded for $t\in [0,t_N^\ast]$, to include the additional factor $e^{-\delta_N t}$ on the right-hand side. Upon integration in time, it follows that
\[
\Vert \omega_N(\cdot,t) \Vert_{L^p} e^{-\delta_N t}
\le 
\frac{\Vert \omega_N(\cdot,0) \Vert_{L^p}}
{1 - CNt \Vert \omega_N(\cdot,0) \Vert_{L^p}}.
\]
We can furthermore estimate the denominator using
\[
\Vert \omega_N(\cdot,0) \Vert_{L^p}
= \Vert K_{a_N} \ast \omega_0 \Vert_{L^p}
\le \Vert K_{a_N} \Vert_{L^1}\Vert \omega_0 \Vert_{L^p}
\le C\log(N)^2 \Vert \omega_0 \Vert_{L^p},
\]
for some constant $C$. Since also $t_N^\ast \ll 1/(N\log(N)^2)$, it then follows that 
\begin{align*}
 \Vert \omega_N(\cdot,t) \Vert_{L^p} e^{-\delta_N t}
&\le 
\Vert \omega_N(\cdot,0) \Vert_{L^p} \frac{e^{\delta_N t_N^\ast}}{1-C\Vert \omega_0 \Vert_{L^p} N \log(N)^2 t_N^\ast} \\
&\equiv \Vert \omega_N(\cdot,0) \Vert_{L^p} \left(1 + c_N\right),
\end{align*} 
where $c_N \to 0$ depends only on the initial data.
\end{proof}

\section{Convergence results}
\label{sec:convergence}

Combining Theorems \ref{thm:Lpcontrol1} and \ref{thm:Lpcontrol2} of the last two sections, we can now conclude that the $L^p$-norm of the vorticity can be uniformly controlled on compact intervals $[0,T]$.
\begin{thm}[vorticity $L^p$ control] \label{thm:Lpbound}
Let $\vec{u}_0\in L^2(\T^2;\R^2)$ be given initial data for the incompressible Euler equations with vorticity $\omega_0 \in L^p(\T^2)$, $1\le p < \infty$. Let $T>0$ be given. If the parameters for the spectral viscosity approximation ensure spectral decay, then
\[
\Vert \omega_N(\cdot,t) \Vert_{L^p}
\le 
\left(1+o(1)\right) \Vert \omega_N(\cdot,0) \Vert_{L^p} + o(1).
\]
The $o(1)$ error terms converge to $0$ as $N\to \infty$, uniformly for $t\in [0,T]$.
\end{thm}

\begin{remark}
We point out that Theorem \ref{thm:Lpbound} provides a bound on the $L^p$-norm of $\omega_N(\cdot,t)$, in terms of the $L^p$-norm of $\omega_N(\cdot,0)$, rather than $\omega_0$. This is made necessary because we include the case $p=1$, for which the Fourier projection 
\[
\P_{a_N}: \; L^1 \to L^1, \; \omega_0 \mapsto \P_{a_N}\omega_0
\]
is \emph{not} a bounded operator (while for $1<p<\infty$, it is). Instead, in the case $p=1$, a more careful approximation of the initial data needs to be made to ensure uniform boundedness in $L^1$ of the approximation sequence $\omega_N$ with initial data $\omega_0\in L^1$, i.e. we can not choose the initial data projection kernel $K_{a_N} = D_{a_N}$ as the Dirichlet kernel, in this case.
\end{remark}

\subsection{Convergence for $\omega_0\in L^p(\T^2)$, $1<p<\infty$}
\label{sec:Lpconv}

With a suitable choice of the parameters, we can now prove convergence for initial vorticity $\omega_0 \in L^p$:

\begin{thm} \label{thm:convLp}
If $\omega_0\in L^p(\T^2)$ with $1<p<\infty$, and if the approximation parameters ensure spectral decay,
then the approximants $\vec{u}_N$ obtained by solving \eqref{eq:Euler} converge -- possibly up to the extraction of a subsequence -- strongly in $C([0,T];L^2(\T^2;\R^2))$ to a limit $\vec{u}$ that is a weak solution of incompressible Euler equations. Furthermore, the $L^p$-norms of the approximants are uniformly bounded $\Vert \omega_N(\cdot,t)\Vert_{L^p} \le C$, and we have the estimate
\begin{align}
\Vert \omega(\cdot,t) \Vert_{L^p} \le \Vert \omega_0 \Vert_{L^p},
\end{align}
for almost all $t\in [0,T]$.
\end{thm}

\begin{proof}
By Theorem \ref{thm:Lpbound}, the $L^p$-norm of the vorticity satisfies a uniform bound
\[
\Vert \omega_N(\cdot,t) \Vert_{L^p}
\le 
(1+c_N) \Vert \omega_N(\cdot,0) \Vert_{L^p} + c_N,
\]
for $t\in [0,T]$, where $c_N \to 0$. Additionally, since we have $\omega_N(\cdot,0) = K_{a_N} \ast \omega_0$ it follows that
\[
\Vert \omega_N(\cdot,0)-\omega_0 \Vert_{L^p}
\to 0,
\]
if $K_{a_N} = D_{a_N}$ is the Dirichlet kernel, corresponding to Fourier projection onto the first $a_N$ modes (this relies on $1<p<\infty$), or $K_{a_N}$ is another kernel satisfying similar properties. In particular, it follows that 
\[
\Vert \omega_N(\cdot,0) \Vert_{L^p} \to 
\Vert \omega_0 \Vert_{L^p}.
\]
So we have a uniform $L^p$-bound on the vorticity, implying the existence of a convergent subsequence of $\vec{u}_N$ to a weak solution of the incompressible Euler equations by Corollary \ref{cor:Lpcompconv}. 

Next, we note that the space $L^\infty([0,T];L^p(\T^2))$ is the dual of $L^1([0,T];L^q(\T^2))$, where $1<q<\infty$ is chosen such that $\frac 1p + \frac 1q = 1$. The latter space being separable (here, we use $1<p<\infty$), we can apply the sequential version of the Banach-Alaoglu theorem to the uniformly bounded sequence $\{\omega_N\}_{N\in \N}$. Therefore, passing to a further subsequence if necessary, we can assume that $\omega_N \weakstarto \omega$ in $L^\infty([0,T];L^p(\T^2))$. It then follows that 
\[
\esssup_{t\in [0,T]} \Vert \omega(\cdot,t) \Vert_{L^p} 
\le 
\limsup_{N\to \infty} \;
\esssup_{t\in [0,T]} \Vert \omega_N(\cdot,t) \Vert_{L^p} 
\le 
\Vert \omega_0 \Vert_{L^p},
\]
from the weak-$\ast$ lower semicontinuity of the norm.
\end{proof}

\subsection{The case $p=1$ and Delort solutions}
\label{sec:Delort}

In the last section, we have established convergence results for the numerical approximation for initial data $\vec{u}_0$ with vorticity $\omega_0 = \curl(\vec{u}_0) \in L^p$, for $1<p<\infty$. This essentially amounted to proving that the vorticity of the numerical approximation $\omega_N$ remains uniformly bounded in $L^p$, as $N\to \infty$. Convergence then follows from compactness arguments, using the fact that we have a compact embedding $L^p \embedsc H^{-1}$. However, this compactness of the embedding is no longer true, in the case $p=1$. Due to the a priori $L^2$ bound on $\vec{u}_N$, we still have that $\omega_N \in H^{-1}$ is uniformly bounded. However, since $L^1$ is not reflexive, a uniform bound $\Vert \omega_N \Vert_{L^1} \le M$ does not guarantee that we can pass (in a weak sense) to a limit $\omega_N \to \omega$ with $\omega \in L^1$. Instead, the limiting object might be a (signed) measure. We will denote the space of finite, non-negative Radon measures on $\T^2$ by $\M_{+}$, in the following.

In this section, we will prove convergence for initial data $\omega_0 = \omega_0' + \omega_0'' \in H^{-1}$, where $\omega_0' \in \M_{+}$ is a non-negative measure, and $\omega_0''\in L^1$. Our proof of convergence will rely on the following fact, first (implicitely) established by Delort \cite{Delort1991}, and later explicitely pointed out by Vechhi and Wu \cite{Vecchi1993}, see also the discussion in \cite{Schochet1995}.

\begin{thm}[Delort \cite{Delort1991}, Vecchi and Wu \cite{Vecchi1993}, Shochet \cite{Schochet1995}] \label{thm:Delort}
Let $\omega_N(x,t)$ be a sequence of vorticities, satisfying the following conditions:
\begin{enumerate}[(i)]
\item $\Vert \omega_N(\cdot,t) \Vert_{H^{-1}} \le M$, uniformly for $t\in [0,T]$,
\item $\Vert \omega_N(\cdot,t) \Vert_{L^1} \le M$, uniformly for $t\in [0,T]$,
\item for all $\epsilon>0$, there exists $\delta >0$, such that 
\[
|A|< \delta \implies \int_A \omega_{N,-}(\cdot,t) \, dx < \epsilon, \quad \forall t\in [0,T], \; \forall\, N\in \mathbb{N}
\]
where $\omega_{N,-} \defeq \max(0,-\omega_N) \ge 0$ denotes the absolute value of the negative part of $\omega_N$.
\end{enumerate}
Then there exists a subsequence (not reindexed), and a measure $\omega \in \left(\M_{+} + L^1\right) \cap H^{-1}$, such that $\omega_N \weaklyto \omega$ in the sense of measures. Furthermore, for the corresponding sequence of velocities $\vec{u}_N = (u_N^1,u_N^2)$, one has $\vec{u}_N \weaklyto \vec{u}$ weakly in $L^2$, and 
\begin{gather}
\left\{
\begin{gathered}
u_N^1 u_N^2 \to u^1 u^2, \\
\left(u_N^1\right)^2 - \left(u_N^2\right)^2 \to \left(u^1\right)^2 - \left(u^2\right)^2
\end{gathered}
\right\}
\text{ in $\D'([0,T]\times \T^2)$}.
\end{gather}
In particular, under the assumptions of this Theorem, this implies that one can pass to the limit in the non-linear terms of the weak form of the incompressible Euler equations (in primitive formulation), i.e. for any divergence-free test function $\vec{\phi}\in C^\infty([0,T]\times \T^2;\R^2)$, we have
\[
\int_0^T \int_{\T^2} (\vec{u}_N\otimes \vec{u}_N):\nabla \vec{\phi} \, dx\, dt
\to
\int_0^T \int_{\T^2} (\vec{u}\otimes \vec{u}):\nabla \vec{\phi} \, dx\, dt.
\]
\end{thm}

For initial data $\vec{u}_0\in L^2$, the uniform $H^{-1}$-bound on the vorticity is easily established. The uniform $L^1$-bound on the vorticity has been established in Theorem \ref{thm:Lpbound}, provided that $\omega_N(\cdot,0)$ remains uniformly bounded in $L^1$. This is a non-trivial issue: For initial data $\omega_0 \in \M_{+} \cap H^{-1}$ (or indeed $\omega_0 \in L^1$), the direct Fourier projections $\P_N \omega_0$ may not necessarily be bounded in $L^1$, since $\Vert \P_N \Vert_{L^1 \to L^1} \sim \log(N)^2$. In this case, it is therefore necessary to be more careful in the approximation of the initial data. A discussion of one possible way to obtain suitable approximations of the initial data will now be given. 

\begin{remark} 
The uniform $L^1$-boundedness of the sequence requires an initial approximation for which the vorticity does not only converge in $H^{-1}$, but also in the sense of (signed) measures with a uniform $L^1$-bound. One way to ensure $L^1$ boundedness is as follows: Fix a mollifier $\psi\in C^\infty$ with support in a unit ball $B_1(0)$, say. Denote $\psi_\rho(x) := \rho^{-2} \psi(x/\rho)$. We obtain the initial data for the numerical approximation by convolution with a smoothing kernel $\omega_0 \mapsto \psi_\rho \ast \omega_0$, and subsequently project to the lowest Fourier modes $\le N$, viz. 
\[
\omega_0 \mapsto D_N\ast(\psi_\rho \ast \omega_0) = \left( D_N\ast \psi_\rho\right) \ast \omega_0,
\]
where $D_N(x) = \sum_{|\vec{k}|_\infty\le N} \exp(i\vec{k}\cdot \vec{x})$ is the Dirichlet kernel. Since $\psi_\rho$ is smooth, we are assured that $D_N\ast \psi_\rho \to \psi_\rho$ uniformly as $N\to \infty$ (for fixed $\rho>0$). In particular, it follows that $\Vert D_N \ast \psi_\rho \Vert_{L^1} 
\to 
\Vert \psi_\rho \Vert_{L^1}$. The idea is now to choose a sequence $\rho_N$, such that $N \gg \rho_N^{-1}$ (i.e. such that the convolution kernel is asymptotically resolved by the numerical approximation). If the convolution is adequately resolved, then we would expect that $K_N \defeq D_N \ast \psi_{\rho_N}$ is a suitable kernel to ensure convergence of the initial data.
\end{remark}

\begin{prop} \label{prop:goodkernel}
Let $\Psi\in C_c^\infty(\R^2)$ be a non-negative function, $\int_{\R^2} \Psi(x) \, dx = 1$, and assume that $\Psi$ is compactly supported in $(-\pi,\pi)^2$. Define $\Psi_\rho \defeq \rho^{-2}\Psi(x/\rho)$, a compactly supported mollfier. Let $\psi_\rho$ be the periodization of $\Psi_\rho$, such that we can consider $\psi_\rho$ as an element in $C^\infty(\T^2)$. Let $K_N \defeq D_N \ast \psi_{\rho_N}$ for some sequence $\rho_N\to 0$. If $\rho_N \sim N^{-1+\delta}$ with $\delta>0$, then $K_N$ is a good kernel, in the sense that $K_N \ast \phi \to \phi$ for all $\phi \in C^\infty(\T^2)$, and
there exists a constant $C$, such that $\Vert K_N\Vert_{L^1} \le C$. In addition, we have
\[
\Vert \psi_{\rho_N} - K_N \Vert_{L^1} \to 0, \quad \text{as }N\to \infty.
\]
\end{prop}
\begin{proof}
We can associate to $\Psi_\rho$ an element of $\psi_\rho\in C^\infty(\T^2)$, by considering the periodization
\[
\psi_\rho(x) = \sum_{n\in \Z^2} \Psi_\rho(x+2\pi n).
\]
We now recall that the Fourier coefficients of $\psi_\rho$ are given by evaluating the Fourier transform $\widehat{(\Psi_\rho)}(\xi)\in C^\infty(\R^2)$ at integer points \cite[Chap. VII, Theorem 2.4]{SteinWeiss}:
\[
\widehat{(\psi_\rho)}_{\vec{k}} = \widehat{(\Psi_\rho)}(\vec{k}) = \widehat{\Psi}(\rho \vec{k}), \quad \vec{k}\in \N^2.
\]
Next, we note that
\begin{align*}
\Vert \psi_\rho - D_N \ast \psi_\rho \Vert_{L^1}
&= 
\int_{T^2} \Bigg|\sum_{|\vec{k}|_\infty>N} \widehat{(\psi_\rho)}_{\vec{k}} e^{i\vec{k}\cdot \vec{x}} \Bigg|\, dx 
\\
&=
\int_{T^2} \Bigg|\sum_{|\vec{k}|_\infty>N} \widehat{\Psi}(\rho \vec{k})e^{i\vec{k}\cdot\vec{x}} \Bigg|\, dx \\
&\le C \left(\sum_{|\vec{k}|_\infty>N} |\hat{\Psi}(\rho \vec{k})|^2\right)^{1/2}.
\end{align*}
Since $\Psi$ is a Schwartz function, also its Fourier transform $\widehat{\Psi}$ is a Schwartz function. In particular, it follows that for any integer $m>0$ there exists exists a constant $C_m>0$, such that e.g. 
$
|\widehat{\Psi}(\vec{\xi})| \le C_m |\vec{\xi}|^{-m}
$, for all $\vec{\xi} \in \R^2$. But then 
\[
\Vert \psi_\rho - D_N \ast \psi_\rho \Vert_{L^1}
\lesssim \rho^{-m} \left( \sum_{|\vec{k}|_\infty>N} |\vec{k}|^{-2m}\right)^{1/2}
\sim \rho^{-m}N^{-(m-1)}.
\]
In particular, if $\rho_N \sim N^{-1+\delta}$ with $\delta >0$, then $\rho_N^{-m}N^{-(m-1)} \sim N^{1-\delta m}$. Choosing now $m$ sufficiently large such that $1-\delta m < 0$, it follows that
\[
\Vert \psi_{\rho_N} - K_N \Vert_{L^1}
= \Vert \psi_{\rho_N} - D_N \ast \psi_{\rho_N} \Vert_{L^1} \to 0,
\]
and, hence also
\[
\Vert K_N \Vert_{L^1}
\le \Vert \psi_{\rho_N} \Vert_{L^1} + \Vert \psi_{\rho_N} - K_N \Vert_{L^1}
= \Vert \psi \Vert_{L^1} + o(1),
\]
is uniformly bounded by some constant $C$.
\end{proof}

We make the following

\begin{defn}\label{dfn:mvdata}
We will say that the SV method has \emph{suitably approximated initial data}, if $\omega_N(\cdot,0)= K_{a_N} \ast \omega_0$ is obtained by convolution with a kernel $K_N$ as described in Proposition \ref{prop:goodkernel}.
\end{defn}

The following proposition gives us some control on the negative part $\omega_{N,-} \defeq \max(0,-\omega_N)$ of $\omega_N$, if the initial approximation is chosen as in Proposition \ref{prop:goodkernel}:

\begin{prop} \label{prop:initialapprox}
Consider initial data $\omega_0 = \omega_0' + \omega_0'' \in H^{-1}$, where $\omega_0' \in \M_{+}$ is a finite non-negative measure and $\omega_0'' \in L^1$. If $\omega_N(\cdot,0)$ is obtained as suitably approximated initial data for the SV method, then for any $\e>0$, there exists $c>0$ and $N_0 \in \N$, such that
\[
\int\limits_{\T^2} \left[\omega_N(\cdot,0) + c\right]_{-} \, dx < \e, \qquad \forall \, N\ge N_0.
\]
\end{prop} 

\begin{remark}
Note that $[\omega_N(\cdot,0)+c]_{-} \ne 0$, only on the set $\{ x \; | \; \omega_N(x,0) < -c \}$. The above proposition therefore gives us some control on the size of the negative part of the approximation $\omega_N$. The proposition will be used below to show that the negative vorticity cannot concentrate on small sets.
\end{remark}

\begin{proof}
Note that $w\mapsto \left[w\right]_{-}\defeq \max(0,-w)$ is convex, homogenous and bounded from above by $|w|$. From these properties, it follows that
\[
\left[\omega_N(\cdot,0)+c\right]_{-}
\le 
|\omega_N(\cdot,0)-\psi_{\rho_N} \ast \omega_0| 
+ \left[\psi_{\rho_N} \ast \omega_0+c\right]_{-}.
\]
Next, note that $\psi_{\rho_N}\ge 0$ and $c>0$, implies that 
\[
\left[\psi_{\rho_N} \ast \omega_0+c\right]_{-}
\explain{\le}{(\omega_0' \ge 0)}
\left[\psi_{\rho_N} \ast \omega_0''+c\right]_{-}
\explain{\le}{\text{(Jensen)}}
\psi_{\rho_N}\ast \left[\omega_0''+c\right]_{-}.
\] 
Therefore, we obtain upon spatial integration, using also that $\int_{\T^2} \psi_{\rho_N} \, dx = 1$:
\[
\int \left[\omega_N(\cdot,0) + c \right]_{-} \, dx
\le \Vert \omega_N(\cdot,0) - \psi_{\rho_N}\ast \omega_0 \Vert_{L^1} + \int_{\T^2} \left[\omega_0'' + c\right]_{-} \, dx.
\]
Since $\omega_0''\in L^1$, we can now choose $c>0$ large enough to ensure that the second term is smaller than $\e/2$. 
From the estimate
\begin{align*}
\Vert \omega_N(\cdot,0) - \psi_{\rho_N}\ast \omega_0 \Vert_{L^1}
&= 
\Vert (K_N - \psi_{\rho_N})\ast \omega_0 \Vert_{L^1}
\\
&\le 
\Vert K_N - \psi_{\rho_N}\Vert_{L^1} \Vert \omega_0 \Vert_{\M},
\end{align*}
and the fact that $\Vert K_N - \psi_{\rho_N}\Vert_{L^1} \to 0$, by Proposition \ref{prop:goodkernel}, we can find $N_0\in \N$, such that $\Vert \omega_N(\cdot,0) - \psi_{\rho_N}\ast \omega_0 \Vert_{L^1} < \e/2$. For this choice of $c>0$ and $N_0\in \N$, we then have
\[
\int \left[\omega_N(\cdot,0) + c \right]_{-} \, dx
< \e, \quad \text{for all }N\ge N_0.
\]
\end{proof}

The next goal is to show that the result of Proposition \ref{prop:initialapprox} remains true also at later times $t>0$. To this end, we first show the following improvement on the mere $L^1$-boundedness implied by Theorem \ref{thm:Lpbound}.

\begin{prop} \label{prop:convexest}
Let $\phi \in C^1$ be a convex function, such that 
\[
\left|\phi'(\omega)\right| \le D,
\]
for some constant $D$.
If there exists a constant $M$, such that 
\[
\int |\omega_N(\cdot,0)| \, dx \le M, \quad \text{for all }N \in \N,
\] 
then the numerical solutions $\omega_N(x,t)$ (computed with parameters ensuring spectral decay) satisfy, in addition
\begin{align}
\int \phi(\omega_N(\cdot,t))\, dx
\le 
\int \phi(\omega_N(\cdot,0))\, dx
+ c_N, \quad \text{for }t\in [0,T],
\end{align}
with a sequence $c_N$ converging to zero, $c_N \to 0$. Furthermore, the sequence $c_N$ depends on $\phi$ only via the constant $D$, i.e. the bound on $|\phi'|$.
\end{prop}

\begin{proof}
The proof again relies on a combination of a short-time estimate on the interval $[0,t_N^\ast]$, combined with the spectral decay estimate for $t\ge t_N^\ast$. For the short-time estimate, we multiply the evolution equation \eqref{eq:evoeqshort} by $\phi'(\omega_N)$ and integrate by parts to find, cp. equation \eqref{eq:estimateshort}:
\begin{align*}
\frac{d}{dt} \int_{\T^2} \phi(\omega_N) \, dx
&\le 
-\langle \phi'(\omega_N), \P_N (\vec{u}_N\cdot \nabla \omega_N ) \rangle \\
&\qquad
+ \langle \phi'(\omega_N), \e_N \Delta (R_{m_N}\ast \omega_N ) \rangle.
\end{align*}
The second term can be estimated (Lemma \ref{lem:err2}), by 
\[
\langle \phi'(\omega_N), \e_N \Delta (R_{m_N}\ast \omega_N ) \rangle
\le D \e_N m_N^2 \log(N)^2 \Vert \omega_N \Vert_{L^1}.
\]
By Theorem \ref{thm:Lpbound}, there exists a constant $C>0$ depending only on the initial data, such that we have a uniform bound 
\[
\Vert \omega_N(\cdot,t) \Vert_{L^1} 
\le C \Vert \omega_N(\cdot,0) \Vert_{L^1} + C
\le C(1+M),
\]
where the second inequality follows from the assumption of the current proposition. This implies that the second term can be bounded by a constant, uniformly in $N$. Estimating the non-linear term as in the proof of Lemma \ref{lem:shortLp1}, we can then find a constant $C>0$, depending on the initial data (but independent of $\phi$ and $N$), such that
\[
\frac{d}{dt} \int \phi(\omega_N) \, dx
\le CD Na_N \log(N)^2.
\]
In particular, this implies that for $0\le t \le t_N^\ast$:
\[
\int \phi(\omega_N(\cdot, t)) \, dx 
\le 
\int \phi(\omega_N(\cdot, 0)) \, dx 
+ \underbrace{CD Na_N \log(N)^2 t_N^\ast}_{\to 0 \; (\text{as }N\to \infty)}.
\]
Again, we note that the last term on the right-hand side converges to $0$, by assumption on the parameters ensuring spectral decay (check from the  Definition \ref{dfn:spectraldecay}).

To finish the proof, we observe that for $t \ge t_N^\ast$, we find from the evolution equation for $\omega_N$, equation \eqref{eq:vorterr}:
\begin{align*}
\frac{d}{dt} \int_{\T^2} \phi(\omega_N) \, dx
&\le 
\langle \phi'(\omega_N), (I-\P_N) (\vec{u}_N\cdot \nabla \omega_N ) \rangle \\
&\qquad
+ \langle \phi'(\omega_N), \e_N \Delta (R_{m_N}\ast \omega_N ) \rangle.
\end{align*}
The two terms on the right hand side, can be estimated as
\begin{align*}
\langle \phi'(\omega_N), (I-\P_N) (\vec{u}_N\cdot \nabla \omega_N ) \rangle
\le D \Vert (I-\P_N) (\vec{u}_N\cdot \nabla \omega_N ) \Vert_{L^1}.
\end{align*}
By Lemma \ref{lem:err1}, there exists a constant $C>0$ depending only on the initial data, such that $\Vert (I-\P_N) (\vec{u}_N\cdot \nabla \omega_N ) \Vert_{L^1} \le C N^{-1}$. 
It now follows that
\begin{align*}
\frac{d}{dt} \int_{\T^2} \phi(\omega_N) \, dx
&\le \underbrace{CD(N^{-1} + \e_N m_N^2 \log(N)^2)}_{\to 0 \;( \text{as } N\to \infty)},
\end{align*}
for some constant $C$, independent of $N$ and $\phi$. Integrating in time, it now follows that for $t\in [t_N^\ast,T]$:
\begin{align*}
 \int_{\T^2} \phi(\omega_N(\cdot,t)) \, dx
&\le 
 \int_{\T^2} \phi(\omega_N(\cdot,t_N^\ast)) \, dx
 +c^{(1)}_N \\
 &\le 
 \int_{\T^2} \phi(\omega_N(\cdot,0)) \, dx
 +c^{(1)}_N + c^{(2)}_N,
\end{align*}
with
\[
c^{(1)}_N \defeq CD(N^{-1} + \e_N m_N^2\log(N)^2) T \to 0, 
\quad 
(\text{as }N\to \infty).
\]
and 
\[
c^{(2)}_N \defeq CDNa_N\log(N)t_N^\ast \to 0,
 \quad 
 (\text{as }N\to \infty),
\] 
This proves the claim with $c_N \defeq c_N^{(1)} + c_N^{(2)}$.
\end{proof}

\begin{lem} \label{lem:L1negative}
If $\omega_N$ is obtained from the SV method, with suitably approximated initial data and parameters ensuring spectral decay, then for any $\epsilon>0$, there exists a $c>0$ and $N_0\in \N$, such that
\[
\int_{\T^2} \left[\omega_N(\cdot,t)+c\right]_{-} \, dx < \e, \quad \text{for all } t\in [0,T], \text{ and for all } N \ge N_0.
\]
\end{lem}

\begin{proof}
By Proposition \ref{prop:goodkernel}, there exists $c>0$ and $N_0\in \N$, such that at time $t=0$:
\[
\int_{\T^2} \left[\omega_N(\cdot,0) + c\right]_{-} \, dx < \e/2, \quad \text{for all }N\ge N_0.
\]
Next, we approximate $\phi(\xi) \defeq \left[\xi + c\right]_{-}$ by a family of smooth functions $\phi_\epsilon$, $\epsilon \to 0$ (e.g. by mollifying), with $\left|\phi_\e'\right| \le 1$. It follows from a Proposition \ref{prop:convexest}, that there exists a sequence $c_N\to 0$, such that for any $t \in [0,T]$:
\[
\int \phi_\epsilon(\omega_N(\cdot,t)) \, dx 
\le 
\int \phi_\epsilon(\omega_N(\cdot,0)) \, dx
+ 
c_N,
\]
where $c_N$ is \emph{independent} of $\e$. Therefore, passing to the limit $\epsilon \to 0$, it follows that
\[
\int \left[\omega_N(\cdot,t)+c\right]_{-} \, dx 
\le 
\int \left[\omega_N(\cdot,0)+c\right]_{-} \, dx
+ 
c_N. 
\]
By assumption on our choice of $c>0$, the first term on the right-hand side is bounded by $\e/2$ for all $N\ge N_0$. Since the second term converges to $0$, we can find a larger $N_0\in \N$ if necessary, so that also $c_N< \e/2$ for all $N\ge N_0$. For such a choice of $N_0$, we conclude that
\[
\int \left[\omega_N(\cdot,t)+c\right]_{-} \, dx 
< \e, \quad \text{ for all } t\in [0,T] \text{ and } N\ge N_0.
\]
\end{proof}

As a consequence of Lemma \ref{lem:L1negative}, we now prove that the sequence $\omega_{N,-}$ satisfies the equi-integrability property (iii) of Theorem \ref{thm:Delort}.

\begin{lem} \label{lem:equiint}
Under the assumptions of Lemma \ref{lem:L1negative}, the sequence $\omega_{N,-}$ is uniformly equi-integrable on $[0,T]$, in the following sense: For all $\epsilon >0$, there exists a $\delta>0$, such that
\begin{align}\label{eq:unifint}
|A|< \delta \implies \int_A \omega_{N,-}(\cdot,t) \, dx < \epsilon, \quad \text{for all }N, \text{ and } t\in [0,T].
\end{align}
\end{lem}

\begin{proof}
Let $\epsilon >0$. We have to find $\delta>0$, such that \eqref{eq:unifint} is satisfied. By Lemma \ref{lem:L1negative}, there exists $c>0$ and $N_0\in \N$, such that
\[
\int_{\T^2} \left[\omega_{N}(\cdot,t)+c\right]_{-} \, dx < \e/2,
\]
for all $N\ge N_0$ and $t\in [0,T]$.
We now observe that for any subset $A\subset \T^2$, we have
\[
\int_{A} \omega_{N,-}(\cdot,t) \, dx
\le 
\int_{A} \left(c +  \left[\omega_{N}(\cdot,t)+c\right]_{-}\right) \, dx
= c |A| + \int_{A} \left[\omega_{N}(\cdot,t)+c\right]_{-} \, dx.
\]
Since the second term is smaller than $\e/2$ by our choice of $c$, it now suffices to choose $\delta < \e/(2c)$, to find
\[
|A| < \delta \implies 
\int_{A} \left|\omega_{N,-}(\cdot,t)\right| \, dx
<\epsilon, \quad \text{for all } N\ge N_0, \text{ and all } t\in [0,T].
\]
On the other hand, let $M := \sup_{N<N_0} \Vert \omega_N \Vert_{L^\infty([0,T]\times \T^2)}$. Note that for $N=1,\ldots,N_0-1$, each $\omega_N$ is a smooth function on $[0,T]\times \T^2$. In particular, this implies that $M<\infty$ is finite. Choosing now $\delta < \e/M$, it follows that we also have
\[
|A|< \delta \implies \int_{A} \omega_{N,-}(\cdot,t) \, dx < \e, \quad \text{for }N=1,\ldots, N_0-1, \text{ and for } t\in [0,T].
\]
This proves the claim.
\end{proof}

\begin{thm} \label{thm:convmeasure}
Let $\omega_N$ be obtained by solving the approximate Euler equations, with parameters ensuring spectral decay, and suitably approximated initial data obtained from $\omega_0 = \omega_0' + \omega_0'' \in H^{-1}$, where $\omega_0' \in \M_{+}$ and $\omega_0'' \in L^1$.
Then the sequence $u_N$ converges weakly (up to the extraction of a subsequence) to a weak solution $u\in L^2$ of the Euler equations. Furthermore, the limiting vorticity $\omega$ is an element of $\omega \in (\M_{+}+L^1)\cap H^{-1}$, i.e. $\omega$ can be written as a sum $\omega = \omega_{+} + \omega_{-}$, where $\omega_{+}(\cdot,t)\in \M_{+}$ is a finite, non-negative measure on $\T^2$, and $\omega_{-}(\cdot,t) \in L^1(\T^2)$.
\end{thm}

\begin{proof}
By Proposition \ref{prop:L2bound}, we have $\Vert \vec{u}_N(\cdot,t)\Vert_{L^2} \le \Vert \vec{u}_0 \Vert_{L^2}$ for all $N$ and $t\in [0,T]$. Therefore there exists a subsequence $\vec{u}_N$, and $\vec{u}\in L^\infty([0,T];L^2(\T^2;\R^2))$, such that $\vec{u}_N \weaklyto u$ weakly in $L^2([0,T]\times \T^2)$. By Theorem \ref{thm:Lpbound}, the associated sequence of vorticities $\omega_N$ satisfies uniform bounds $\Vert \omega_N(\cdot,t) \Vert_{L^1} \le M$, for all $t \in [0,T]$. By Lemma \ref{lem:equiint}, we also have uniform equi-integrability. From this, it then follows that the relevant non-linear terms in the incompressible Euler equations converge in the sense of distributions, according to Delort's result (Theorem \ref{thm:Delort}). Thus, from the weak consistency of the spectral approximation (cp. Theorem \ref{thm:approxverified}), we conclude that $\vec{u}_N \weaklyto \vec{u}$ in $L^2$, and that $\vec{u}$ is a weak solution of the incompressible Euler equations.

Furthermore, since the non-negative parts $\omega_{N,+}$ are uniformly bounded in $L^1([0,T]\times \T^2)$, we can extract a subsequence of $\omega_{N,+} \, dx \, dt$, converging weakly in the sense of measures to a limiting measure $\omega_{+} \ge 0$. Since the sequence $\omega_{N,+}$ is uniformly bounded in $L^\infty([0,T];L^1(\T^2))$, there exists a constant $M$, such that for any $t_1<t_2$, $t_1,t_2\in [0,T]$:
\[
\int_{(t_1,t_2)\times \T^2} d\omega_{+}
\le \liminf_{N\to \infty} \int_{t_1}^{t_2}\int_{\T^2} \omega_{N,+} \, dx \, dt
\le M(t_2-t_1).
\]
In particular, it follows that $\omega_{+}$ is ``absolutely continuous with respect to $dt$'', in the sense that we can disintegrate $\omega_{+} = \omega_{+}(\cdot,t) \, dt$, with $\omega_{+}(\cdot,t)$ a finite, non-negative measure on $\T^2$ for $t\in [0,T]$, and for any $f\in C(\T^2)$, the mapping 
\[
t \mapsto \int_{\T^2} f(x) \, d\omega_{+}(t)
\] 
is Lebesgue-measureable. 

On the other hand, by the equi-integrability of the negative parts $\omega_{N,-}$, the Dunford-Pettis theorem \ref{thm:DunfordPettis} now implies that the sequence $\omega_{N,-}$ is weakly compact in $L^1([0,T]\times \T^2)$. Furthermore, we again have for any $t_1<t_2$, with $t_1,t_2\in [0,T]$:
\[
\int_{t_1}^{t_2} \int_{\T^2} \omega_{N,-} \, dx \, dt \le M(t_2-t_2).
\]
Passing to the limit $N\to \infty$ (employing weak compactness, $\omega_{N,-} \weaklyto \omega_{-}$ in $L^1$, and possibly after the extraction of a further subsequence), it follows that also
\[
\int_{t_1}^{t_2} \int_{\T^2} \omega_{-} \, dx \, dt \le M(t_2-t_2).
\]
Hence, we conclude that $\int_{\T^2} \omega_{-}(x,t) \, dx \le M$ for almost all $t\in [0,T]$. Since $\omega_{-}\ge 0$, this implies in particular that $\omega_{-} \in L^\infty([0,T];L^1(\T^2))$.

Using finally the uniform a priori bound
 \[
 \Vert \omega_N(\cdot,t) \Vert_{H^{-1}}
 \le 
 \Vert \vec{u}_N(\cdot,t) \Vert_{L^2}
 \le 
 \Vert \vec{u}_0 \Vert_{L^2},
 \]
 we conclude that the numerical approximation converges to a Delort-type solution with limiting vorticity $\omega(\cdot,t) = \omega_{+}(\cdot,t) + \omega_{-}(\cdot,t)\in (\M_{+} + L^1) \cap H^{-1}$.
\end{proof}

\section{Numerical experiments} \label{sec:numerical}
In this section, we will present a suite of numerical experiments to illustrate the convergence results proved in the last section. We start with a brief description of some essential details
of the implementation of the spectral viscosity method.

\subsection{Numerical implementation}
\label{sec:svimp}
We use an implementation of the spectral viscosity method \eqref{eq:Euler}, \eqref{eq:EulerFourier}, based on the the \emph{SPHINX} code, presented in  \cite{LeonardiPhD}. The non-linear term in \eqref{eq:EulerFourier} is implemented via $O(N^2\log N)$-costly fast Fourier-transforms. Aliasing is avoided by the use of a padded grid, employing the 2/3-rule \cite{LeonardiPhD}. This implies that if our computation includes all Fourier modes ranging over $|\vec{k}|_\infty \le N$, then the corresponding pseudo-spectral grid (without de-aliasing) would have grid points $\vec{x}^{\psi}_{i,j} \defeq (i,j)/N_G$, with $N_G \defeq 2N$ and $i,j\in \{0,\dots,N_G-1\}$. On the other hand, the padded grid with de-aliasing will have grid points $\vec{x}^P_{k,\ell} = (k,\ell)/(3N_G/2)$, where $k,\ell \in \{0,\dots,3N_G/2\}$. In the \emph{SPHINX} code, the spectral scheme is implemented in the primitive formulation \eqref{eq:EulerFourier}. Time-stepping is performed with an adaptive, explicit third-order Runge-Kutta scheme. We remark also that in the numerical implementation, the domain has been chosen to be a torus of unit periodicity, $T^2=[0,1]^2$, rather than $\T^2 = [0,2\pi]^2$. Clearly, the results of the previous sections remain true, up to rescaling.

For our simulations, the diffusion parameter $\e_N$ in \eqref{eq:EulerFourier} is chosen to be of the form $\e_N=\e/N_G = \e/(2N)$, where $\e$ is a fixed constant. This scaling for $\e_N$ with $N$ has been found to be sufficient to cause the required decay of the highest Fourier modes, to ensure vorticity control.

It has been suggested in \cite{Tadmor1989} (in the context of the Burgers equation), that the numerical stability of the SV method is greatly enhanced in practice, if the Fourier coefficients $\widehat{Q}_{\vec{k}}$ are smooth functions of $\vec{k}$. Therefore, for all following simulations carried out with the spectral viscosity method, we have set $\widehat{Q}_{\vec{k}}$ as a smooth cutoff function of the form
\[
\widehat{Q}_{\vec{k}} = 1-\exp\left(- \left(|\vec{k}|/k_0\right)^\alpha\right),
\]
where $k_0 = N/3$ (or $k_0 = N/8$), and $\alpha = 18$. The coefficients $\widehat{Q}_{\vec{k}}$ so obtained are depicted in Figure \ref{fig:Qk}, as a function of $|\vec{k}|/N$. We remark that for $|\vec{k}|=0.1\,N$, we have $\widehat{Q}_{\vec{k}} < 10^{-9}$, whereas for $|\vec{k}|=0.4\,N$, we find $\widehat{Q}_{\vec{k}}>1-10^{-11}$. For all practical purposes, this implies that $m_N \approx 0.1 \, N$, and that $\widehat{Q}_{\vec{k}}$ effectively changes from $0$ to $1$ over the interval $|\vec{k}| \in [m_N,4m_N]$ (rather than over the interval $[m_n,2m_N]$). As has already been noted in Remark \ref{rem:mN}, the choice of a factor $2$ is not essential for the theoretical results established in the previous sections.

\begin{figure}
	\centering
  \begin{subfigure}{0.35\textwidth}
    \includegraphics[width=.9\textwidth]{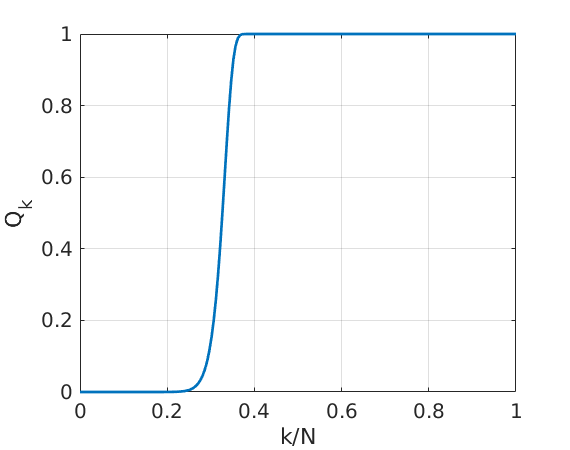}
    \caption{Coefficients $\widehat{Q}_{\vec{k}}$}
\label{fig:Qk}
  \end{subfigure}
  \begin{subfigure}{0.35\textwidth}
    \includegraphics[width=\textwidth]{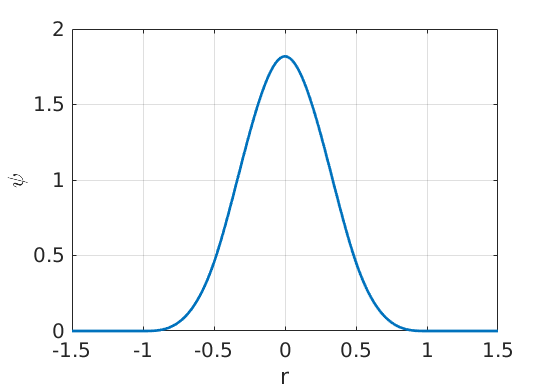}
    \caption{Mollifier $\psi$}
    \label{fig:mollifier}
  \end{subfigure}
  \caption{Coefficients defining the SV projection (left) and mollifier used in the approximation of the vortex sheet initial data (right).}
\end{figure}

\subsection{Sinusoidal vortex sheet}

In our first numerical experiment, we consider approximations to  a vortex sheet, i.e. vorticity concentrated along curves in the two-dimensional periodic domain. In particular, we take initial data
of the following form,
\[
\omega_0(x) \defeq \delta(x-\Gamma) - \int_{T^2} d\Gamma.
\]
Note that we have added a second term to ensure that $\int \omega_0\, dx = 0$. We define the curve $\Gamma$ as the graph $\Gamma \defeq \{\, (x_1,x_2) \, |\, x_1\in [0,1], \,  x_2 = d \sin(2\pi x_1)\, \}$, and we choose $d = 0.2$. We define a mollifier as the following third order B-spline 
\[
\psi(r)
\defeq 
\frac{80}{7\pi}
\left[(r+1)_+^3 - 4 (r+1/2)_+^3 + 6 r_+^3 - 4(r-1/2)_+^3 + (r-1)_+^3\right].
\]
The mollifier is depicted in Figure \ref{fig:mollifier}. We define $\psi_s(x) \defeq s^{-2} \psi(|x|/s)$. The numerical approximation to the above initial data is obtained by setting 
\[
\omega_N(x_{i,j},0) \defeq (\omega_0 \ast \psi_{\rho_N})(x_{i,j}),
\]
where $\rho_N$ determines the thickness (smoothness) of the approximate vortex sheet, and $x_{i,j}$, $i,j\in \{1,\dots, N_G\}$ denote the grid points. The convolution at a point $x\in \T^2$ is computed by numerical quadrature:
\begin{align*}
(\omega_0 \ast \psi_{\rho_N})(x)
&= \int \psi_{\rho_N}(x-y) \, d\Gamma(y) \\
&= \int_0^1 \psi_{\rho_N}(\,x-(\xi,g(\xi))\,) \sqrt{1+|g'(\xi)|^2} \, d\xi \\
&\approx
\frac{\rho_N}{M} \sum_{i=-M}^M \psi_{\rho_N}\left(\, x-(\xi_i,g(\xi_i))\, \right) \, \sqrt{1+|g'(\xi_i)|^2},
\end{align*}
with $\xi_i = x^1 + i\rho_N/M$ are equidistant quadrature points in $x^1$, and $g(\xi) = d\sin(2\pi \xi)$, $g'(\xi) = 2\pi d \cos(\xi)$. The additional factor $\sqrt{1+|g'(\xi)|^2}$ is the length element along the graph $\xi \mapsto (\xi,g(\xi))$. For our simulations, we have used $M=400$.

\subsubsection{Smoothened (fat) vortex sheet.} 
First we consider a smoothened vortex sheet, where $\rho_N$ is a fixed constant, independent of $N$. Consequently, the resulting vorticity is smooth. The initial data (on a sequence of successively finer resolutions) in shown in figure \ref{fig:fvs_init}. As seen from the figure, we have already resolved the vorticity at $512$ Fourier modes (in each direction).  Hence, this test case can serve as a benchmark for the performance of the spectral viscosity method when the initial data (and solution) is smooth.

\begin{figure}[H]
\centering
\begin{subfigure}{.32\textwidth}
\includegraphics[width=\textwidth]{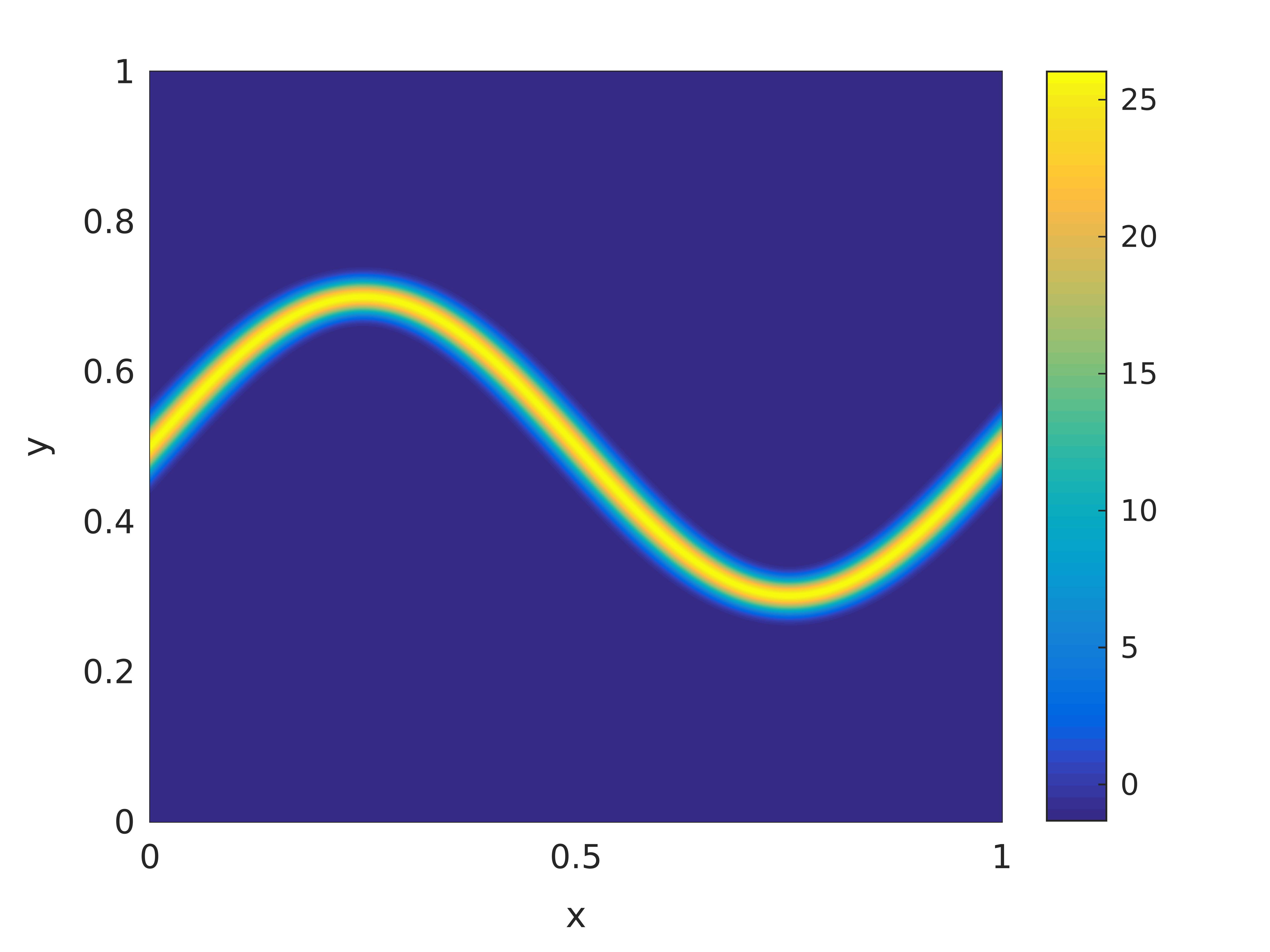}
\caption{$N_G=512$}
\end{subfigure}
\begin{subfigure}{.32\textwidth}
\includegraphics[width=\textwidth]{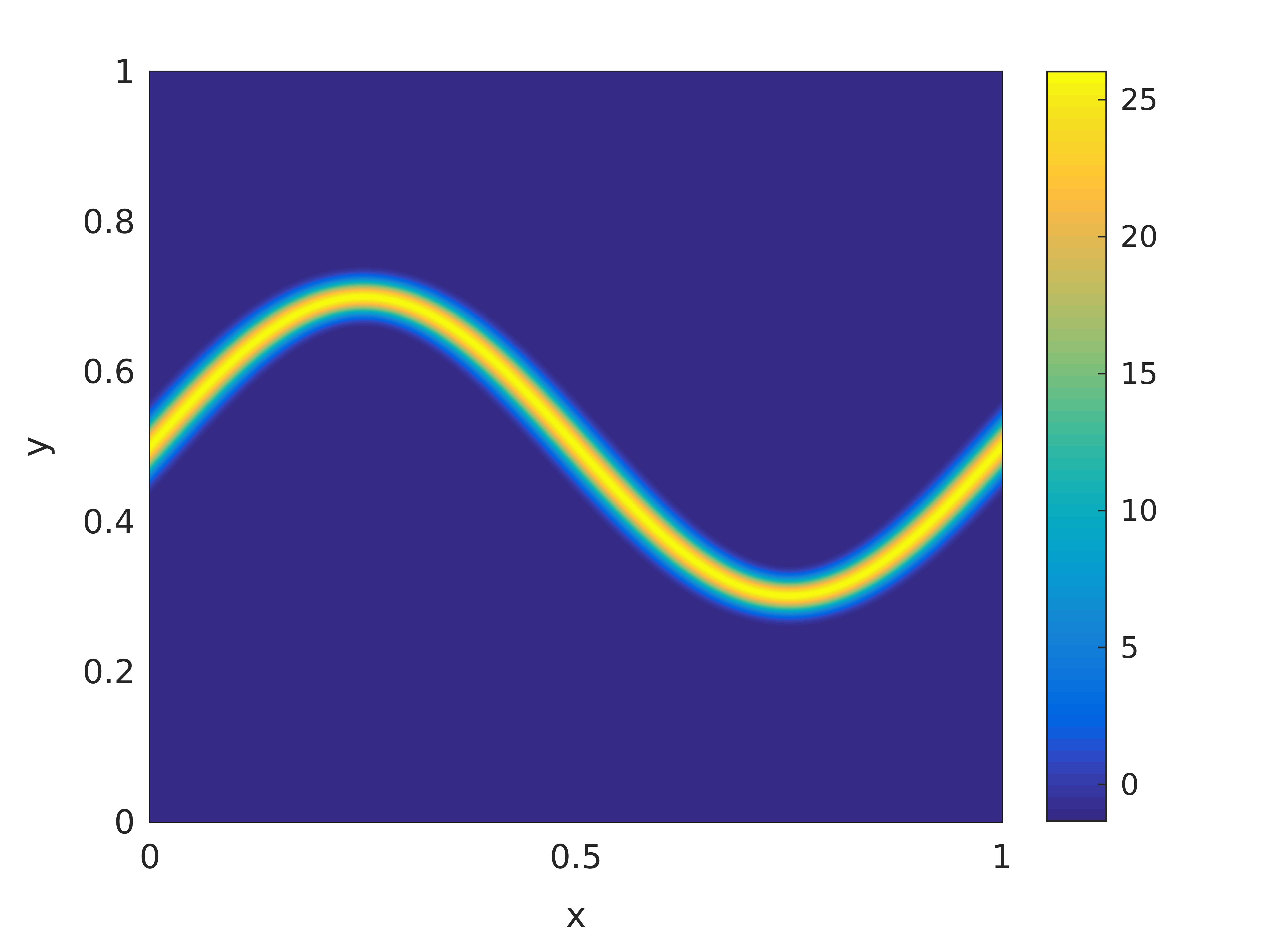}
\caption{$N_G=1024$}
\end{subfigure}
\begin{subfigure}{.32\textwidth}
\includegraphics[width=\textwidth]{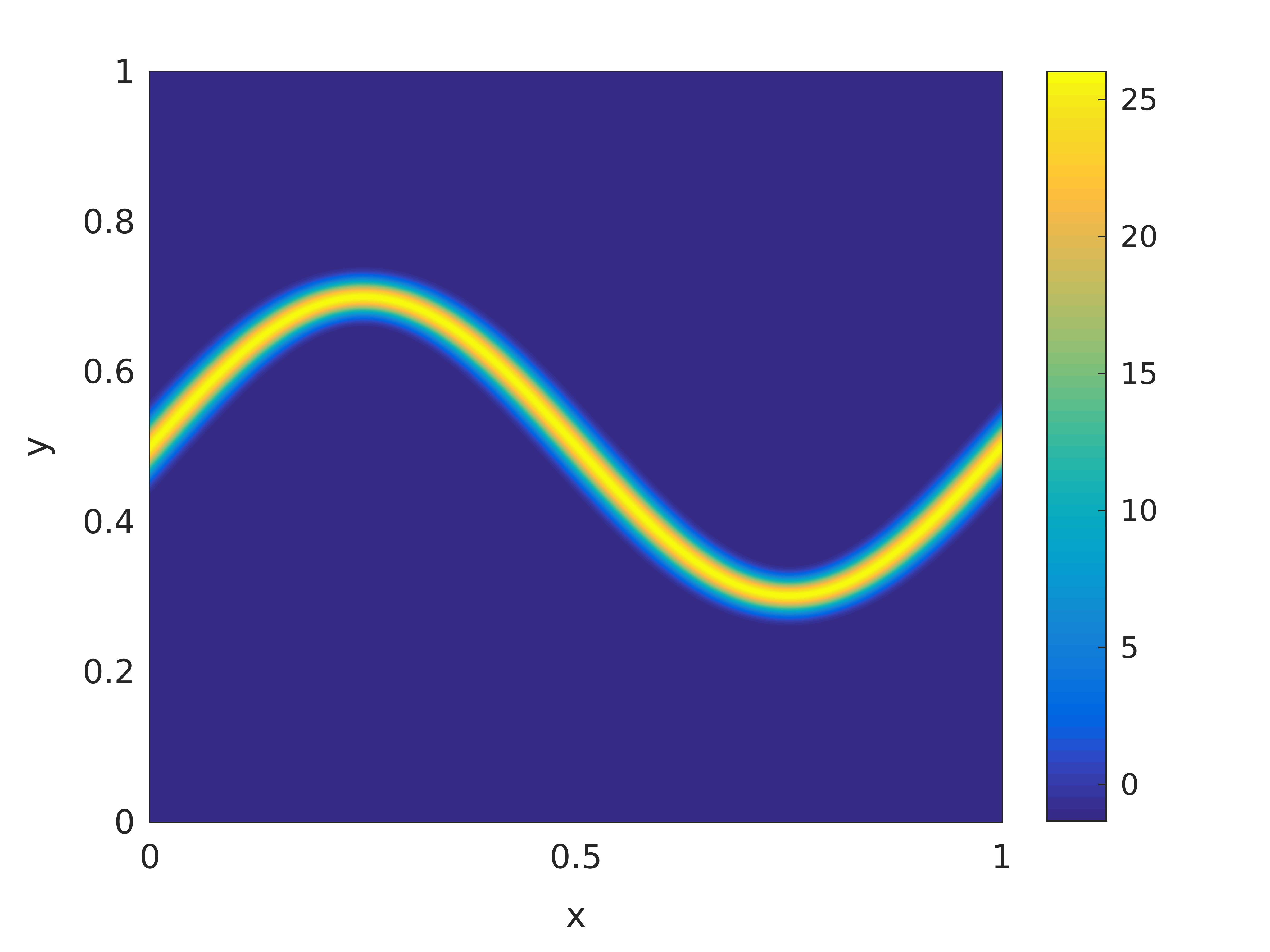}
\caption{$N_G=2048$}
\end{subfigure}
\caption{Numerical approximation of the initial data (vorticity) for the smoothened (fat) vortex sheet with $\rho_N = 0.05$, at three different spectral resolutions.}
\label{fig:fvs_init}
\end{figure}

We approximate the solution of the two-dimensional Euler equations with this initial data with two variations of the spectral viscosity method. To this end, we first consider the \emph{pure spectral method} by setting $\e = 0$ in \eqref{eq:EulerFourier}. This is justified as the initial data is smooth and the classical convergence theory (see \cite{Bardos2015}) holds for the spectral method, without any added viscosity. In figure \ref{fig:fvs_evo}, we present the evolution of this smoothened vortex sheet over time, at the highest resolution of $N_G = 2048$ Fourier modes. As seen from this figure, the initial (fat) vortex sheet has started folding by the time $t=0.4$ and has folded into two distinct vortices at time $t=0.8$. 

\begin{figure}[H]
\centering
\begin{subfigure}{.32\textwidth}
\includegraphics[width=\textwidth]{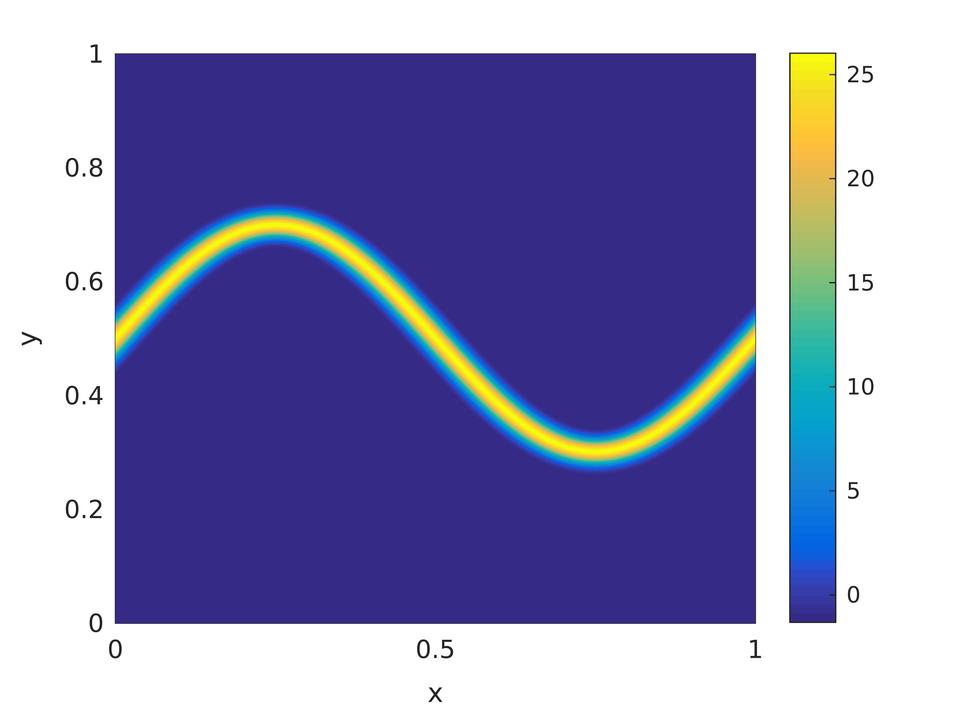}
\caption{$t=0.0$}
\end{subfigure}
\begin{subfigure}{.32\textwidth}
\includegraphics[width=\textwidth]{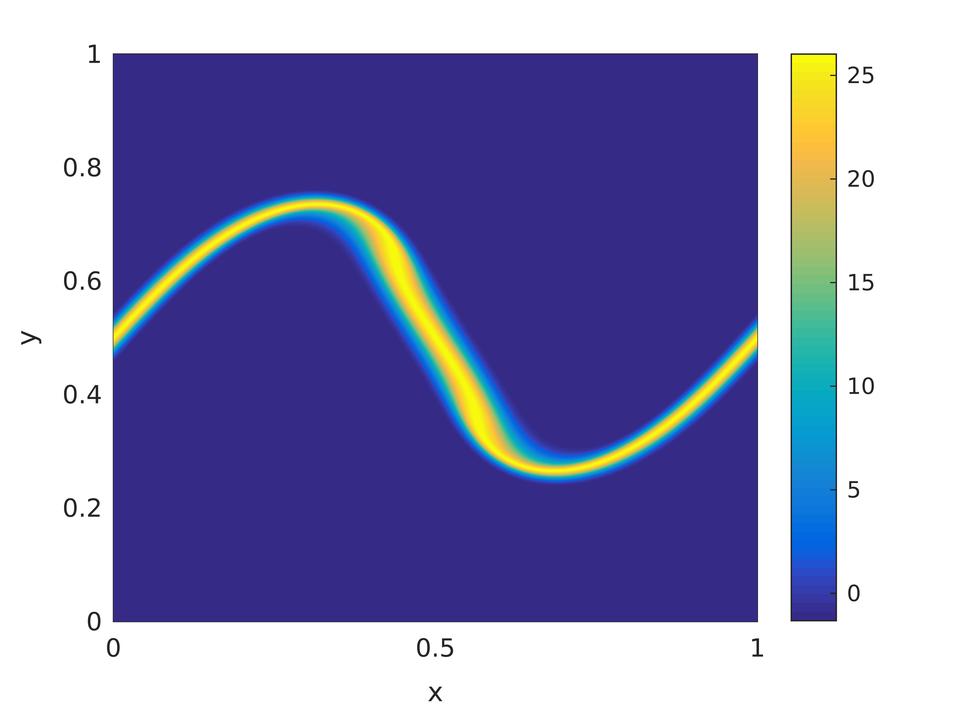}
\caption{$t=0.4$}
\end{subfigure}
\begin{subfigure}{.32\textwidth}
\includegraphics[width=\textwidth]{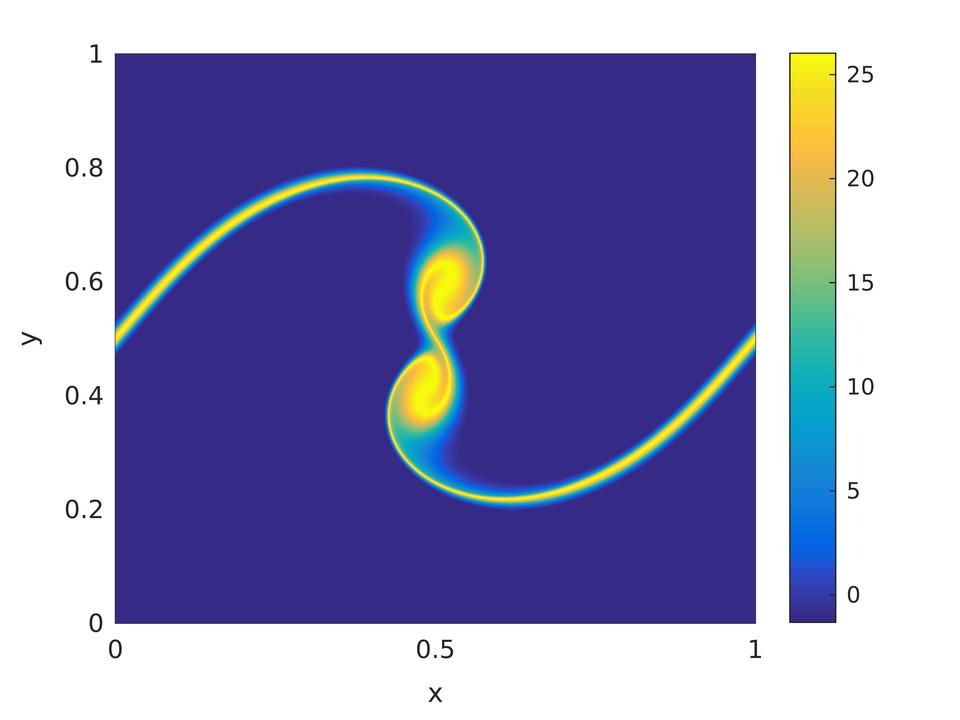}
\caption{$t=0.8$}
\end{subfigure}

\caption{Evolution in time for the smoothened vortex sheet with the pure spectral method, i.e. $(\e,\rho) = (0,0.05)$, on the highest resolution of $N_G=2048$ Fourier modes.}
\label{fig:fvs_evo}
\end{figure}

The convergence of the pure spectral method in this case is presented in figure \ref{fig:fvs_conv} where we present the approximated vorticities, at time $t=1$, on three different levels of resolution. From this figure, we observe that the pure spectral method appears to converge and the vorticity is very well resolved, already at a resolution of $N=512$ Fourier modes.  This convergence can be quantified by computing the following $L^2$-error (of the velocity field):
\begin{equation}
\label{eq:errv}
E_{N_G}(t) \defeq \Vert \vec{u}_{N_G}(\cdot,t) - \vec{u}_{N_{G,\mathrm{max}}}(\cdot,t) \Vert_{L^2},
\end{equation}
Here,  $N_{G,\mathrm{max}} = 2048$ and $\vec{u}_{N_G}$ is the velocity field computed at resolution $N_G$ (grid size). In other words, we compute error with respect to a reference solution computed on a very fine grid. This error (as a function of resolution) in plotted in figure \ref{fig:fvs_err} (A).  We observe from this figure that there is convergence with respect to increasing spectral resolution and the errors are already very low at resolutions of approximately $512$ Fourier modes. We further analyze the performance of the numerical method by computing the Fourier energy spectrum of $\omega_N$ at the highest resolution, which we define by 
\begin{equation}
\label{eq:numspec}
E(\kappa) \defeq \sum_{|\vec{k}|_\infty=\kappa} |\widehat{\omega}_{\vec{k}}|^2.
\end{equation}
The spectrum (for three different times) is shown in figure \ref{fig:fvs_err} (B) and shows that the bulk of the energy (with respect to the vorticity) is concentrated in the low Fourier modes (large scales). Moreover, this spectrum decays very fast and there is almost no contribution from the high Fourier modes. This is along expected lines as the underlying solution is smooth.

\begin{figure}[H]
\centering
\begin{subfigure}{0.32\textwidth}
\includegraphics[width=\textwidth]{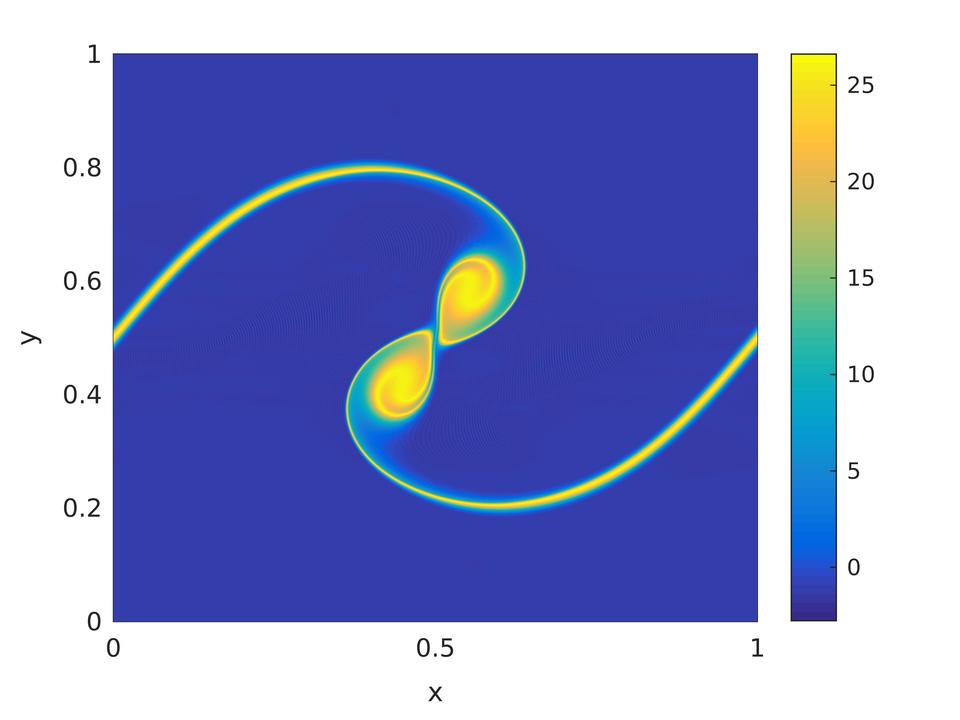}
\caption{$N_G = 512$}
\end{subfigure}
\begin{subfigure}{0.32\textwidth}
\includegraphics[width=\textwidth]{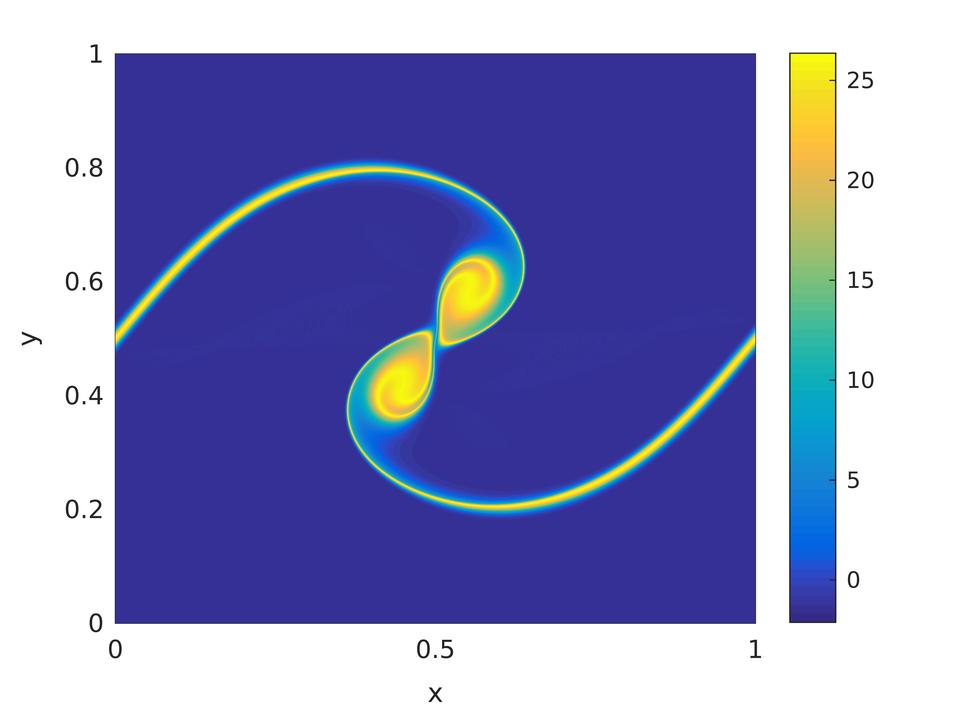}
\caption{$N_G = 1024$}
\end{subfigure}
\begin{subfigure}{0.32\textwidth}
\includegraphics[width=\textwidth]{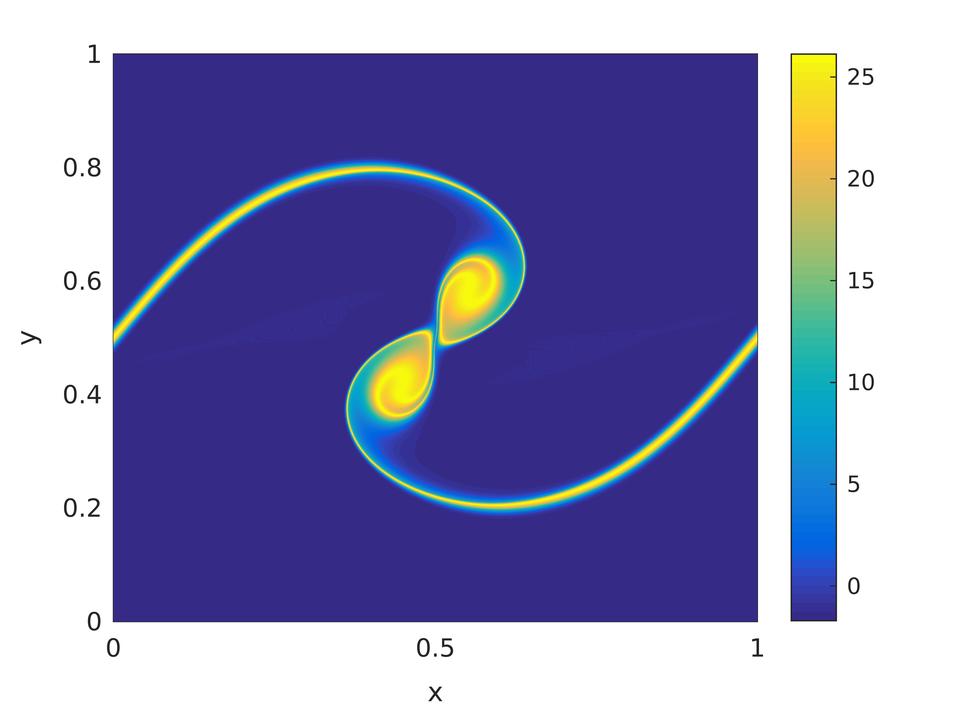}
\caption{$N_G = 2048$}
\end{subfigure}

\caption{Numerical approximations at three different spectral resolutions of the smoothened vortex sheet with the pure spectral method, i.e. $(\e,\rho) = (0,0.05)$, at time $t=1$ }
\label{fig:fvs_conv}
\end{figure}

\begin{figure}[H]
\centering
\begin{subfigure}{0.48\textwidth}
\includegraphics[width=\textwidth]{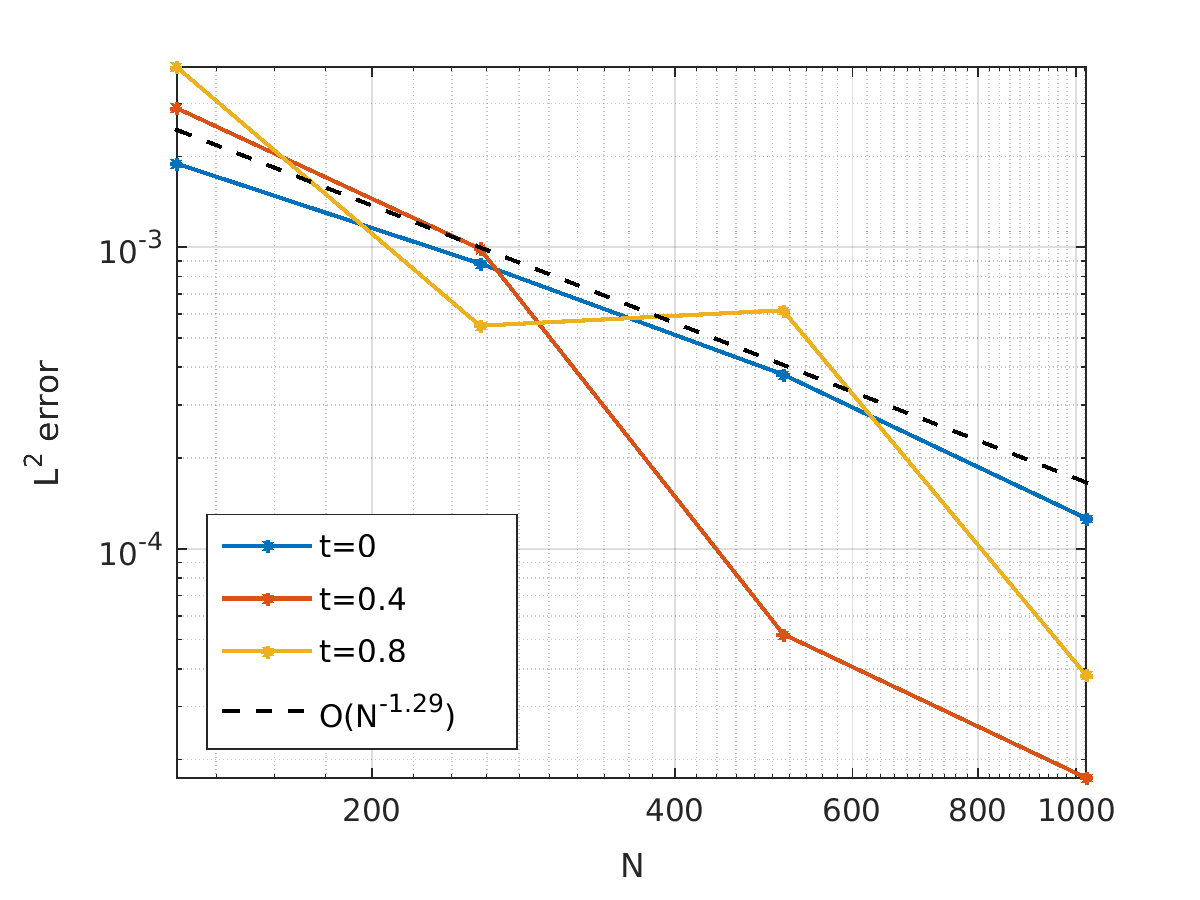}
\caption{$L^2$-error}
\end{subfigure}
\begin{subfigure}{0.48\textwidth}
\includegraphics[width=\textwidth]{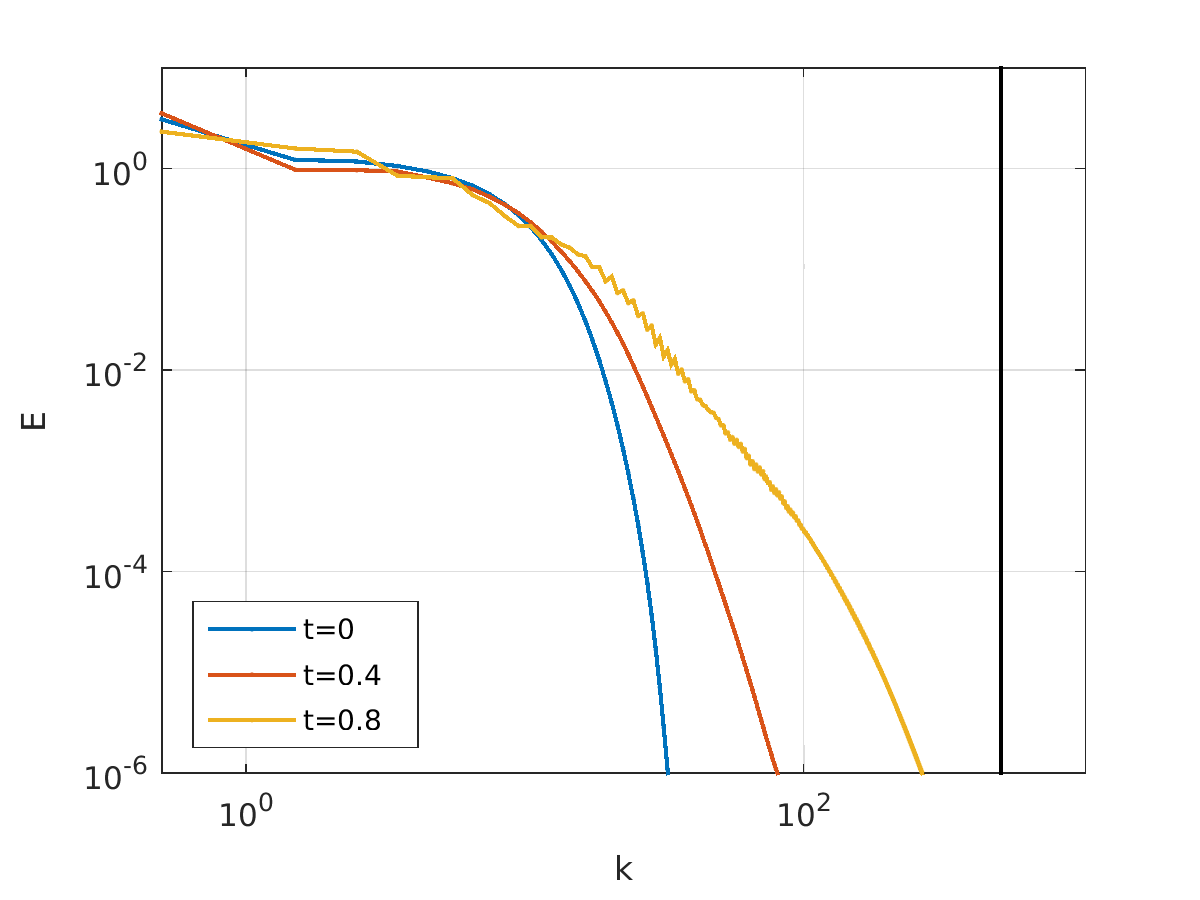}
\caption{Energy spectrum}
\end{subfigure}
\caption{Results for the smoothened vortex sheet with the  with the pure spectral method, i.e. $(\e,\rho) = (0,0.05)$ at time $t=1$. (A): Error of the approximate velocity field \eqref{eq:errv} in $L^2$ (B): Energy spectrum \eqref{eq:numspec} for the highest resolution of $N_G=2048$ at different times. }
\label{fig:fvs_err}
\end{figure}

Next, we approximate solutions of the two-dimensional Euler equations with the smoothened vortex sheet initial data, but with a spectral viscosity method, i.e. with parameters described at the beginning of this section, in particular with
$\e = 0.05$ and the cut-off parameter $k_0 = N/3$. The computed vorticities (for successively refined spectral resolutions) at time $t=1$ are shown in figure \ref{fig:fvs_conv_sv}. As seen from this figure, the computed vorticities look almost indistinguishable from the vorticities computed with the pure spectral method (compare with figure \ref{fig:fvs_conv}). This is further corroborated by the computed energy spectrum \eqref{eq:numspec}, shown in figure \ref{fig:fvs_err_sv} (B), which is also indistinguishable from the pure spectral case (figure \ref{fig:fvs_err}(B)). Moreover, we plot the $L^2$ error of the velocity \eqref{eq:errv} in figure \ref{fig:fvs_err_sv} (A) and observe that the method converges with increasing resolution. Furthermore, the convergence is cleaner than the one seen for the pure spectral method case (compare figure \ref{fig:fvs_err_sv} (A) with figure \ref{fig:fvs_err} (A)). This suggests that adding a little bit of viscosity in the higher modes (as we do with the spectral viscosity method) might improve observed convergence, even for underlying smooth solutions. 

\begin{figure}[H]
\centering
\begin{subfigure}{0.32\textwidth}
\includegraphics[width=\textwidth]{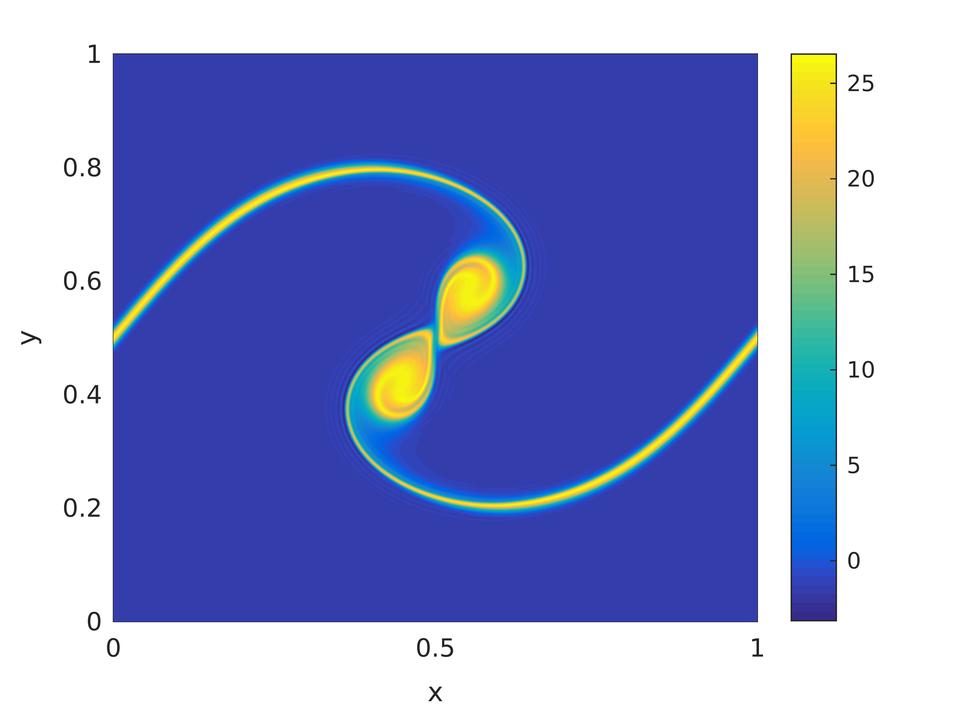}
\caption{$N_G = 512$}
\end{subfigure}
\begin{subfigure}{0.32\textwidth}
\includegraphics[width=\textwidth]{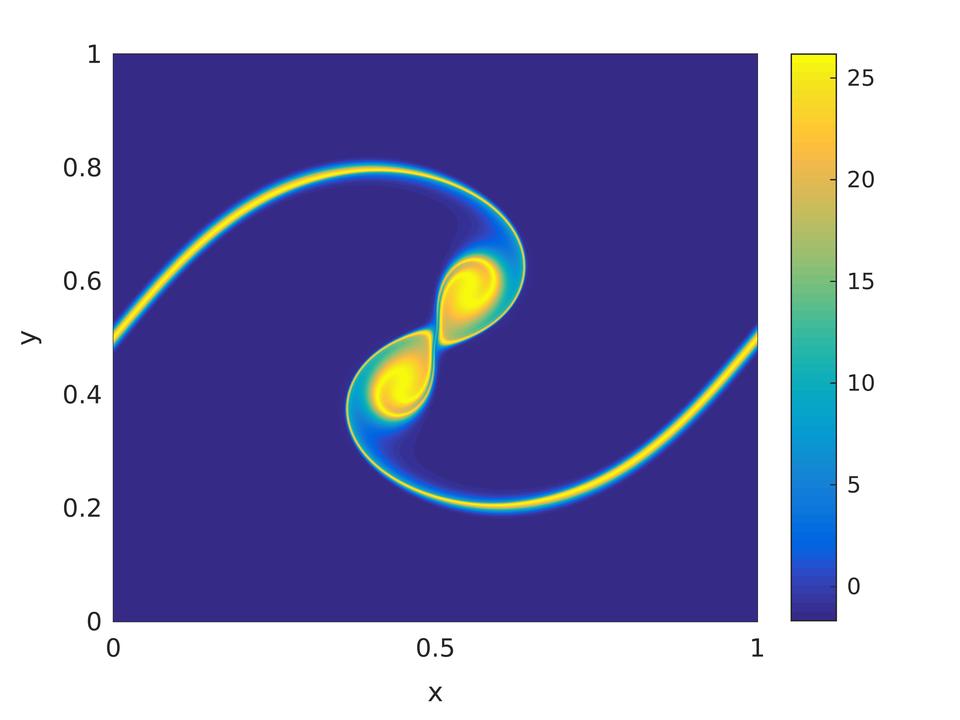}
\caption{$N_G = 1024$}
\end{subfigure}
\begin{subfigure}{0.32\textwidth}
\includegraphics[width=\textwidth]{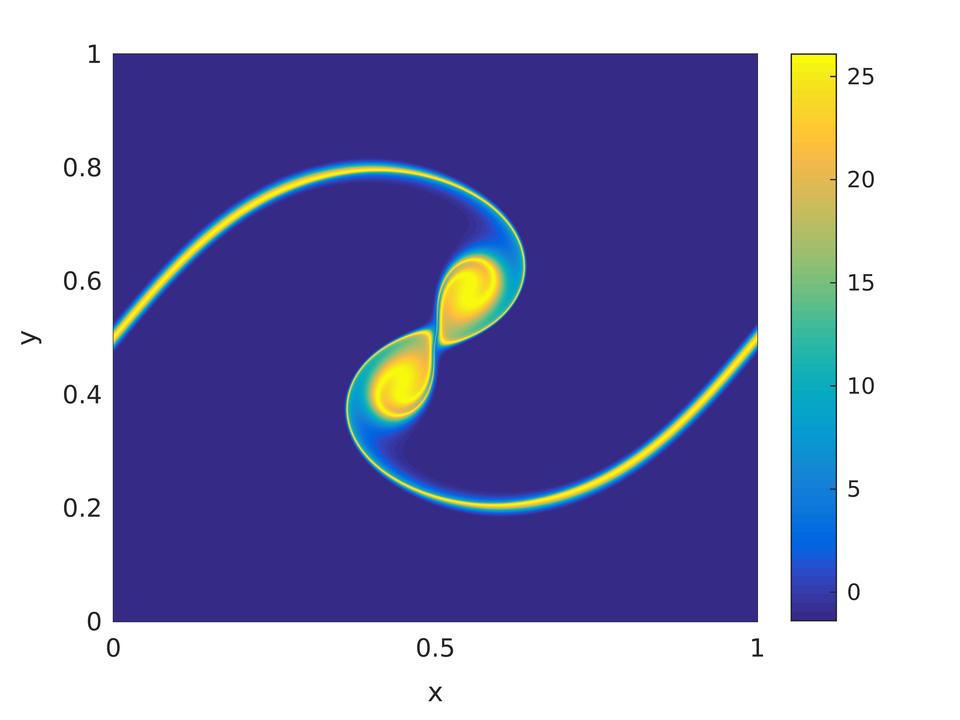}
\caption{$N_G = 2048$}
\end{subfigure}

\caption{Numerical approximations at three different spectral resolutions of the smoothened vortex sheet with the spectral viscosity method, i.e. $(\e,\rho) = (0.05,0.05)$, at time $t=1$ }
\label{fig:fvs_conv_sv}
\end{figure}

\begin{figure}[H]
\centering
\begin{subfigure}{0.48\textwidth}
\includegraphics[width=\textwidth]{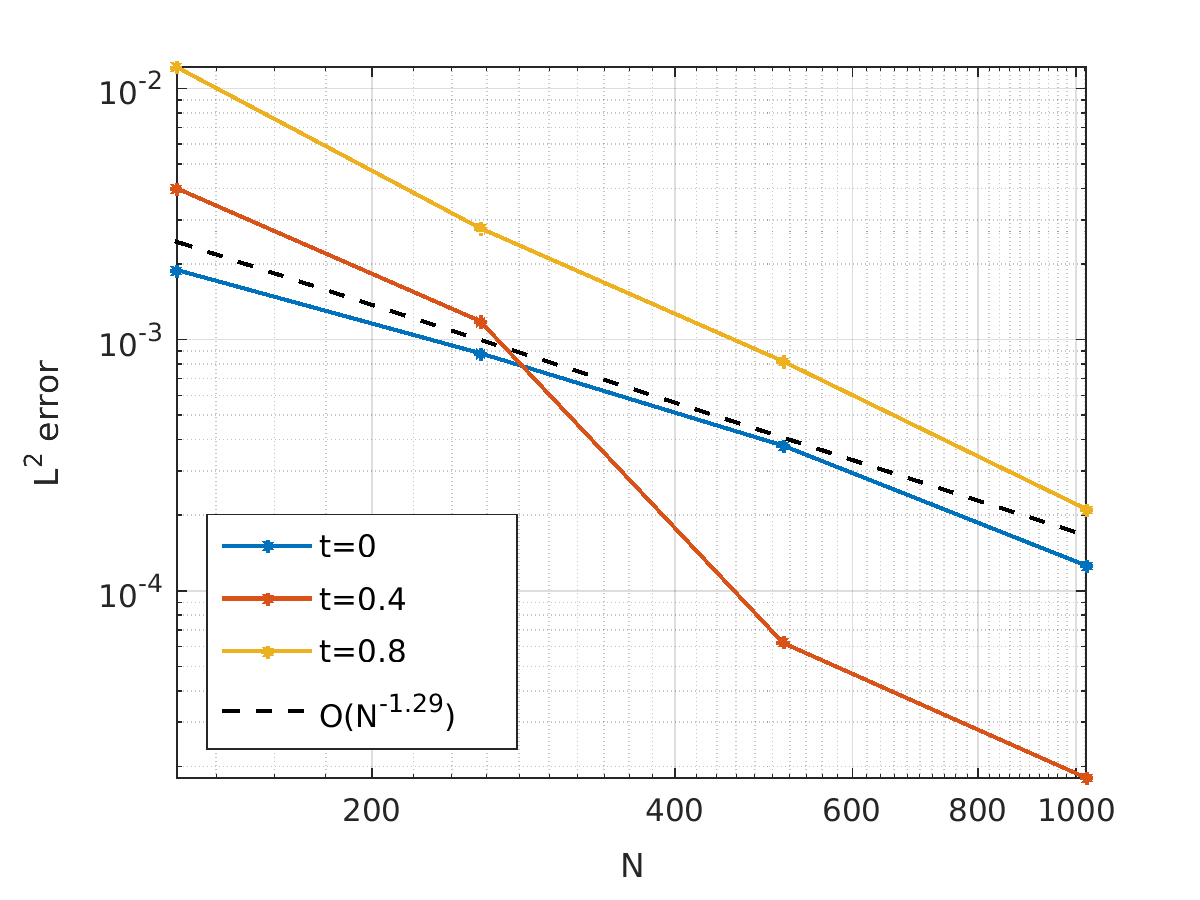}
\caption{$L^2$-error}
\end{subfigure}
\begin{subfigure}{0.48\textwidth}
\includegraphics[width=\textwidth]{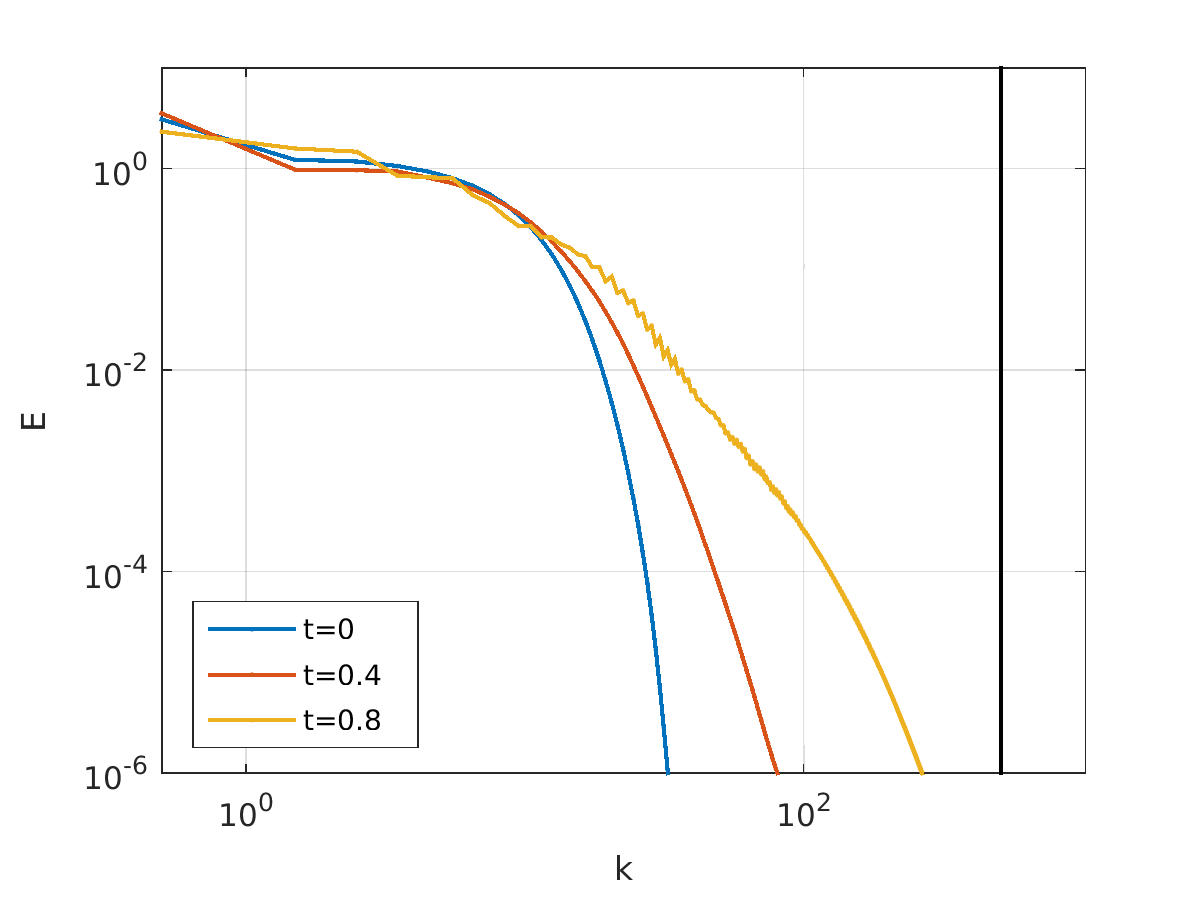}
\caption{Energy spectrum}
\end{subfigure}
\caption{Results for the smoothened vortex sheet with the  with the spectral viscosity method, i.e. $(\e,\rho) = (0.05,0.05)$ at time $t=1$. (A): Error of the approximate velocity field \eqref{eq:errv} in $L^2$ (B): Energy spectrum \eqref{eq:numspec} for the highest resolution of $N_G=2048$ at different times. }
\label{fig:fvs_err_sv}
\end{figure}
\subsubsection{Singular (thin) vortex sheet.} Next, we consider an initial data which belongs to the \emph{Delort class} by  setting  $\rho_N = \rho/N_G = \rho/(2N)$, where $\rho$ is a fixed constant. In particular, this implies that the vortex sheet becomes \emph{thinner} with increasing resolution, in contrast to the case of the smoothened (fat) vortex sheet (figure \ref{fig:fvs_init}). This can also be observed from figure \ref{fig:vortexinit}, where we depict the initial data, for successively increasing resolutions and $\rho = 10$. Moreover, this initial data is \emph{well approximated}, as stipulated by the theory presented in the last section. 
\begin{figure}[H]
\centering
\begin{subfigure}{.32\textwidth}
\includegraphics[width=\textwidth]{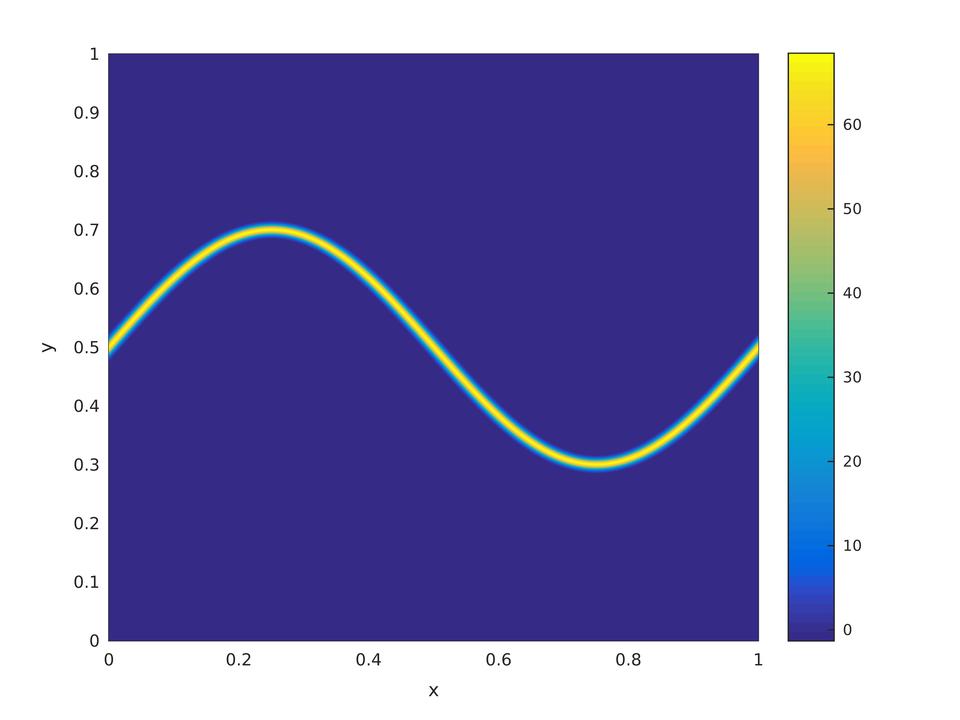}
\caption{$N_G=512$}
\end{subfigure}
\begin{subfigure}{.32\textwidth}
\includegraphics[width=\textwidth]{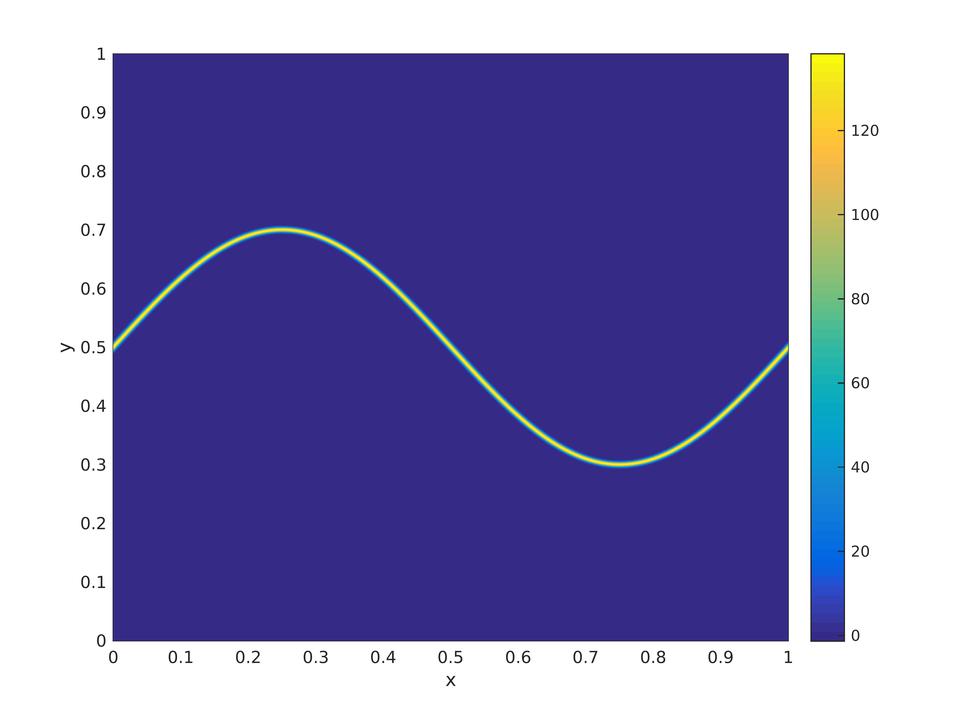}
\caption{$N_G=1024$}
\end{subfigure}
\begin{subfigure}{.32\textwidth}
\includegraphics[width=\textwidth]{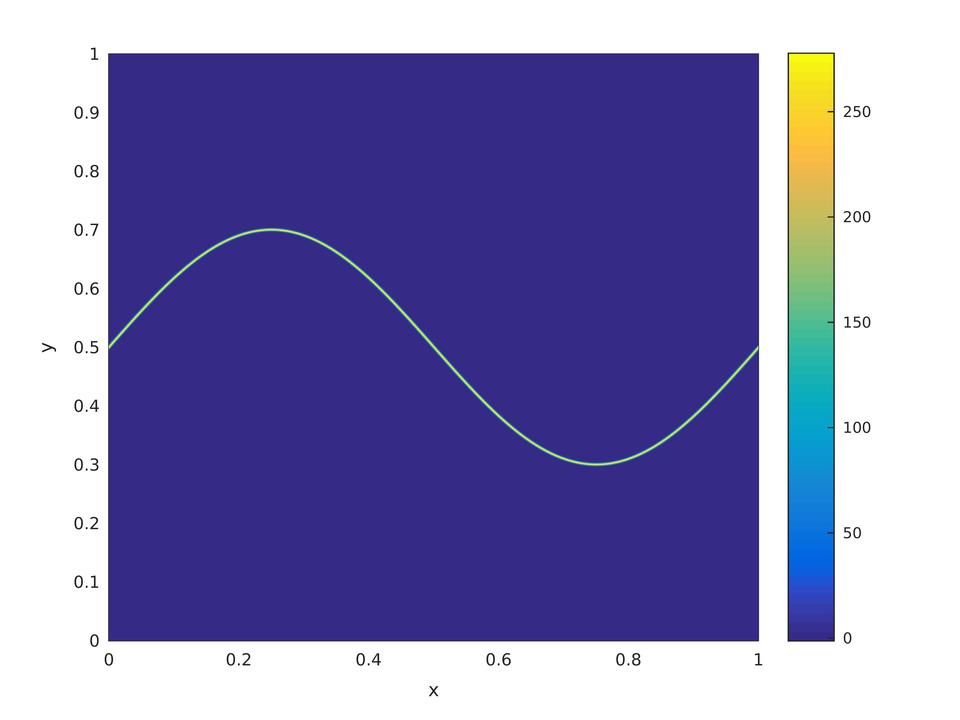}
\caption{$N_G=2048$}
\end{subfigure}
\caption{Numerical approximation of the initial data (vorticity) for the singular vortex sheet with $\rho_N = 10/N$, at three different spectral resolutions. Compare with the smoothened vortex sheet of figure \ref{fig:fvs_init}.}
\label{fig:vortexinit}
\end{figure}

It is clear that a pure spectral method will not suffice in this case. In fact, our numerical experiments showed that the pure spectral method was unstable. Hence, we have to use the spectral viscosity method to approximate the solutions in this case. At the first instance, we consider a spectral viscosity method with the parameters, $\theta = 0$ in \eqref{eq:choiceL1} and $\e = 0.05$. We remark that this particular case of the spectral viscosity method, corresponds to a \emph{vanishing viscosity} method as a Navier-Stokes type viscous damping is applied to every (even low) Fourier modes, i.e. $m_N= 0$ in \eqref{eq:EulerFourier}. Consequently, this method will only be (formally) first-order accurate. On the other hand, it can be expected to more stable than just applying viscous damping to the high Fourier modes. The evolution of the approximate vortex sheet in time, at the highest resolution of $N_G = 2048$ Fourier modes is shown in figure \ref{fig:tvs_evo}.  We observe from this figure that as in the case of the smoothened vortex sheet, the initial vortex sheet rolls up and spirals around two vorticies, but with structures that are considerably thinner than in the case of the smoothened vortex sheet (compare with figure \ref{fig:fvs_evo}). 

\begin{figure}[H]
\centering
\begin{subfigure}{.32\textwidth}
\includegraphics[width=\textwidth]{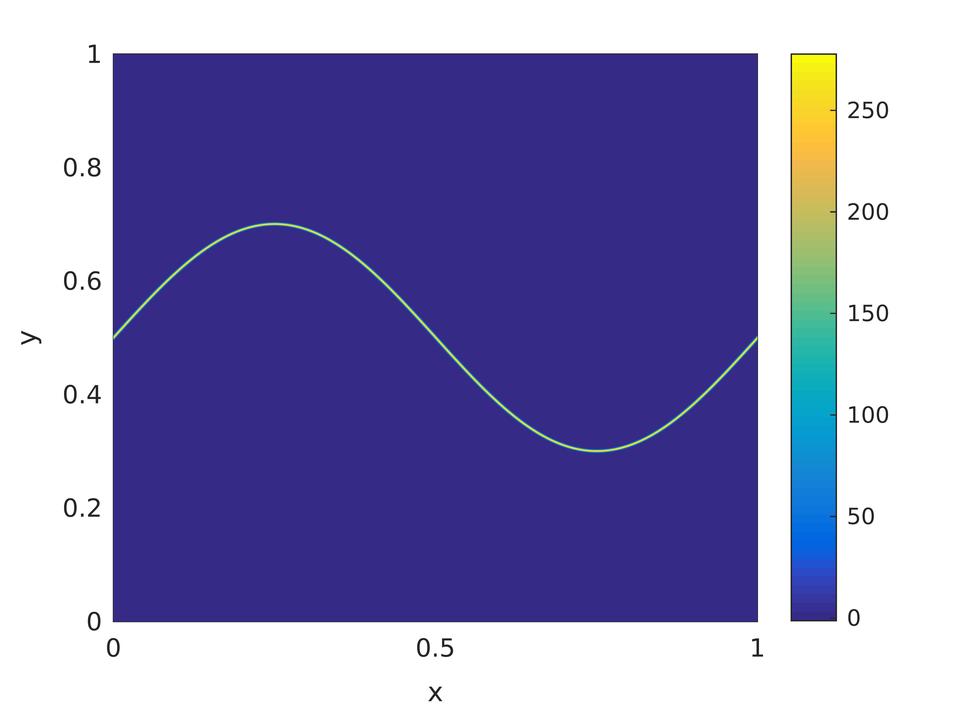}
\caption{$t=0.0$}
\end{subfigure}
\begin{subfigure}{.32\textwidth}
\includegraphics[width=\textwidth]{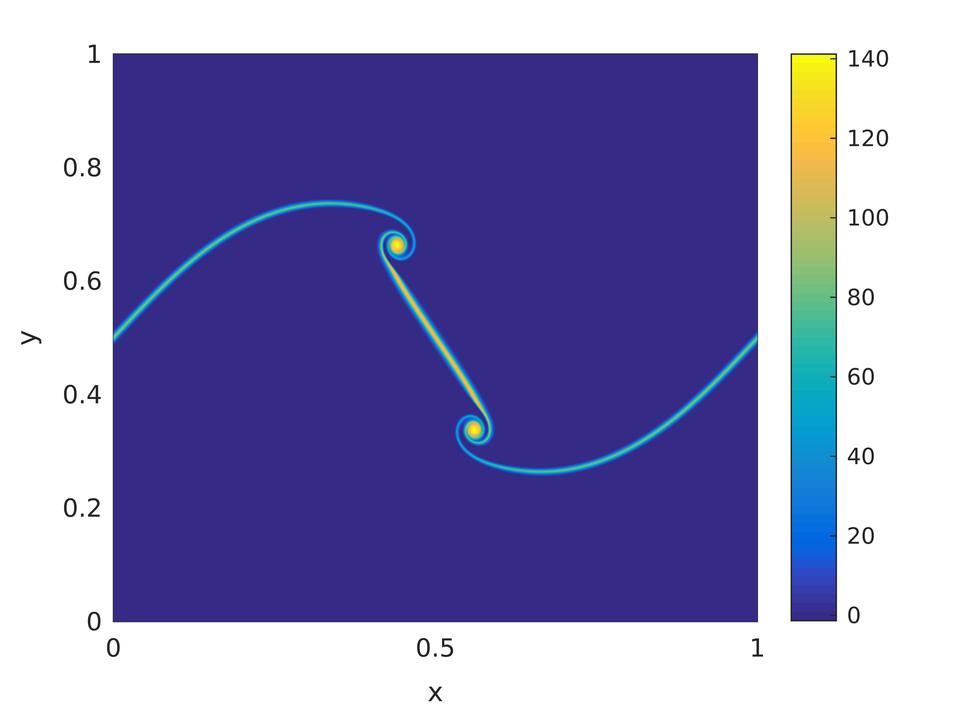}
\caption{$t=0.4$}
\end{subfigure}
\begin{subfigure}{.32\textwidth}
\includegraphics[width=\textwidth]{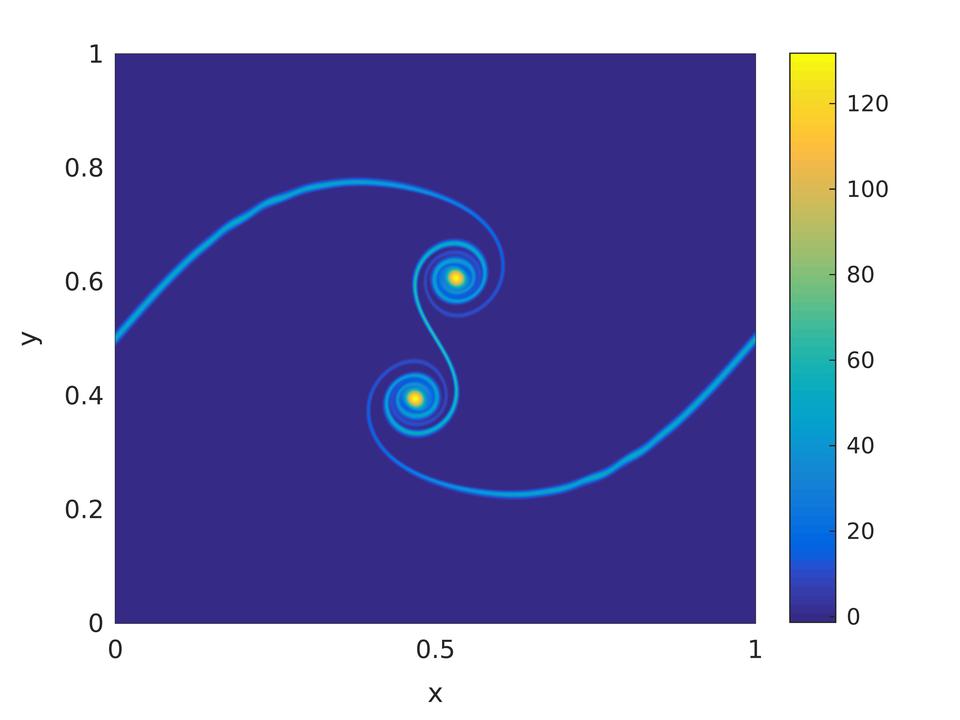}
\caption{$t=0.8$}
\end{subfigure}
\caption{Evolution in time for the singular (thin) vortex sheet with the vanishing viscosity method, i.e. $(\e,\rho) = (0.05,10)$, on the highest resolution of $N_G=2048$ Fourier modes.}
\label{fig:tvs_evo}
\end{figure}

The convergence of the numerical method is investigated qualitatively in figure \ref{fig:tvs_conv}. where we plot the computed vorticities at time $t=1$, at three successively finer resolutions and observe convergence as the resolution is increased. However, we do notice that by time $t=1$, there are small wave like instabilities that are developing along both spiral arms of the rolled up sheet. Nevertheless, these structures do not seem to impede convergence in $L^2$ norm, which is depicted in figure \ref{fig:tvs_err} (A).  We also plot the computed spectrum \eqref{eq:numspec} in figure \ref{fig:tvs_err} (B). We see from this figure that the spectrum, even for the initial data, decays much more slowly with wave number, when compared to the smoothened vortex sheet (figure \ref{fig:fvs_err} (B)). Nevertheless, there seems to be enough dissipation in the system to damp the spectrum at high wave numbers and enable a stable computation of the vortex sheet.

\begin{figure}[H]
\centering
\begin{subfigure}{0.32\textwidth}
\includegraphics[width=\textwidth]{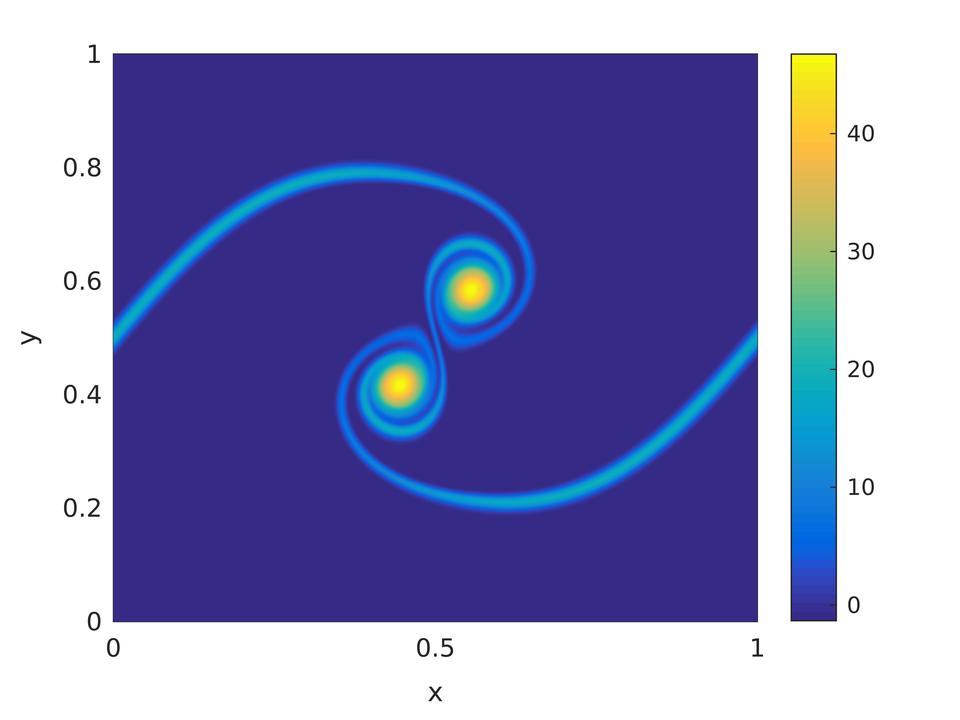}
\caption{$N_G = 512$}
\end{subfigure}
\begin{subfigure}{0.32\textwidth}
\includegraphics[width=\textwidth]{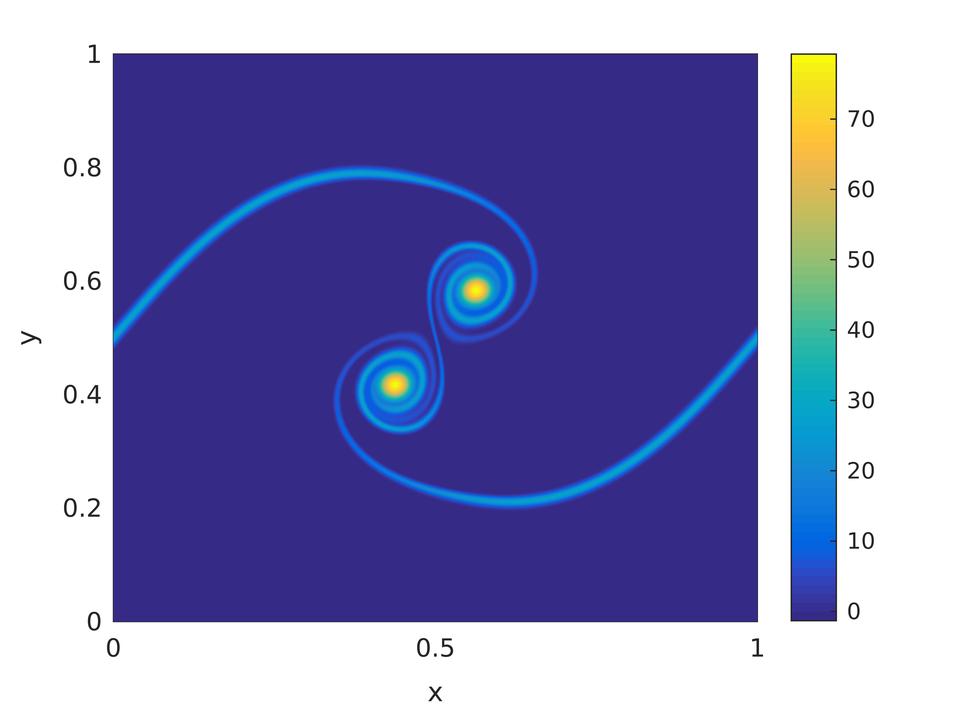}
\caption{$N_G = 1024$}
\end{subfigure}
\begin{subfigure}{0.32\textwidth}
\includegraphics[width=\textwidth]{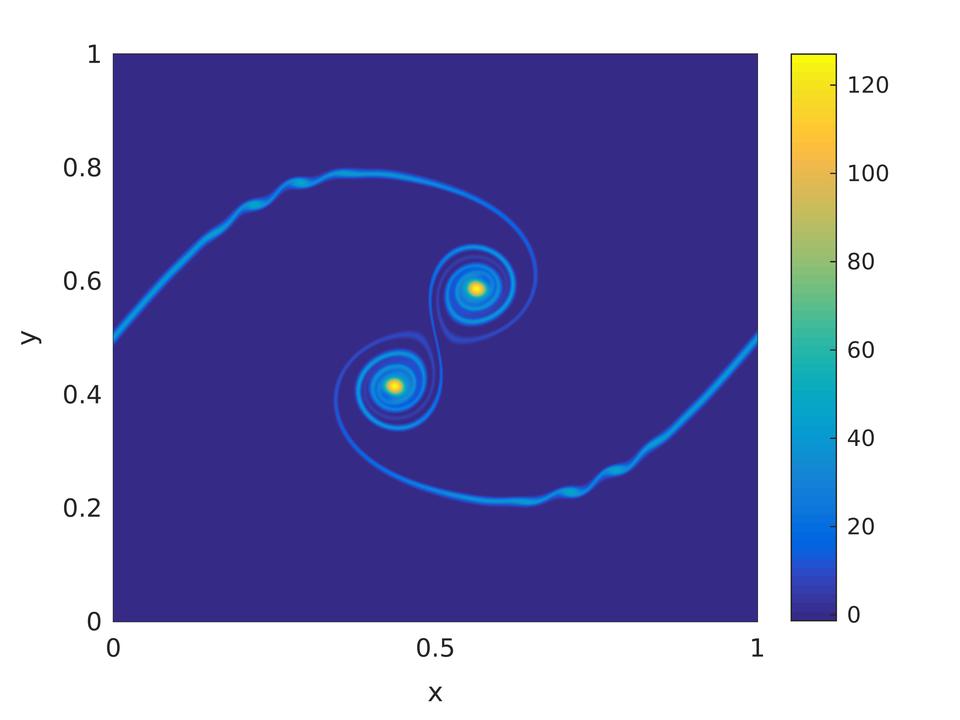}
\caption{$N_G = 2048$}
\end{subfigure}
\caption{Numerical approximations at three different spectral resolutions of the singular vortex sheet with the vanishing viscosity method with $(\e,\rho) = (0.05,10)$, at time $t=1$ }
\label{fig:tvs_conv}
\end{figure}

\begin{figure}[H]
\centering
\begin{subfigure}{0.48\textwidth}
\includegraphics[width=\textwidth]{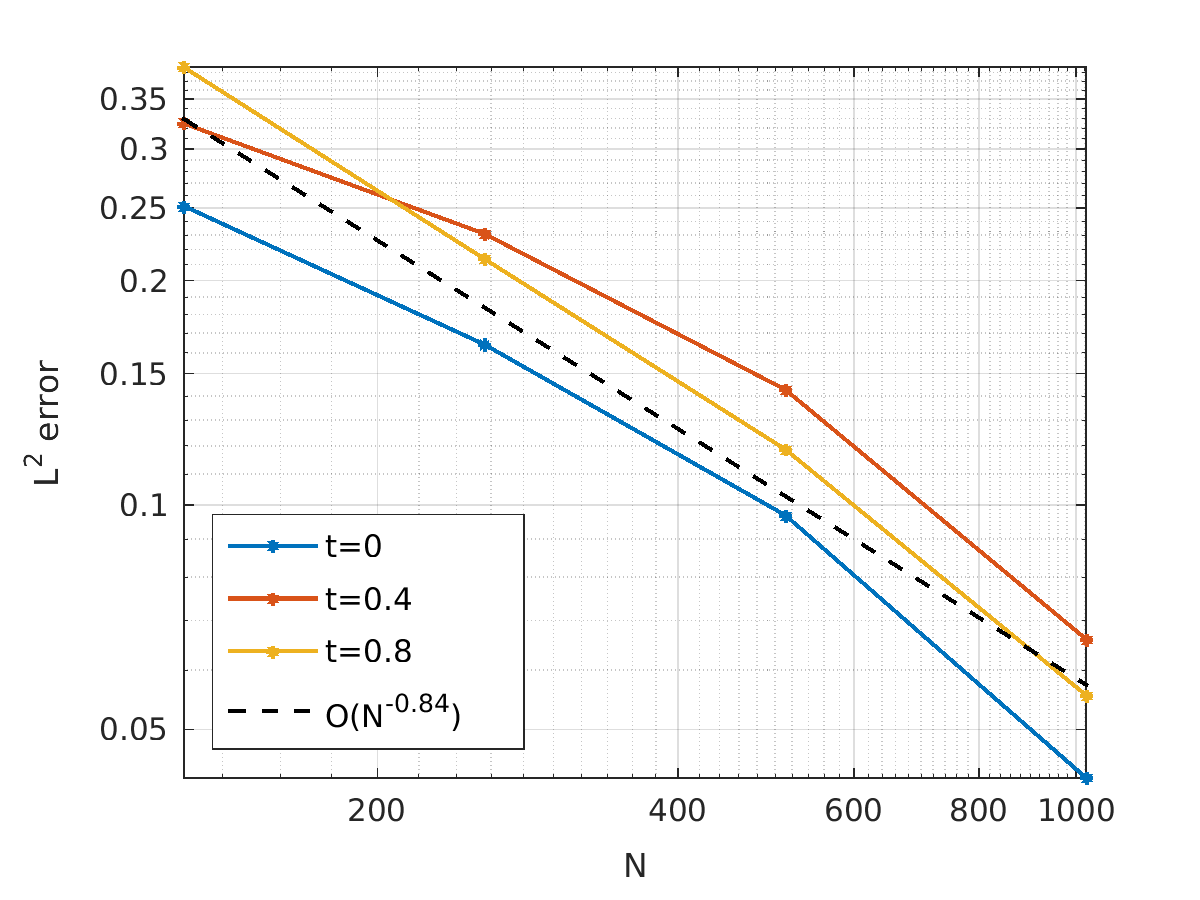}
\caption{$L^2$-error}
\end{subfigure}
\begin{subfigure}{0.48\textwidth}
\includegraphics[width=\textwidth]{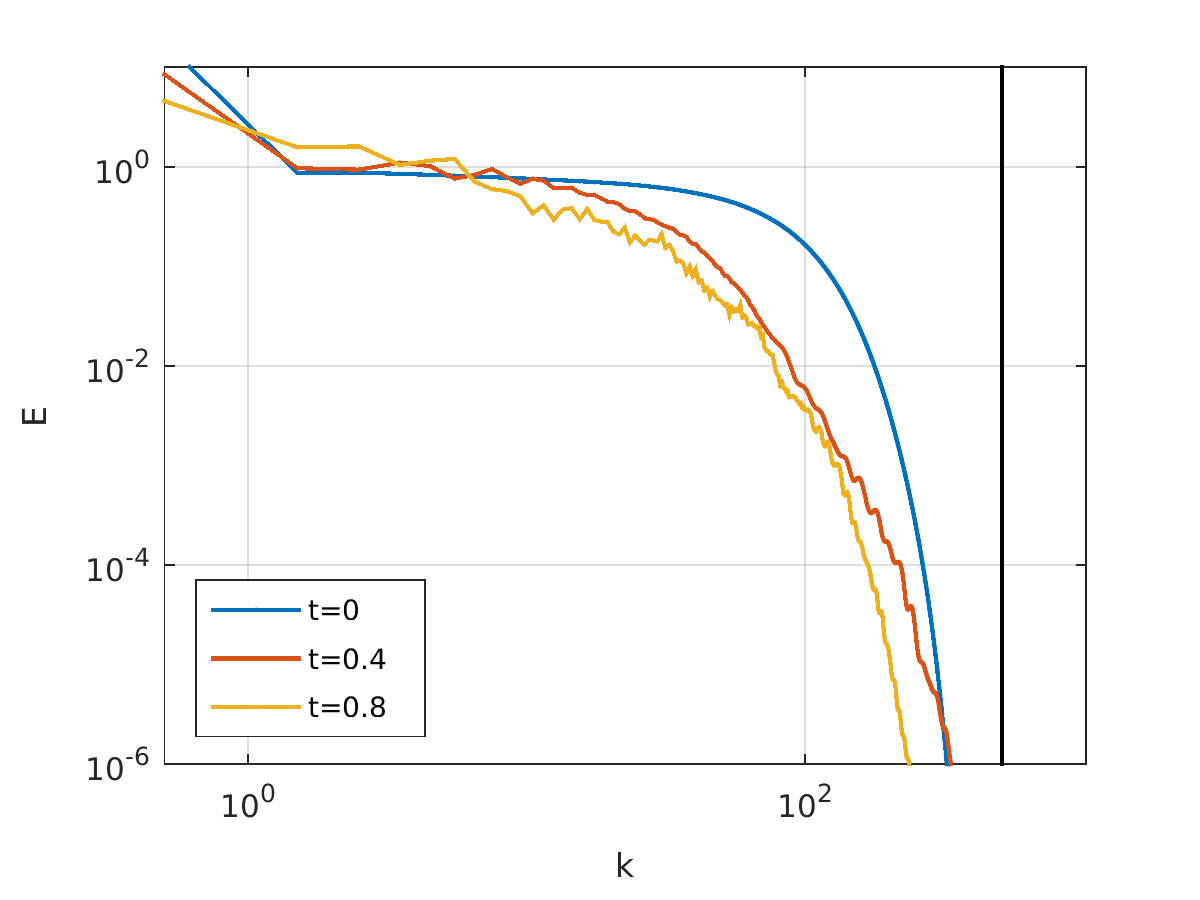}
\caption{Energy spectrum}
\end{subfigure}
\caption{Results for the singular (thin) vortex sheet with the  with the vanishing viscosity method, i.e. $(\e,\rho) = (0.05,10)$ at time $t=1$. (A): Error of the approximate velocity field \eqref{eq:errv} in $L^2$ (B): Energy spectrum \eqref{eq:numspec} for the highest resolution of $N_G=2048$ at different times. }
\label{fig:tvs_err}
\end{figure}

Next, we approximate the singular vortex sheet with a spectral viscosity method, as described in section \ref{sec:svimp}. As for the smoothened vortex sheet, we consider a cut-off parameter $k_0 = \frac{N}{3}$ and viscosity parameter $\e = 0.05$. The time evolution of the computed vorticity with this scheme is shown in figure \ref{fig:tvs_sv_evo}.  In contrast to the situation for the vanishing viscosity method (figure \ref{fig:tvs_evo}), there is a marked appearance of instabilities in the form of small wave like structures along the spiral arms by time $t=0.4$. By a later time of $t=0.8$, these structures evolve into a large number of small vortices and the whole sheet breaks up into small scale structures. The spontaneous emergence of these small scale numerical instabilities clearly impedes convergence of this version of the spectral viscosity method. This lack of convergence is seen from figure \ref{fig:tvs_sv_conv} where plot the approximate vorticities, computed with this spectral viscosity method at time $t=1$, at three successively finer mesh resolutions. From this figure, we observe that although the computed vortex sheet is stable at a moderate resolution of $512$ Fourier modes, it starts becoming unstable at the next level of refinement, i.e. $N=1024$ fourier modes, with the appearance of small vortices along the outer spiral arms. These vortices appear to break up into even smaller structures at the finest level of refinement, i.e. $N=2048$ and the whole sheet disintegrates into a soup of small incoherent vortices. The lack of convergence (at least at later times) is also observed from figure \ref{fig:tvs_sv_err} (A) where we plot the $L^2$ error \eqref{eq:errv},with respect to the velocity field at the finest resolution. Clearly, there is no observed convergence at the time $t=0.8$. The appearance of structures at small scales can also be inferred from the spectrum \eqref{eq:numspec}, plotted in figure \ref{fig:tvs_sv_err} (B). In comparison to the spectrum computed with the vanishing viscosity method (figure \ref{fig:tvs_err} (B)), we observe that the spectrum with this spectral viscosity method shows that a non-negligible amount of energy is contained in the small scales (high wave numbers).  
\begin{figure}[H]
\centering
\begin{subfigure}{.32\textwidth}
\includegraphics[width=\textwidth]{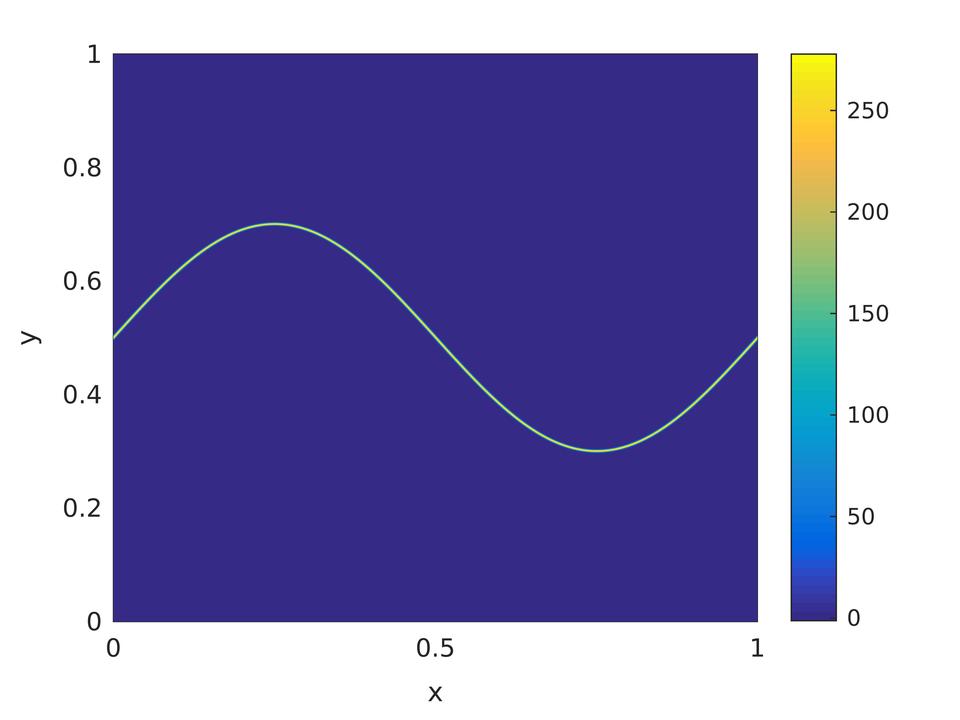}
\caption{$t=0.0$}
\end{subfigure}
\begin{subfigure}{.32\textwidth}
\includegraphics[width=\textwidth]{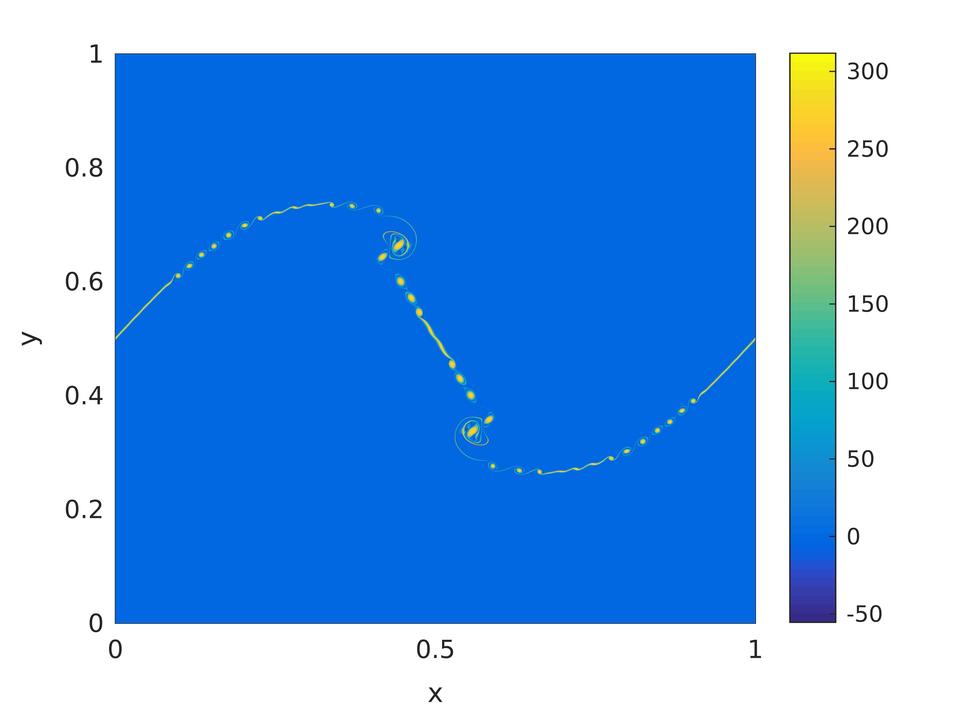}
\caption{$t=0.4$}
\end{subfigure}
\begin{subfigure}{.32\textwidth}
\includegraphics[width=\textwidth]{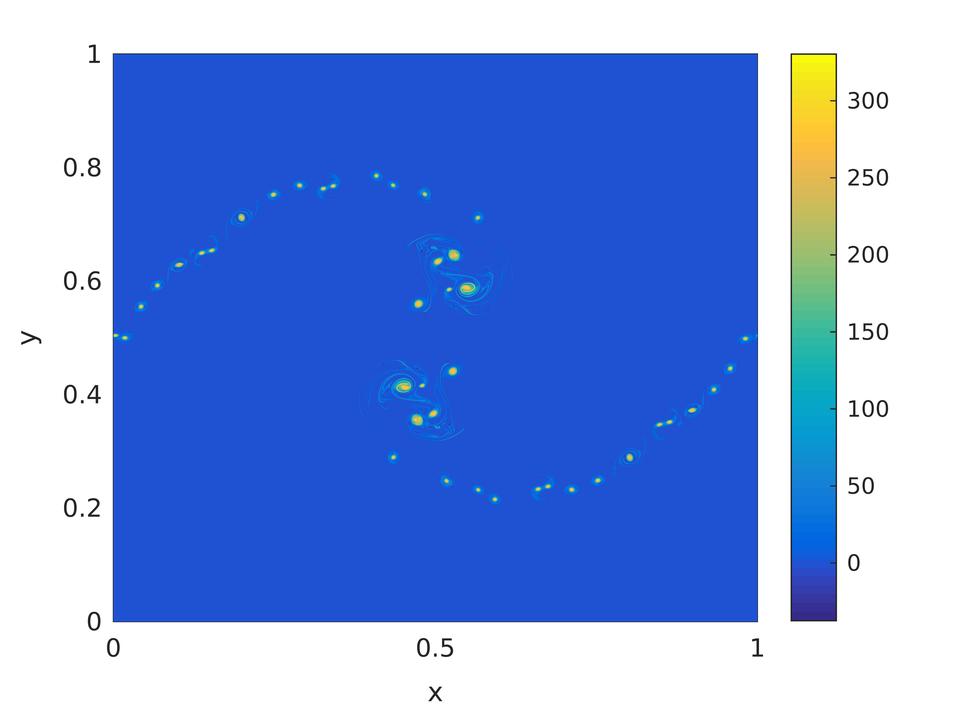}
\caption{$t=0.8$}
\end{subfigure}
\caption{Evolution in time for the singular (thin) vortex sheet with the spectral viscosity method, i.e. $(\e,\rho,k_0) = (0.05,10,N/3)$, on the highest resolution of $N_G=2048$ Fourier modes.}
\label{fig:tvs_sv_evo}
\end{figure}

\begin{figure}[H]
\centering
\begin{subfigure}{0.32\textwidth}
\includegraphics[width=\textwidth]{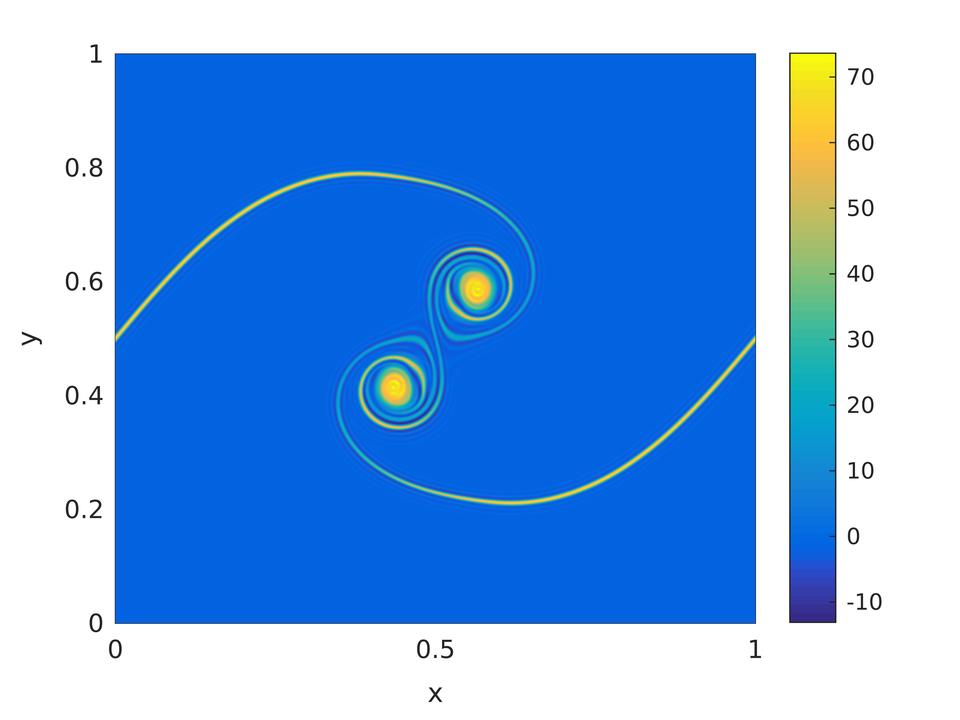}
\caption{$N_G = 512$}
\end{subfigure}
\begin{subfigure}{0.32\textwidth}
\includegraphics[width=\textwidth]{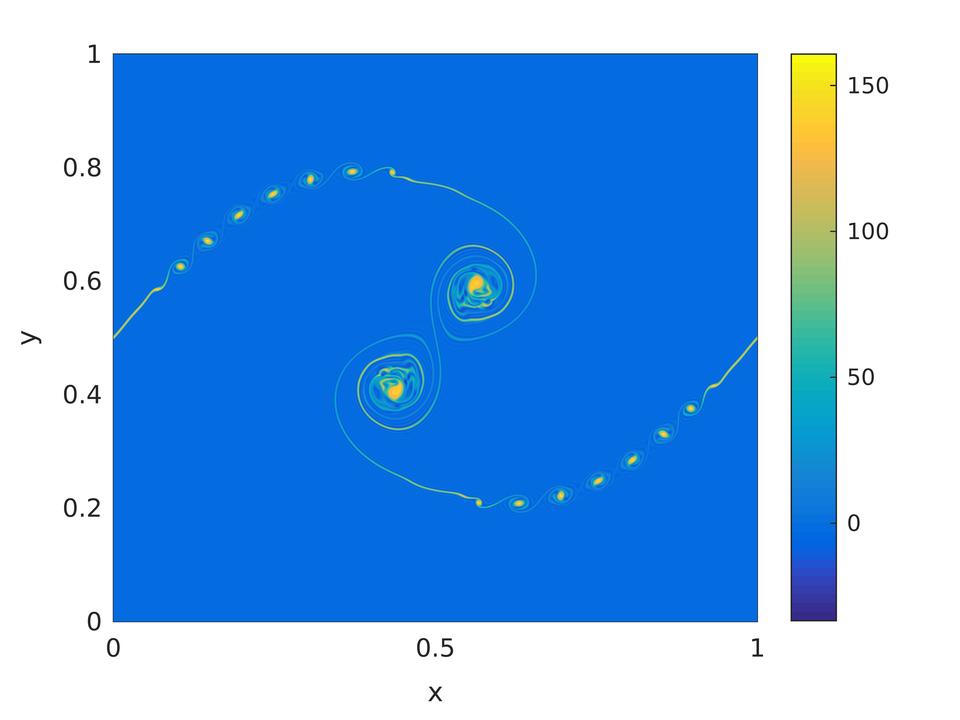}
\caption{$N_G = 1024$}
\end{subfigure}
\begin{subfigure}{0.32\textwidth}
\includegraphics[width=\textwidth]{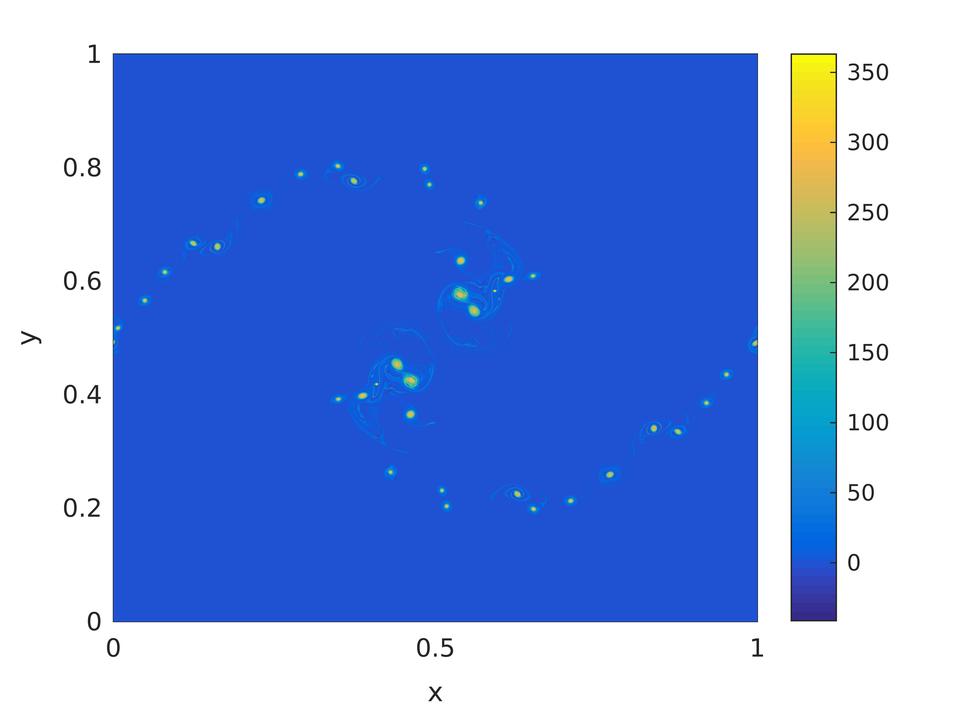}
\caption{$N_G = 2048$}
\end{subfigure}
\caption{Numerical approximations at three different spectral resolutions of the singular vortex sheet with the spectral viscosity method with $(\e,\rho,k_0) = (0.05,10,N/3)$, at time $t=1$ }
\label{fig:tvs_sv_conv}
\end{figure}

These numerical results lead to an interesting dilemma. We have proved in Theorem \ref{thm:convmeasure} that, \revision{up to a subsequence,} the spectral viscosity method converges as the spectral resolution is increased. On the other hand, we see in this experiment that this method may not converge, at least on moderately long time scales. Is there a way to reconcile these two facts. \revision{We argue} that there is no contradiction between the theorem and the numerical observations. As it happens, the solutions of the Euler equations with rough initial data are highly unstable \cite{majda2001}. In particular, very small differences in the initial data can be amplified by possibly double exponential instabilities that lead to very large separation between the underlying solutions, after even a short period of time. Computations of the Euler equations are necessarily approximate and it can happen that even small round off errors are amplified in time and yield small scale vortical structures that eventually can lead to the disintegration of the sheet. These instabilities are damped at low to moderate resolutions but will appear at very high resolutions. Moreover, they tend to accumulate in time and only seems to appear at later times. 

\begin{figure}[H]
\centering
\begin{subfigure}{0.48\textwidth}
\includegraphics[width=\textwidth]{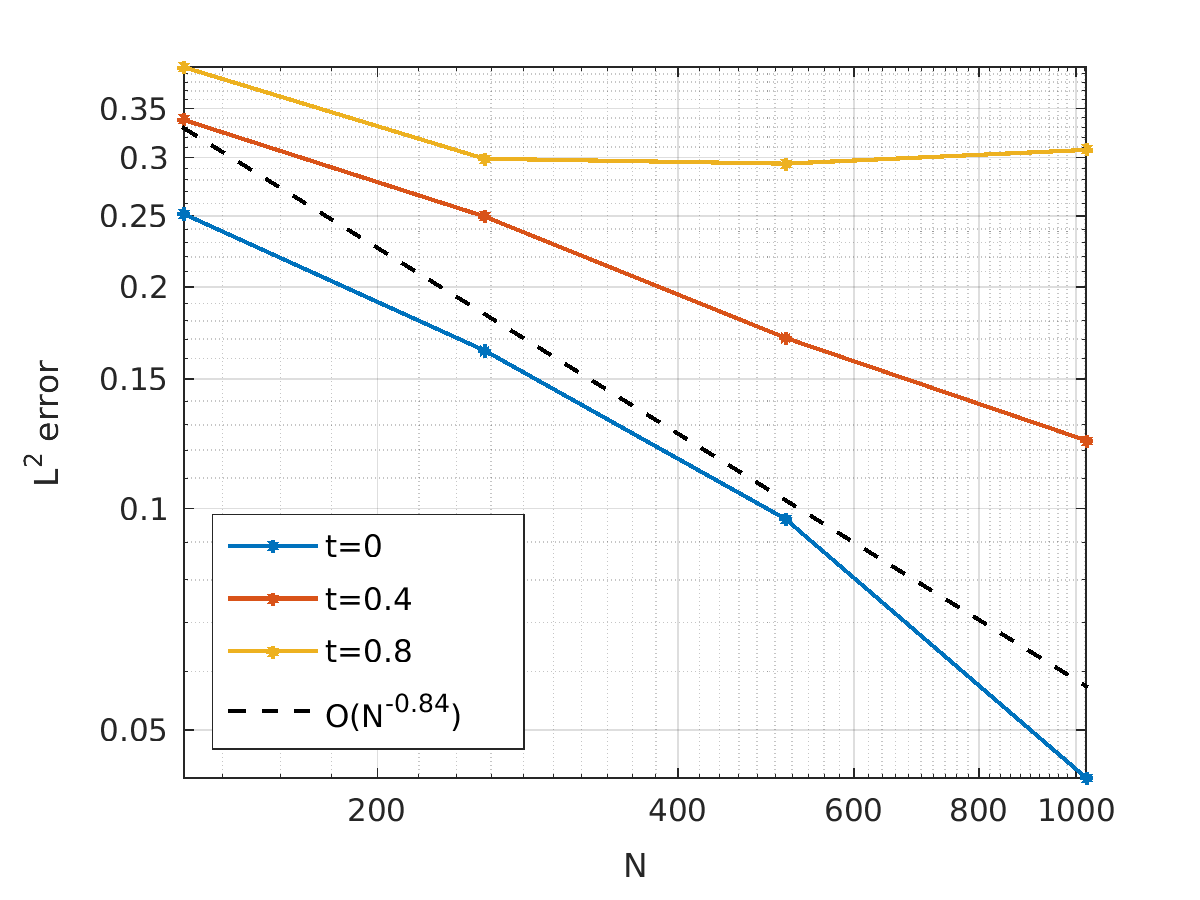}
\caption{$L^2$-error}
\end{subfigure}
\begin{subfigure}{0.48\textwidth}
\includegraphics[width=\textwidth]{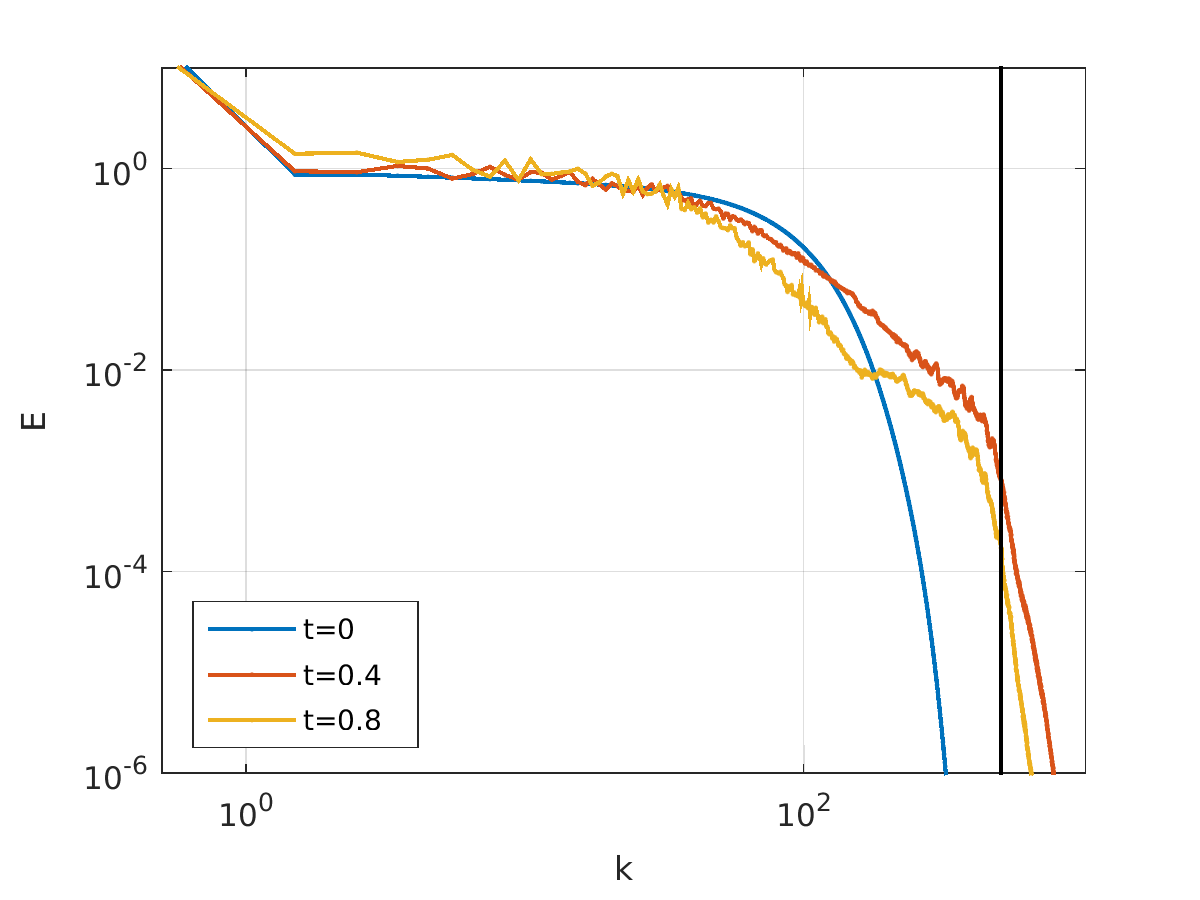}
\caption{Energy spectrum}
\end{subfigure}
\caption{Results for the singular (thin) vortex sheet with the  with the spectral viscosity method with $(\e,\rho,k_0) = (0.05,10,N/3)$ at time $t=1$. (A): Error of the approximate velocity field \eqref{eq:errv} in $L^2$ (B): Energy spectrum \eqref{eq:numspec} for the highest resolution of $N_G=2048$ at different times. }
\label{fig:tvs_sv_err}
\end{figure}

It is interesting to contrast the lack of convergence of the spectral viscosity method (figure \ref{fig:tvs_sv_err} (A)) with the apparent convergence of the vanishing viscosity method (figure \ref{fig:tvs_err} (A)). Clearly the vanishing viscosity method, at least for the parameters considered above, is significantly more dissipative than the spectral viscosity method at the same resolution. This is seen from the computed spectrum (comparing figure \ref{fig:tvs_err} (B) and figure \ref{fig:tvs_sv_err} (B)) as we observe that the vanishing viscosity method damps the small scale instabilities and prevents the transfer of energy into the smallest scales. However, the amount of viscosity is $\e_N = \frac{\e}{N}$. Thus, increasing the resolution further with the vanishing viscosity method can reduce the viscous damping and possibly to the instabilities building up and leading to the disintegration of the sheet. Given that it is unfeasible to increase the resolution beyond $N=2048$ Fourier modes, we mimic this possible behavior by reducing the constant to $\e = 0.01$ in the vanishing viscosity method. The resulting approximate vorticities at time $t=1$, for three different resolutions is shown in figure \ref{fig:tvs_1_conv}. We observe from this figure that the results are very similar to the spectral viscosity method (compare with figure \ref{fig:tvs_sv_conv}) and the sheet disintegrates into a soup of small vortices at the highest resolution. Consequently, there is no convergence of the velocity in $L^2$ as seen from figure \ref{fig:tvs_1_err} (A) and the spectrum shows that more energy is transferred to the smallest scales now than it was when $\e = 0.05$ (compare with figure \ref{fig:tvs_err} (B)). 

\begin{figure}[H]
\centering
\begin{subfigure}{0.32\textwidth}
\includegraphics[width=\textwidth]{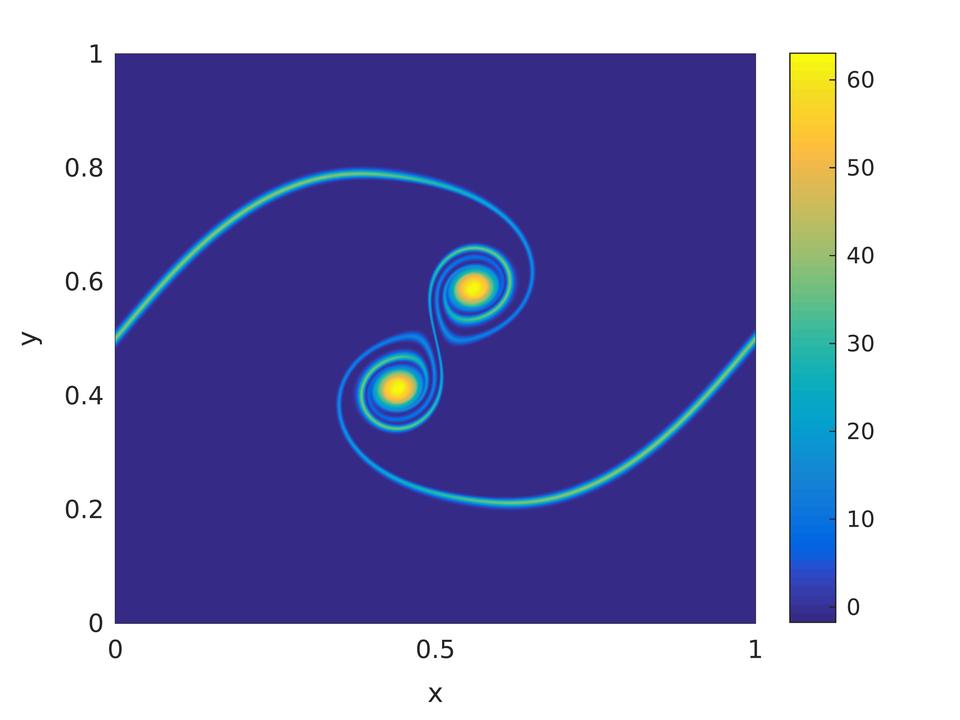}
\caption{$N_G = 512$}
\end{subfigure}
\begin{subfigure}{0.32\textwidth}
\includegraphics[width=\textwidth]{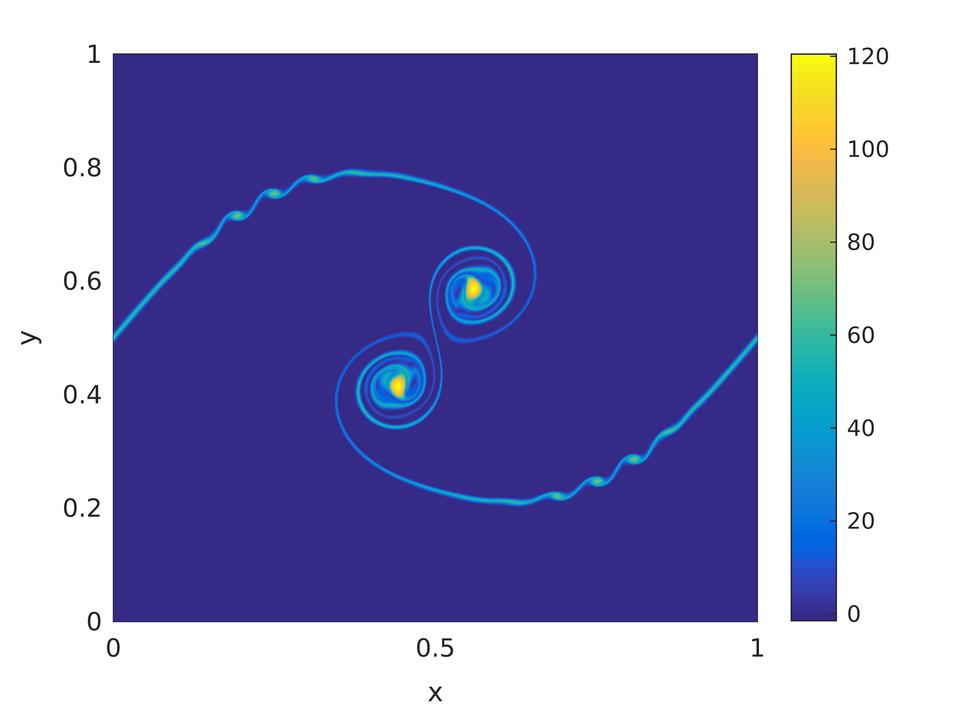}
\caption{$N_G = 1024$}
\end{subfigure}
\begin{subfigure}{0.32\textwidth}
\includegraphics[width=\textwidth]{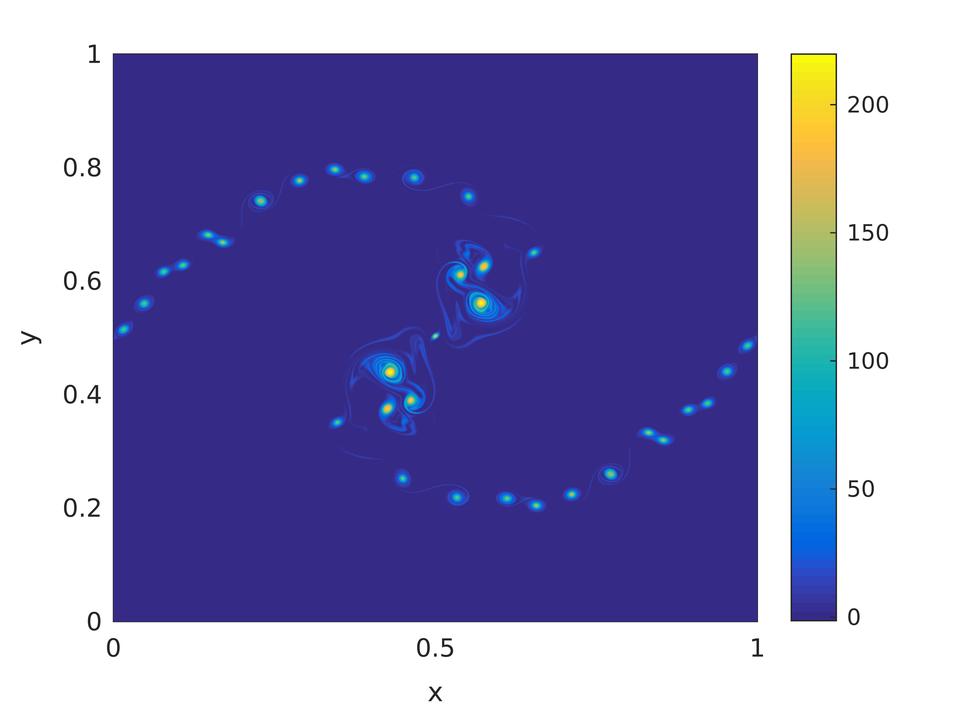}
\caption{$N_G = 2048$}
\end{subfigure}
\caption{Numerical approximations at three different spectral resolutions of the singular vortex sheet with the vanishing viscosity method with $(\e,\rho) = (0.01,10)$, at time $t=1$ }
\label{fig:tvs_1_conv}
\end{figure}

\begin{figure}[H]
\centering
\begin{subfigure}{0.48\textwidth}
\includegraphics[width=\textwidth]{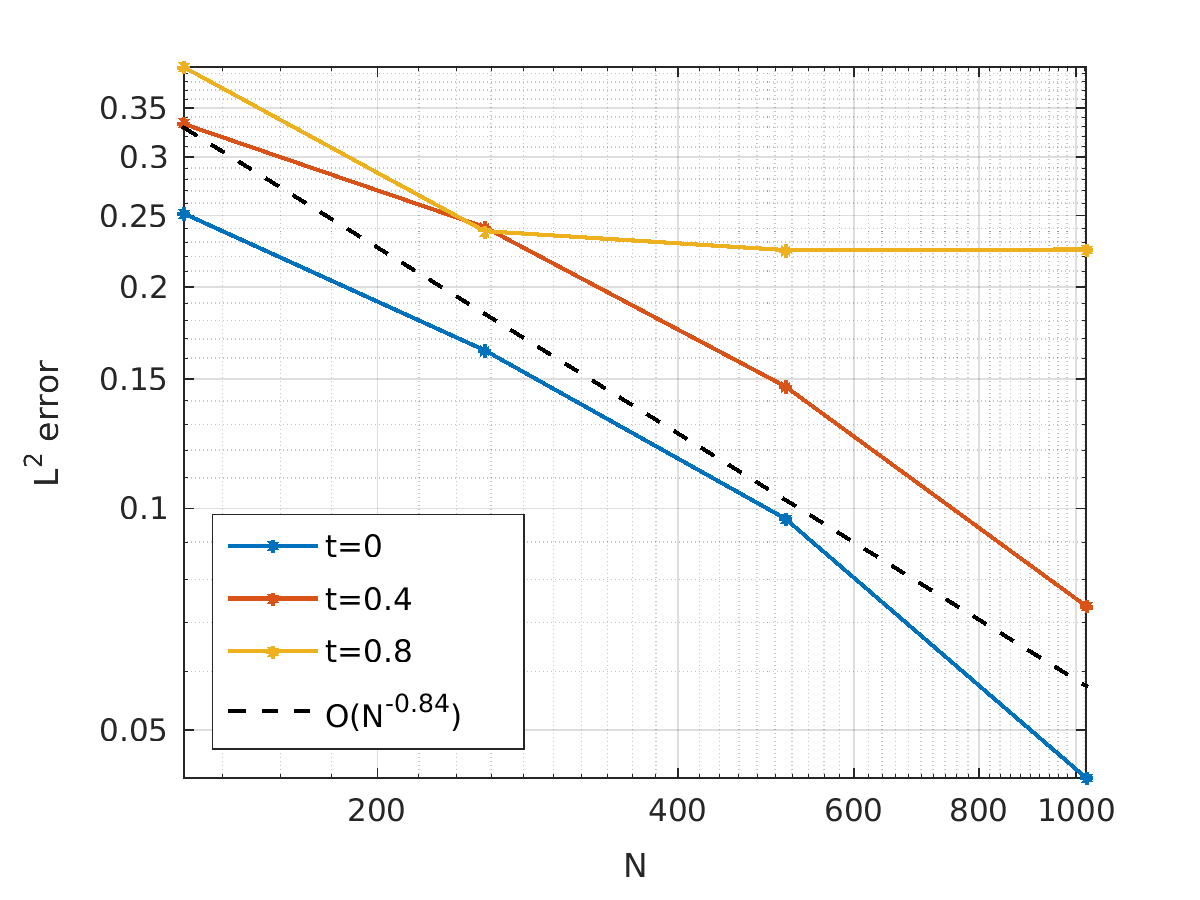}
\caption{$L^2$-error}
\end{subfigure}
\begin{subfigure}{0.48\textwidth}
\includegraphics[width=\textwidth]{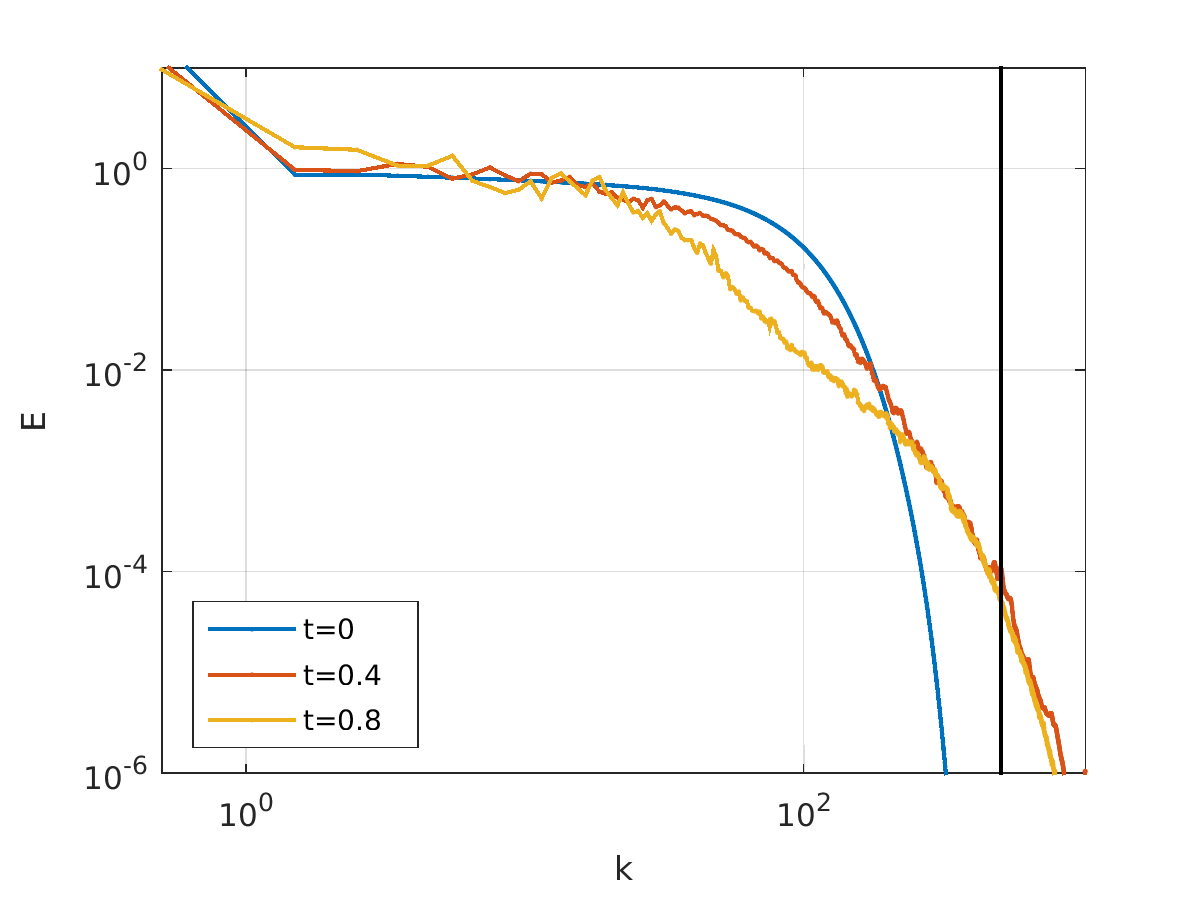}
\caption{Energy spectrum}
\end{subfigure}
\caption{Results for the singular (thin) vortex sheet with the  with the vanishing viscosity method i.e. $(\e,\rho) = (0.01,10)$ at time $t=1$. (A): Error of the approximate velocity field \eqref{eq:errv} in $L^2$ (B): Energy spectrum \eqref{eq:numspec} for the highest resolution of $N_G=2048$ at different times. }
\label{fig:tvs_1_err}
\end{figure}

The lack of convergence of computations of singular vortex sheets, on account of the formation and amplification of small scale instabilities, is well known and can be traced back to the pioneering work of Krasny \cite{Kras1,Kras2} and reference therein. In those papers, the author computed singular vortex sheets by solving the Birkhoff-Rott equations of vortex dynamics and was able to ensure stable computation by controlling the round-off errors with an adaptive increase of the arithmetic precision of the computation. We believe that this fix is only relevant for a few levels of increasing resolution and ultimately at very high resolutions, the vortex sheet will disintegrate into smaller vortices. This is already evidenced by our computations at different resolutions, at different times and with different values of the viscosity parameter $\e$. Paraphrasing \cite{majda2001}, the phenomenon
of the exponential growth of small instabilities `is a feature of the underlying equation itself as opposed to an instability of the numerical method.' 

\begin{figure}[H]
\centering
\begin{subfigure}{0.32\textwidth}
\includegraphics[width=\textwidth]{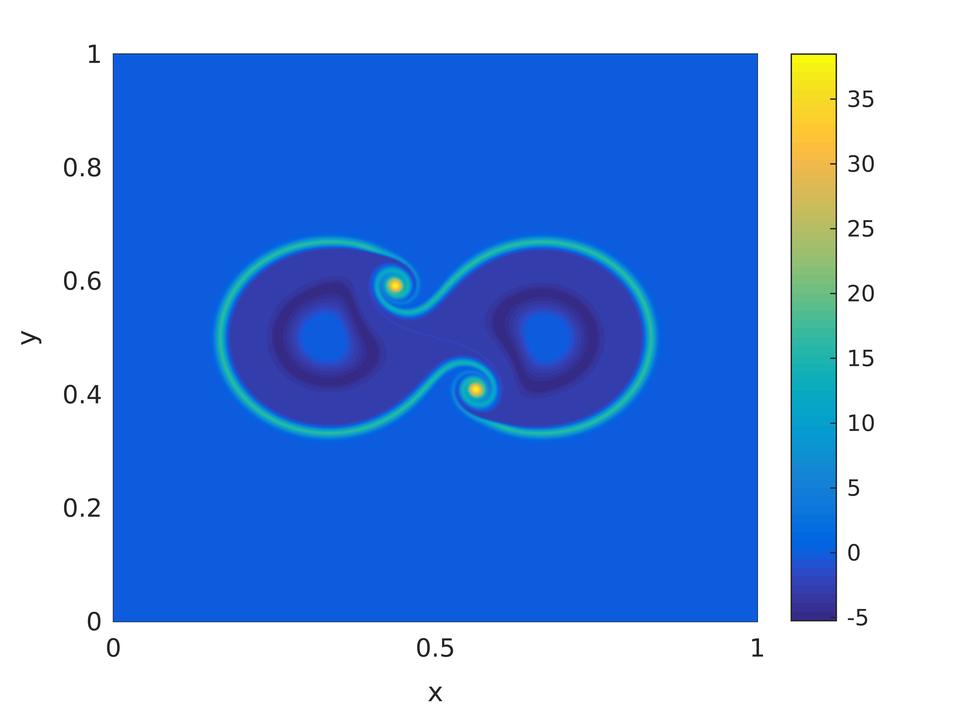}
\caption{$N_G = 512$}
\end{subfigure}
\begin{subfigure}{0.32\textwidth}
\includegraphics[width=\textwidth]{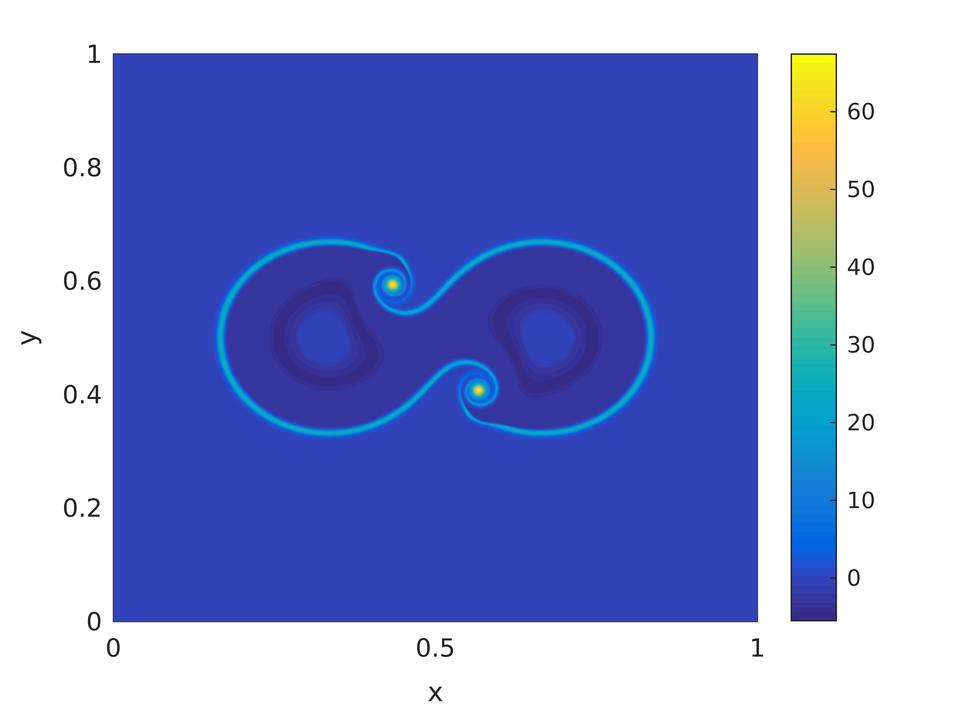}
\caption{$N_G = 1024$}
\end{subfigure}
\begin{subfigure}{0.32\textwidth}
\includegraphics[width=\textwidth]{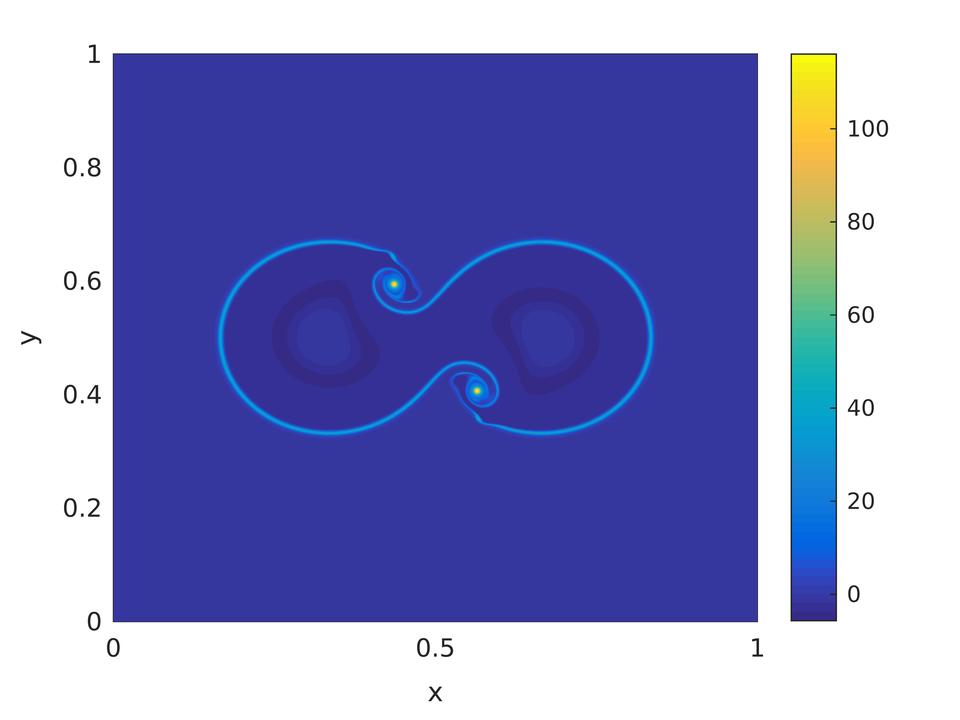}
\caption{$N_G = 2048$}
\end{subfigure}

\caption{Numerical approximations at three different spectral resolutions of the kissing vortices with the vanishing viscosity method with $(\e,\rho) = (0.01,10)$, at time $t=1$ }
\label{fig:kv_conv}
\end{figure}

\subsection{Kissing vortices}
As a second example, we apply the spectral viscosity method to initial data which have been proposed in \cite{Lopes2006} as a possible example of an initial datum that lead to non-uniqueness of Delort solutions.

This example is based on the following observation \cite{Lopes2006}: In polar coordinates $(r,\theta)$, centered at the origin $x_0 = 0 \in \R^2$, a weak stationary solution of the 2d Euler equations can be constructed by setting
\[
\omega_c(x) = \omega_{-}(r) + \omega_{+}\delta(r-1),
\]
where $r=|x|$, $\omega_{-}(r)$ is a suitable smooth function of $r$, and $\omega_{-}(r) = 0$ for $r\ge 1$. Furthermore, by choosing the constant $\omega_{+} > 0$ in a suitable way, one can ensure that the velocity field $\vec{u}_c$ corresponding to $\omega_c$ vanishes outside of the unit disk, i.e. that $\vec{u}_c(x) = 0$, for $|x|> 1$. Following \cite{Lopes2006}, we call such a solution a \emph{confined eddy}.

Since $\vec{u}_c$ has compact support, it is possible to obtain a new stationary weak solution, by superposing two confined eddies with essentially disjoint supports, i.e. we can e.g. set
\[
\omega_0(x) = 
\omega_c\left(\frac{x-x_0}{R}\right)
+
\omega_c\left(\frac{x-x_1}{R}\right),
\]
where $R$ is the radius determining the support each confined eddy. This initial datum $\omega_0$ is then found to be a stationary weak solution of the Euler equations \cite{Lopes2006}, provided that $2R \le |x_1-x_0|$.

\begin{figure}[H]
\centering
\begin{subfigure}{.32\textwidth}
\includegraphics[width=\textwidth]{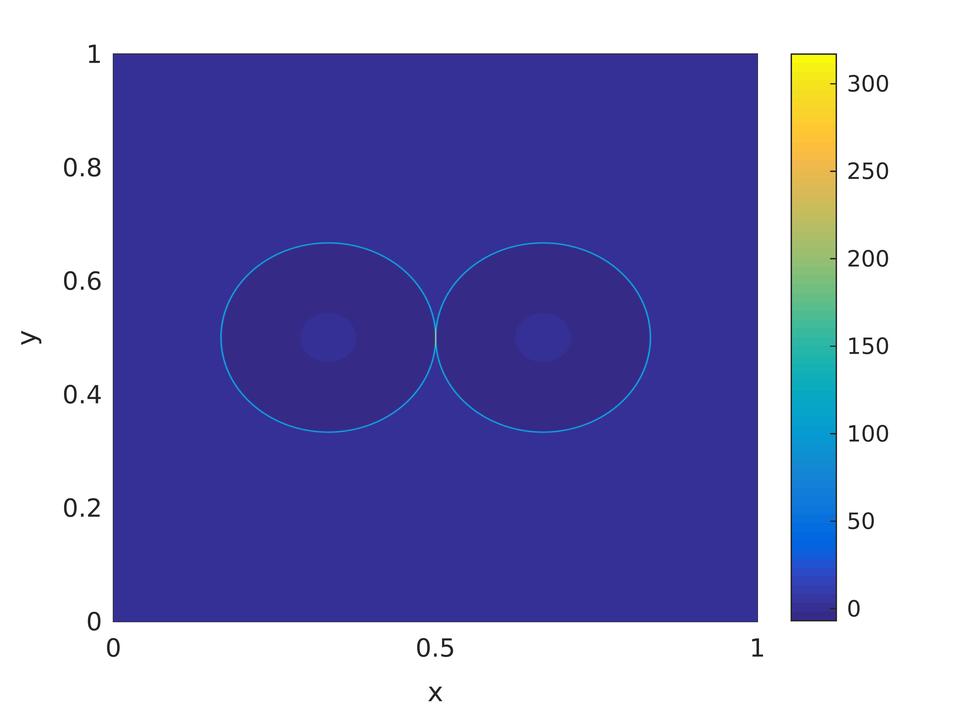}
\caption{$t=0.0$}
\end{subfigure}
\begin{subfigure}{.32\textwidth}
\includegraphics[width=\textwidth]{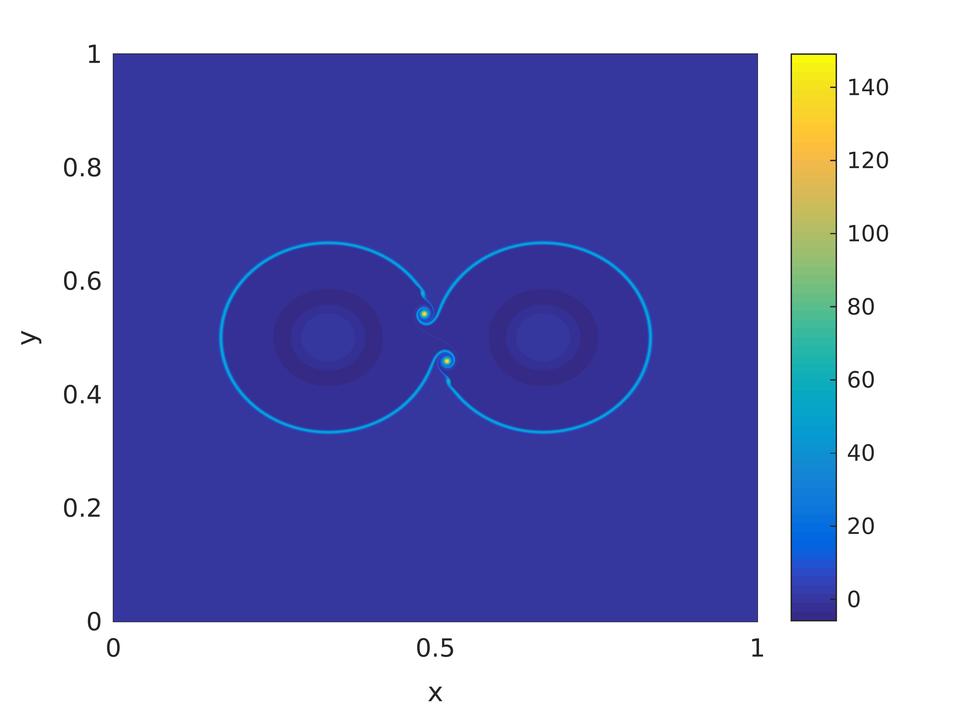}
\caption{$t=0.4$}
\end{subfigure}
\begin{subfigure}{.32\textwidth}
\includegraphics[width=\textwidth]{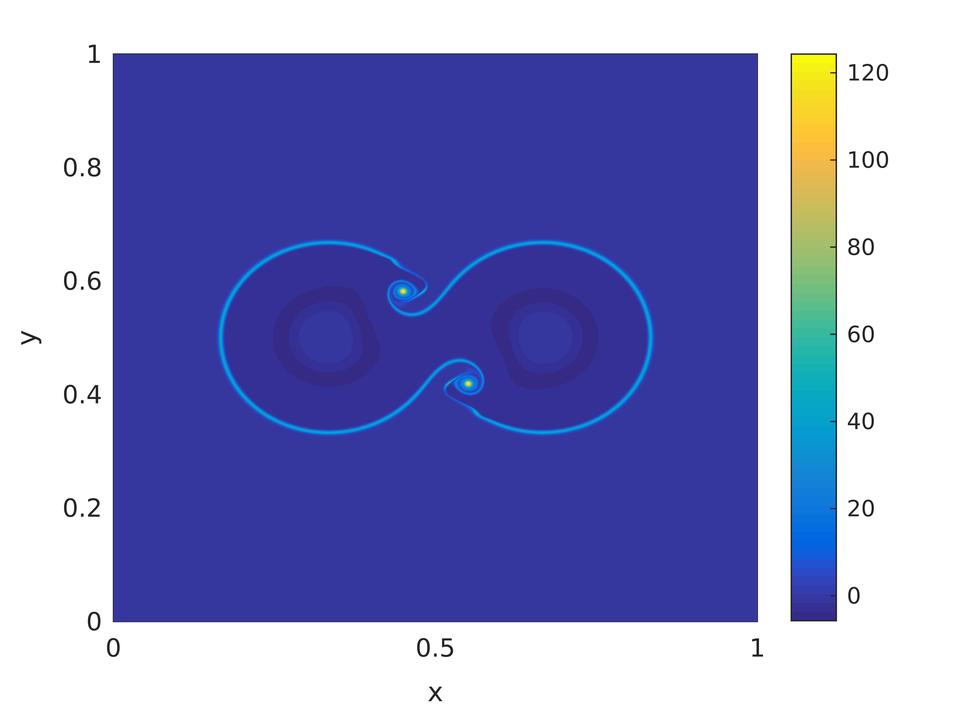}
\caption{$t=0.8$}
\end{subfigure}

\caption{Evolution in time for the kissing vortices with the vanishing viscosity method i.e. $(\e,\rho) = (0.01,10)$, on the highest resolution of $N_G=2048$ Fourier modes.}
\label{fig:kv_evo}
\end{figure}

For the numerical implementation, we choose $x_1 = (-1/3,0)$, $x_2 = (2/3,0)$, $R=1/6$, so that the vortices are tangent at $(1/2,0)$. Each confined eddy is defined via the corresponding velocity: $\vec{u}_c(x) = v(r) \vec{x}^\perp$, where
\begin{equation}
v(r) 
\defeq
\begin{cases}
0, & (r<1/4), \\
2\pi(r-1/4) , & (1/4\le r \le 1/2), \\
\dfrac{\pi \left\{\tanh\left(\frac{1-r}{\rho_N}\right)+1\right\}}4, & (r>1/2).
\end{cases}
\end{equation}
Note that 
\[
\lim_{\rho\to 0} 
\dfrac{\pi \left\{\tanh\left(\frac{1-r}{\rho_N}\right)+1\right\}}4
=
\frac{\pi}{2}\;  1_{[r<1]},
\]
so that $\rho_N$ represents the mollification parameter in our numerical scheme. We choose $\rho_N = \rho/N_G = \rho/(2N)$ with constant $\rho=10$, in the following.

\begin{figure}[H]
\centering
\begin{subfigure}{0.48\textwidth}
\includegraphics[width=\textwidth]{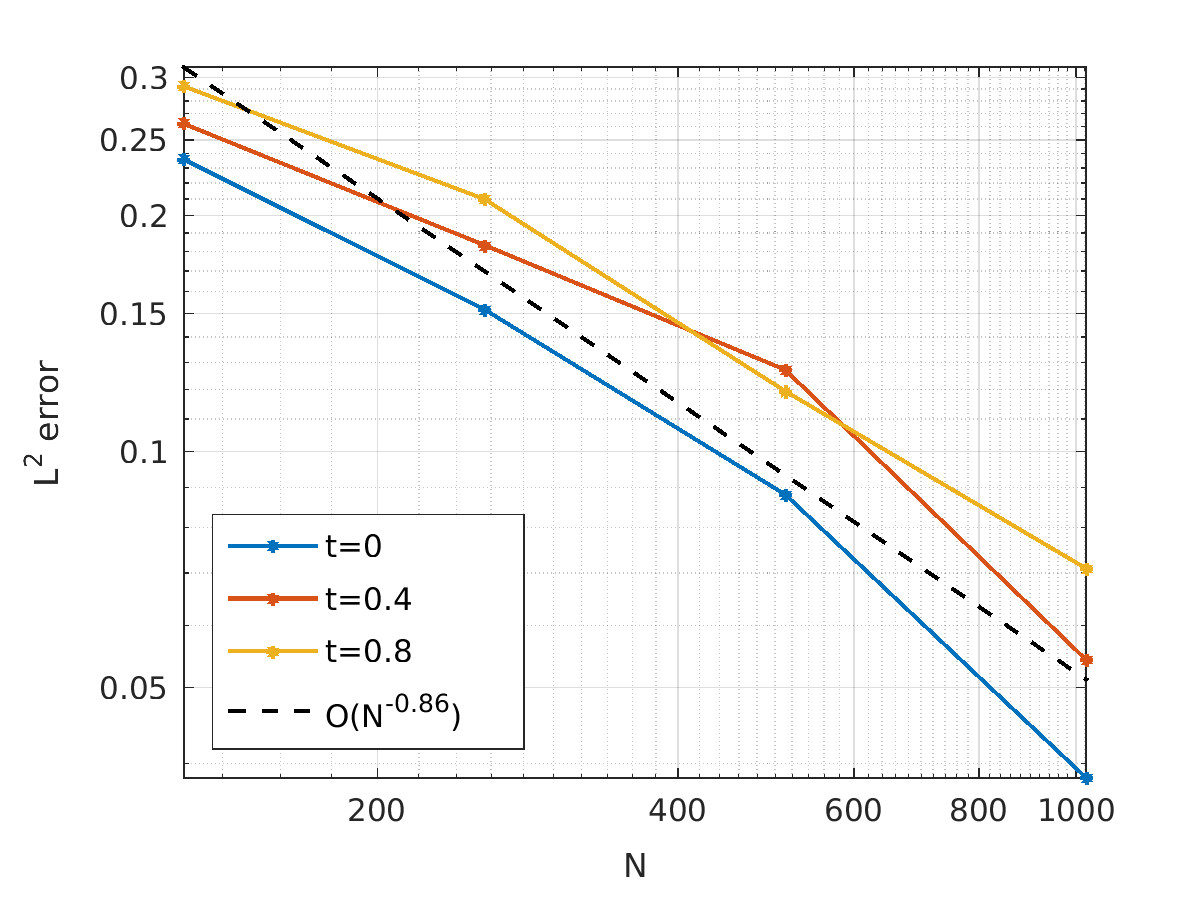}
\caption{$L^2$-error}
\end{subfigure}
\begin{subfigure}{0.48\textwidth}
\includegraphics[width=\textwidth]{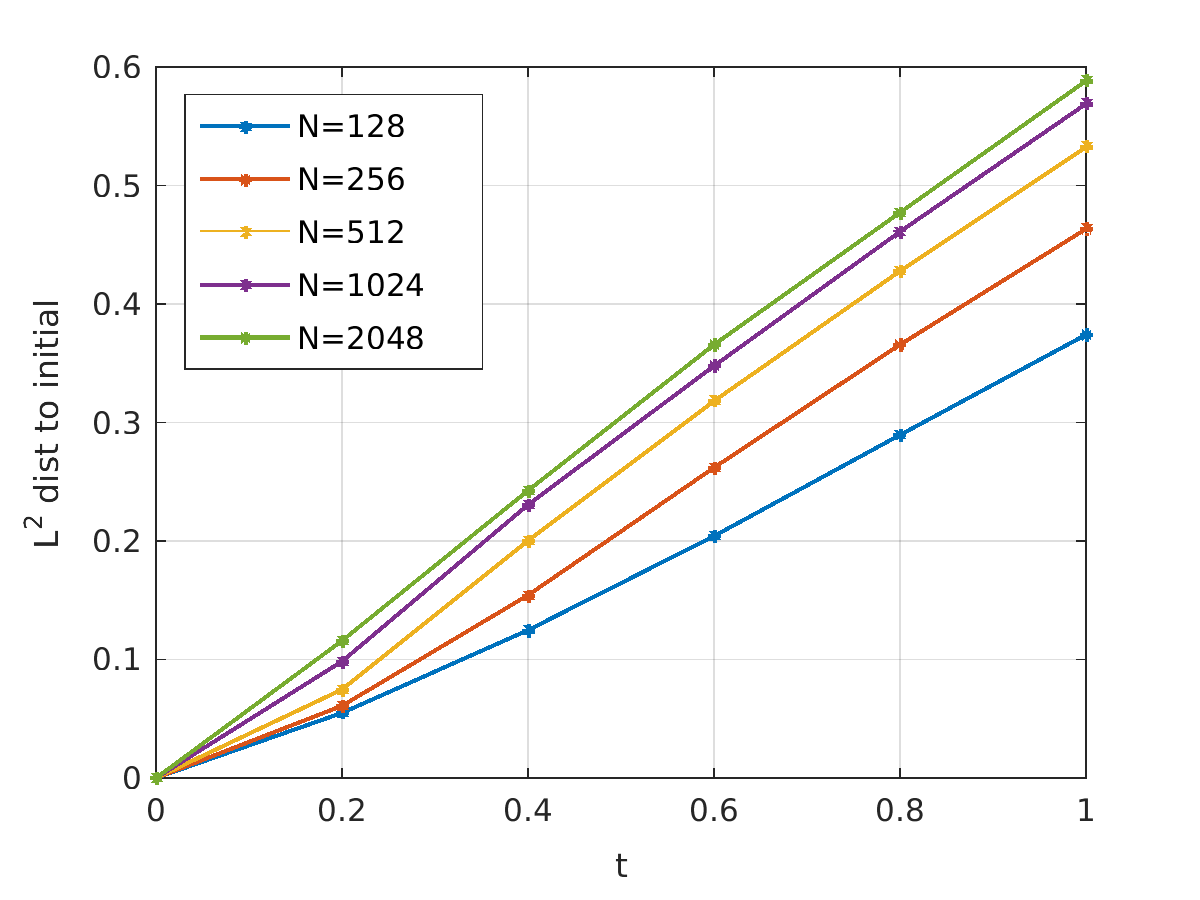}
\caption{$L^2$-difference wrt initial data}
\end{subfigure}
\caption{Results for the kissing vortices with the  with the vanishing viscosity method i.e. $(\e,\rho) = (0.01,10)$ at time $t=1$. (A): Error of the approximate velocity field \eqref{eq:errv} in $L^2$ (B): Difference in $L^2$ between the computed velocity field and the initial data for different resolutions as a function of time. }
\label{fig:kv_err}
\end{figure}

We start by approximating the solutions of the two-dimensional Euler equations with the above initial data, by a vanishing viscosity method with parameters $(\e,\rho, k_0) = (0.01,10,0)$ and present the computed vorticities, on a sequence of successively refined levels of resolution, at time $t=1$, in figure \ref{fig:kv_conv}. As seen from the figure and verified from the $L^2$-approximation error of the velocity field \eqref{eq:errv}, the computed solution appears to converge in this regime. More interestingly, the solutions appears to converge to a vorticity distribution that it is very different from the initial datum. This time evolution is shown in figure \ref{fig:kv_evo} and we observe from the figure that the two initial confined eddies are twisted by the time evolution and spiral into two distinct vortices. In figure \ref{fig:kv_err} (B), we plot the difference between the initial datum and the computed velocity field in $L^2$ for each time and plot the evolution of this quantity in time. We observe from this figure that the difference increases linearly over time. Moreover, this difference increases with resolution. 

\begin{figure}[H]
\centering
\begin{subfigure}{.32\textwidth}
\includegraphics[width=\textwidth]{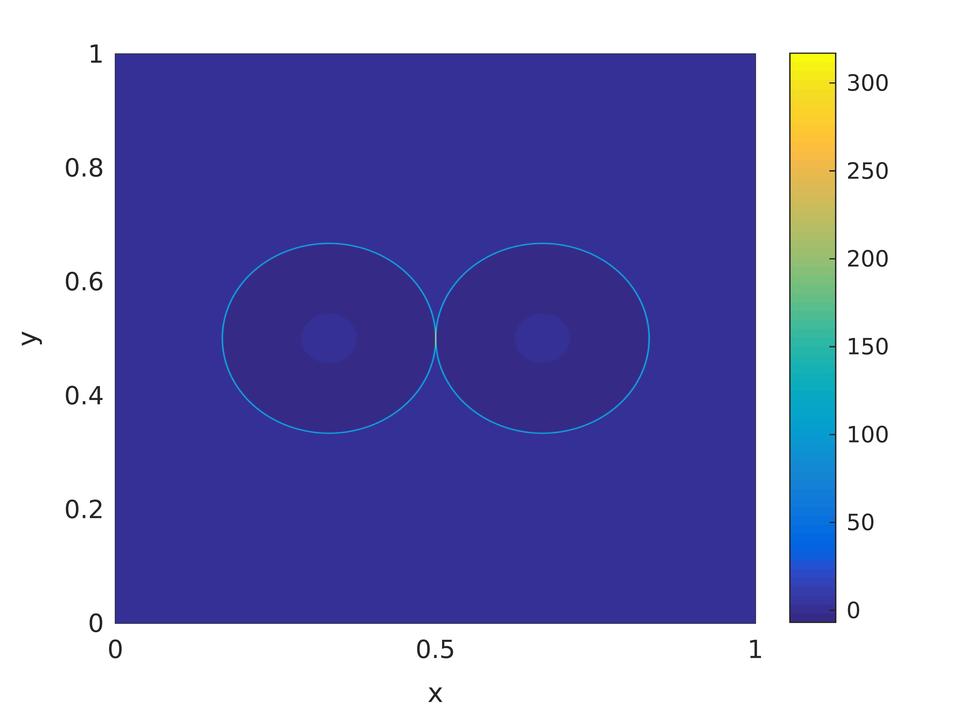}
\caption{$t=0.0$}
\end{subfigure}
\begin{subfigure}{.32\textwidth}
\includegraphics[width=\textwidth]{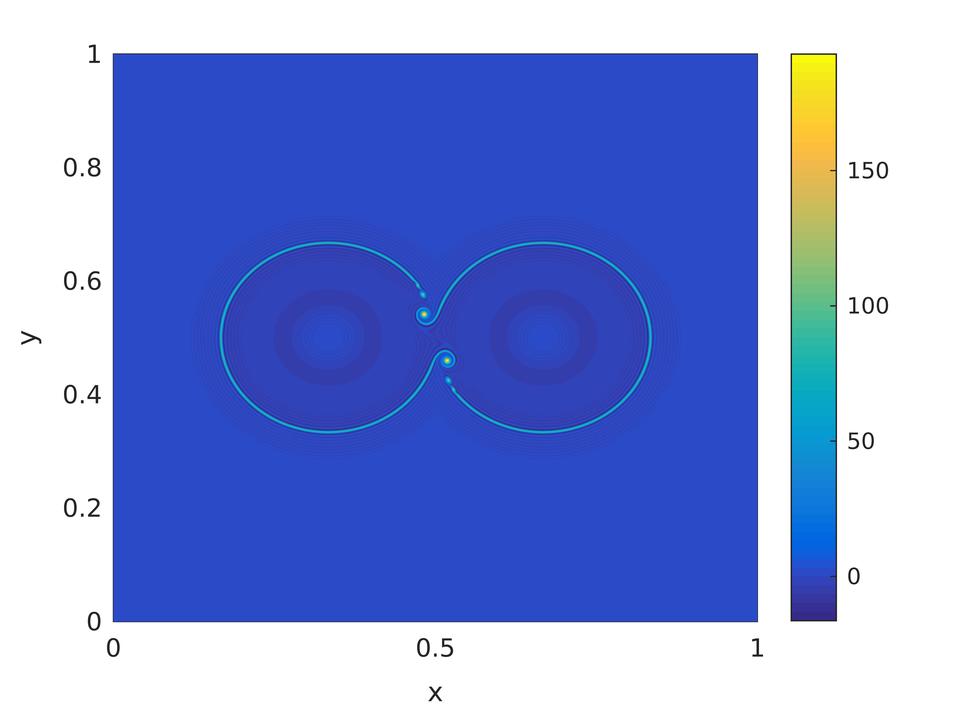}
\caption{$t=0.4$}
\end{subfigure}
\begin{subfigure}{.32\textwidth}
\includegraphics[width=\textwidth]{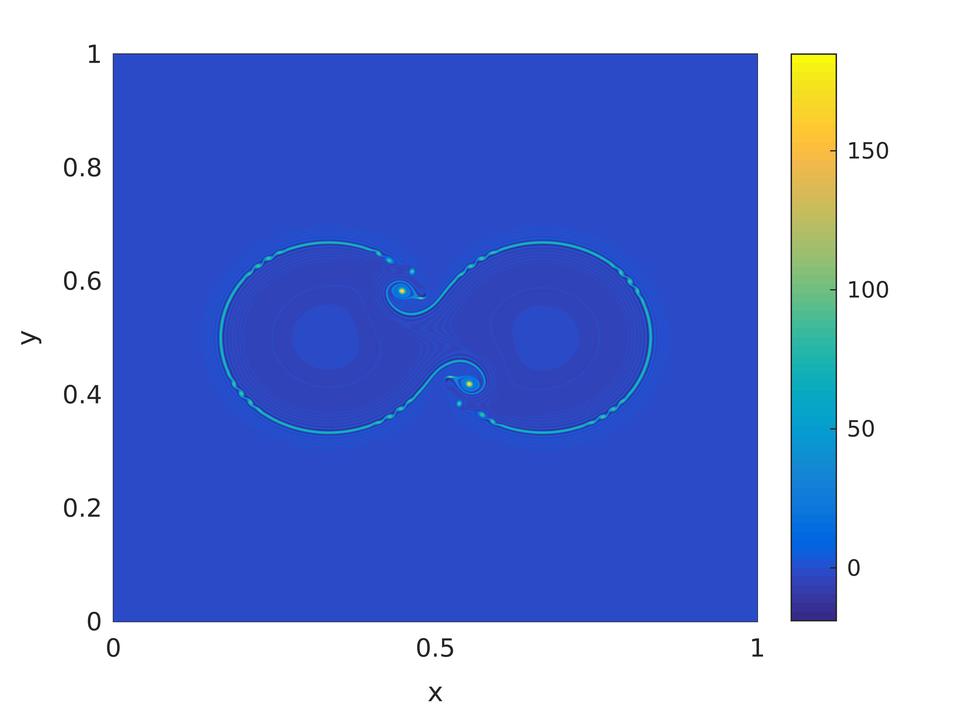}
\caption{$t=0.8$}
\end{subfigure}
\caption{Evolution in time for the kissing vortices with the spectral viscosity method i.e. $(\e,\rho,k_0) = (0.01,10,N/8)$, on the highest resolution of $N_G=2048$ Fourier modes.}
\label{fig:kv_sv_evo}
\end{figure}

Similar results are also obtained with a spectral viscosity method with parameters $(\e,\rho,k_0) = (0.01,10,N/8)$. The convergence of the computed velocity field is verified from figure \ref{fig:kv_sv_err} (A) and the time evolution of the vorticity (at the highest spectral resolution) is shown in figure \ref{fig:kv_sv_evo}. Clearly, the computed vorticity is very similar to the one computed with the vanishing viscosity method and very different from the initial datum as inferred from figure \ref{fig:kv_sv_err} (B).

\begin{figure}[H]
\centering
\begin{subfigure}{0.48\textwidth}
\includegraphics[width=\textwidth]{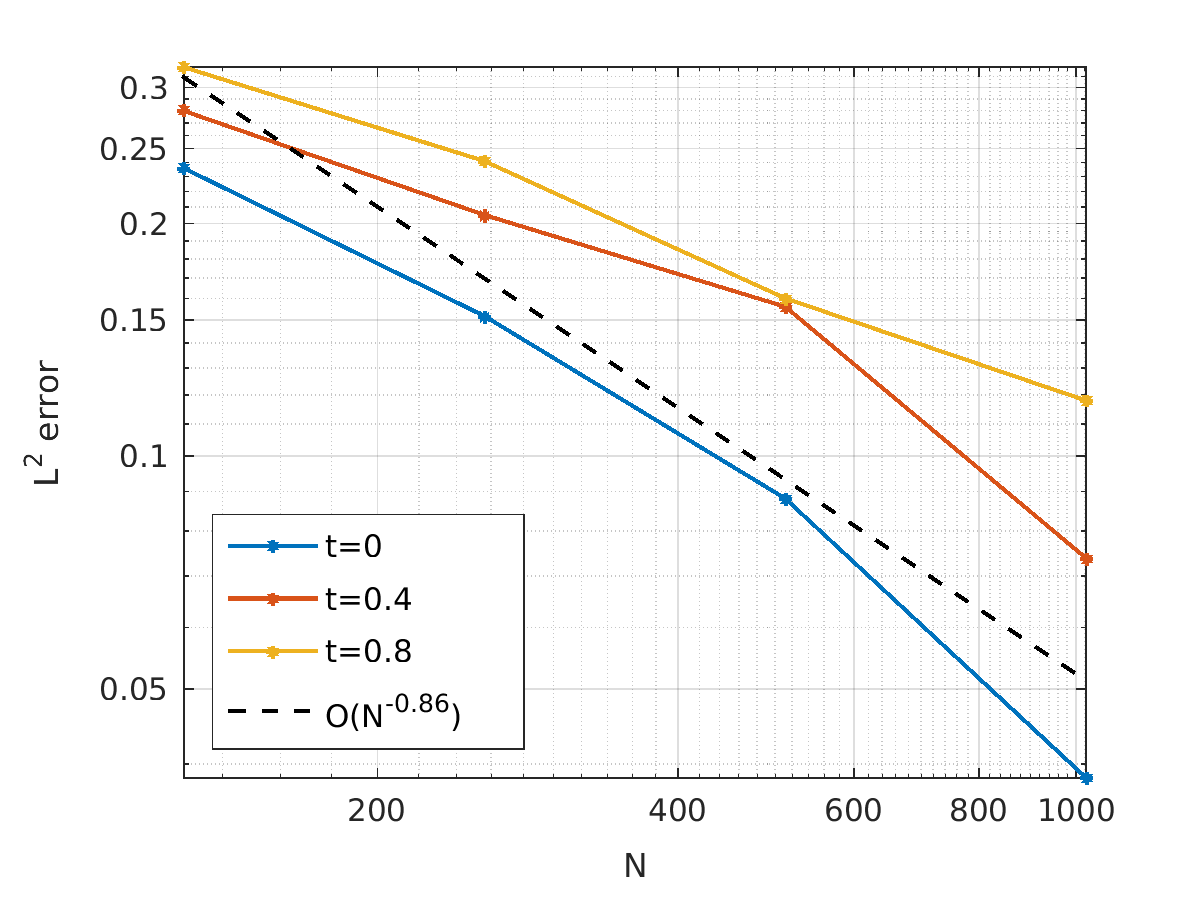}
\caption{$L^2$-error}
\end{subfigure}
\begin{subfigure}{0.48\textwidth}
\includegraphics[width=\textwidth]{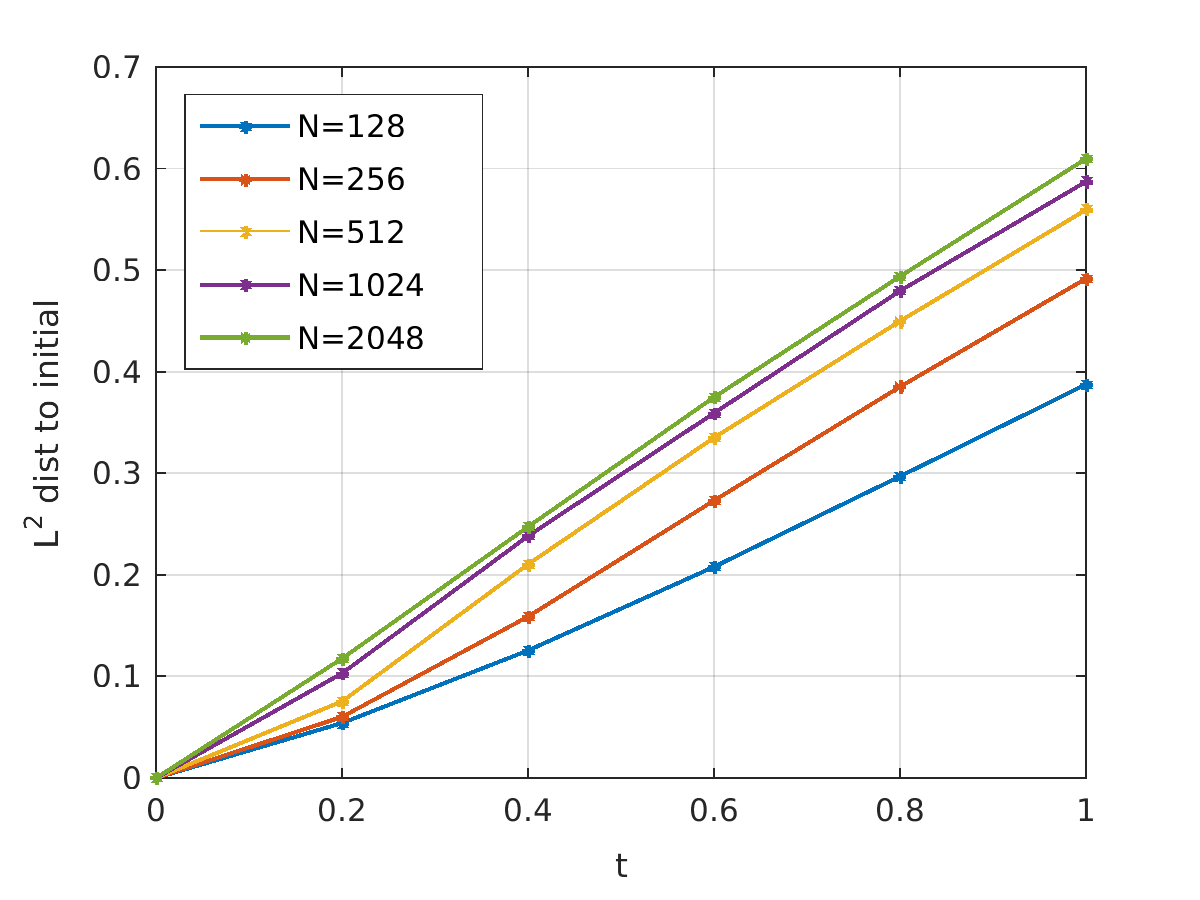}
\caption{$L^2$-difference wrt initial data}
\end{subfigure}
\caption{Results for the kissing vortices with the  with the spectral viscosity method i.e. $(\e,\rho,k_0) = (0.01,10,N/8)$ at time $t=1$. (A): Error of the approximate velocity field \eqref{eq:errv} in $L^2$ (B): Difference in $L^2$ between the computed velocity field and the initial data for different resolutions, as a function of time. }
\label{fig:kv_sv_err}
\end{figure}

Both computations clearly indicate that the solutions computed with the spectral viscosity method converge to a velocity field that is different from the initial datum. This suggests non-uniqueness of weak solutions for the two-dimensional incompressible Euler equations when the initial data is in the \emph{Delort class}. This non-uniqueness was already suggested by the computations reported in \cite{Lopes2006}. We add further weight to this conclusion by observing the same behavior but with a different numerical method, particularly one that is proved to converge to a Delort solution on refinement of resolution.   
\section{Conclusion}
\label{sec:conc}
In this paper, we considered the two-dimensional incompressible Euler equations. In contrast to the three-dimensional case, global well-posedness results are available in two space dimensions. In particular, existence and uniqueness of weak solutions is proved under the
assumption that the initial vorticity is in $L^{\infty}$. Moreover, global (in time) existence of weak solutions is proved for significantly less regular initial data, for instance when the initial vorticity belongs to the so-called \emph{Delort class}. Such rough initial data are encountered in practice when one considers the evolution of vortex sheets in an ideal fluid.

Although many different numerical methods have been developed to approximate the incompressible Euler equations, convergence results for these schemes have mostly been available in the regime where the initial data and the underlying solutions were smooth. Notable exceptions were considered in \cite{Levy1997} and \cite{Lopes2000}, where the authors prove convergence of central finite difference schemes for the vorticity formulation of the equations under the assumption that the initial vorticity is in $L^p$, for $1 < p \leq \infty$\revision{, and more generally if the vorticity belongs to a rearrangement invariant space that is compactly supported in $H^{-1}$. For vortex methods \cite{LiuXin1995,Schochet1996,LiuXin2001}, convergence is known when the initial vorticity is a bounded measure of definite sign, or if the vorticity is in $L(\log L)$ without any sign restriction.} However, no rigorous convergence results are available for the case of Delort class initial data. Thus, there has so far remained a considerable gap between the mathematical existence results and rigorous convergence results for numerical approximations. 

In this paper, we have proposed a \emph{spectral viscosity} method to approximate the two-dimensional Euler equations. Based on the spectral viscosity framework of Tadmor \cite{Tadmor1989} and references therein, our method is a spectral method that discretizes the Euler equations in Fourier space. Viscosity (damping) is only added in the high wave-number Fourier modes. Consequently, the method is formally spectrally (superpolynomially) accurate for smooth solutions. Till now, convergence of this method was only proved for smooth solutions of the incompressible Euler equations \cite{Bardos2015}. 

We prove that our spectral viscosity method converges to a weak solution as long as the initial vorticity either bounded in $L^p$ for $1 \leq p \leq \infty$ or in the Delort class. Thus, we provide the first rigorous convergence results for a numerical approximation of the two-dimensional incompressible Euler equations with initial data in the Delort class. This also closes the gap between available existence results for the underlying PDE and convergence results for numerical approximation. 

Our proof relies on the following key ingredients:
\begin{itemize}
\item The equivalence of the spectral viscosity method for the velocity-pressure formulation \eqref{eq:Euler} and the vorticity formulation \eqref{eq:vorticity}. This equivalence holds for any resolution i.e. truncation of the underlying Fourier expansion.
\item A spectral decay estimate for the high wave-number modes.
\item A patching up of long-time estimates on the vorticity (obtained by the spectral decay estimate) and short-time estimates.
\item A novel approximation of rough initial data that amounts to resolving the initial singularities.
\item Application of the compensated compactness theorems of Delort by controlling the negative part of the approximated vorticity. In particular, we ensure that the negative part of the vorticity, as approximated by the spectral viscosity method, cannot concentrate on sets of small measure.
\end{itemize}

It is unclear if these ingredients, particularly the equivalence between the velocity-pressure and vorticity formulations, can be transferred to other numerical methods. Thus, for the time being, the spectral viscosity method is the only method that can rigorously be proved to converge to weak solutions for the incompressible Euler equations with rough initial data. \revision{As our results are based on a spectral Fourier expansion, they are inherently limited to the periodic case. It is not clear, whether the method can be extended to other boundary conditions, and in particular to schemes providing numerical approximations of flows in the whole plane. Furthermore, due to the lack of theoretical existence results on domains with boundary, a convergence proof on such domains appears to be out of reach at present.}

We present some representative numerical experiments to test the proposed spectral viscosity method. We observe from the experiments that the spectral viscosity method performs as well as the pure (standard) spectral method for smooth initial data. 
Moreover, we have also presented experiments with rough initial data that demonstrated the performance of the spectral viscosity method and compared it with the vanishing viscosity method. We observed that both methods were able to compute the problem of kissing vortices robustly and provided numerical evidence for possible non-uniqueness of weak solutions of the incompressible Euler equations, when the initial data is in the Delort class.

We also computed vortex sheets with the spectral and vanishing viscosity methods and observed convergence to complicated roll-ups of the sheet in many cases, particularly for small times. However for very high spectral resolutions and for long times, the computed solutions contained small scale instabilities that amplified (either with time or in resolution or both) and led to the disintegration of the vortex sheet into a soup of small vortices. We argue that this phenomena is generic to such rough data and cannot be alleviated at the level of numerical computations, particularly at very high resolutions. On the other hand, many papers in recent years such as \cite{FMTacta,LM2015,LeonardiPhD} and references therein, have presented computations of vortex sheets and demonstrated that although each deterministic simulation can be unstable, yet statistical quantities (ensemble averages) are computed robustly. This implies that statistical notions of solutions such as dissipative measure-valued solutions \cite{Diperna1987,LM2015,FMTacta} and the more recent statistical solutions \cite{FLM17} might be more appropriate as a solution framework for the incompressible Euler equations, certainly from the perspective of numerical approximation.

%

   \appendix 

\section{Miscellaneous results}

\label{sec:Bernstein}

We shall need some estimates for trigonometric polynomials $f_M(x) = \sum_{|\vec{k}|\le M} \widehat{f}_{\vec{k}} e^{i\vec{k}\cdot \vec{x}}$. We denote by $\P_N$ the projection onto this space. We take them from \cite{Gunzburger2010} (though they may have appeared elsewhere).

\begin{thm} \label{thm:BernsteinLpLq}
Let $1<p\le q<\infty$, or $1<p<q\le \infty$. Then
\[
\Vert \P_N f \Vert_q 
\le C_p N^{d\left(\frac{1}{p} - \frac{1}{q}\right)} \Vert f \Vert_p.
\]
\end{thm}
and
\begin{thm} \label{thm:BernsteinLpderiv}
Let $s\ge 0$. Then, 
\[
\Vert |\nabla|^s \P_N f\Vert_p
\le 
N^s C_p \Vert f\Vert_p.
\]
\end{thm}

Let us furthermore state a multidimensional version of the one dimensional Bernstein inequality. We first recall the one dimensional case:

\begin{thm}[Bernstein]
Let $f_N$ be a trigonomentric polynomial on $\T$, of order $N$. Then we have the following $L^p$ inequality ($1\le p \le \infty$) for its derivative
\[
\Vert f_N' \Vert_{L^p} 
\le 
N \Vert f_N \Vert_{L^p}.
\]
\end{thm}

We will require the following (multi-dimensional) inequality for the $L^p$-norm of the Laplacian.
\begin{thm}
Let $f_N: \T^d \to \C$ be a trigonometric polynomial of degree at most $N$. Then for any $1\le p \le \infty$:
\[
\Vert \Delta f_N \Vert_{L^p}
\le 
N^2d \Vert f_N \Vert_{L^p}.
\]
\end{thm}

\begin{proof}
Since the constant $N^2d$ in this estimate is independent of $p$, it will suffice to consider $p<\infty$. The result for the $L^\infty$-norm then follows by letting $p\to \infty$. From the one-dimensional inequality applied to the trigonometric polynomial 
\[
x_i \mapsto f_N(x_1,\ldots,x_i,\ldots,x_d),
\]
where the other variables $x_j$, $j\ne i$ are frozen, we immediately obtain
\[
\int \left|\frac{\partial^2 f_N}{\partial x_i^2}\right|^p \, dx_i
\le 
N^{2p} \int \left|f_N\right|^p \, dx_i.
\]
Integrating over $x_1, \ldots,x_{i-1},x_{i+1},\ldots, x_d$, it then follows that
\[
\int \left|\frac{\partial^2 f_N}{\partial x_i^2}\right|^p \, dx
\le 
N^{2p} \int \left|f_N\right|^p \, dx,
\]
and therefore
\begin{align*}
\left(\int \left|\Delta f_N\right|^p \, dx\right)^{1/p}
&\le 
\sum_{i=1}^d \left(\int \left|\frac{\partial^2 f_N}{\partial x_i^2}\right|^p \, dx\right)^{1/p}
\\
&\le 
\sum_{i=1}^d N^2 \left(\int \left|f_N\right|^p \, dx\right)^{1/p}
\\
&=
N^2 d \left(\int \left|f_N\right|^p \, dx\right)^{1/p}.
\end{align*}
\end{proof}

We also recall the following characterization of weakly compact subsets of $L^1([0,T]\times \T^2)$, due to Dunford-Pettis theorem (for a proof, see \cite{DunfordSchwartz}).
\begin{thm}[Dunford-Pettis]\label{thm:DunfordPettis}
A subset $K \subset L^1([0,T]\times \T^2)$ is weakly compact, if and only if
\begin{itemize}
\item $K$ is bounded in the $L^1$-norm,
\item for every $\epsilon>0$, there exists a $\delta>0$ such that 
\[
|A|< \delta \implies \int_A f \, dx \, dt < \e, \quad \text{for all }f\in K.
\]
\end{itemize}
\end{thm}

{\color{black}
We shall also need the following ``Aubin-Lions lemma''. For a proof and thorough discussion of compactness in spaces $L^p([0,T];B)$ with $B$ a Banach space, we refer to \cite{Simon1986} and references therein.

\begin{thm}[\cite{Simon1986}, Thm. 5] 
\label{thm:AubinLions}
Fix $T>0$. Let $X\subset B\subset Y$ be Banach spaces, with compact embedding $X\to B$. If $1\le p \le \infty$ and 
\begin{itemize}
\item $F \subset L^p([0,T];X)$ is bounded,
\item $\Vert f(\cdot+h) - f(\cdot) \Vert_{L^p([0,T];Y)} \to 0$ as $h\to 0$, uniformly for $f\in F$.
\end{itemize}
Then $F$ is relatively compact in $L^p([0,T];B)$ (and in $C([0,T];B)$ if $p=\infty$).
\end{thm}
}

\bibliographystyle{spmpsci}
\bibliography{SpectralViscosity}   

\end{document}